\documentclass[12pt]{amsart}
\usepackage{geometry} % see geometry.pdf on how to lay out the page. There's lots.
\usepackage{xcolor}
\usepackage{ mathrsfs }

\usepackage{graphicx}
 
\usepackage[colorlinks,linkcolor=cyan,citecolor=teal]{hyperref}

\usepackage[title]{appendix}

\makeatletter
\newcounter{savesection}
\newcounter{apdxsection}
\renewcommand\appendix{\par
  \setcounter{savesection}{\value{section}}%
  \setcounter{section}{\value{apdxsection}}%
  \setcounter{subsection}{0}%
  \gdef\thesection{\@Alph\c@section}}
\newcommand\unappendix{\par
  \setcounter{apdxsection}{\value{section}}%
  \setcounter{section}{\value{savesection}}%
  \setcounter{subsection}{0}%
  \gdef\thesection{\@arabic\c@section}}
\makeatother
\usepackage{imakeidx}
\usepackage{enumitem} 
\def\mfa{\mathfrak{a}}
\def\mfb{\mathfrak{b}}
\def\mfl{\mathfrak{l}}

\def\mfu{\mathfrak{u}}

\def\mfs{\mathfrak{s}}
\def\mfe{\mathfrak{e}}
\def\mfn{\mathfrak{n}}
\def\mfw{\mathfrak{w}}

\def\he{\mathfrak{he}}
\def\mfm{\mathfrak{m}}

\def\scrC{\mathscr{C}}

\def\Ld{\mathrm{L}}
\def\Pd{\mathrm{P}}

\def\mfc{\mathfrak{c}}

\def\mfL{\mathfrak{L}}

\def\wto{\rightharpoonup}

\def\a{\alpha}
\def\b{\beta}
\def\g{\gamma}
\def\la{\lambda}
\def\om{\omega}
\def\th{\theta}

\def\G{\Gamma}
\def\La{\Lambda}

\def\ts{\otimes}
\def\bts{\bigotimes}
\def\tm{\times}

\def\N{\mathbb{N}}
\def\Z{\mathbb{Z}}
\def\R{\mathbb{R}}
\def\C{\mathbb{C}}
\def\Yb{\mathbb{Yb}}

\def\ZD{\mathbb{Z}^{\downarrow N}}

\def\<{\langle}
\def\>{\rangle}

\newtheorem{lem}{Lemma}[section]
\newtheorem{prop}{Proposition}[section]
\newtheorem{defn}{Definition}[section]
\newtheorem{coro}{Corollary}[section]
\newtheorem{thm}{Theorem}[section]
\newtheorem{ex}{Example}[section]

\newcommand{\End}{\mathrm{End}}
\newcommand{\Hom}{\mathrm{Hom}}
\newcommand{\Tr}{\mathrm{Tr}}
\newcommand{\id}{\mathrm{id}}
\newcommand{\Id}{\mathrm{Id}}
\newcommand{\EE}{\mathbb{E}}
\newcommand{\PP}{\mathbb{P}}
\newcommand\Gbb{\mathbb{G}}
\newcommand\YM{\mathrm{YM}}
\newcommand\Wg{\mathrm{Wg}}

\newtheorem{rmk}{Remark}[section]
\title[Large $N$ limit of Wilson loops on orientable closed surfaces]{ Large $N$ limit of Wilson loops on orientable closed surfaces in the light of  Koike-Schur-Weyl duality and Spin Networks}

\author{Antoine Dahlqvist}
\address{University of Sussex, School of Mathematical and Physical Sciences, Pevensey 3 Building, Brighton, UK}
\email{antoine.dahlqvist@sussex.ac.uk}
\begin{document}

\begin{abstract}
 {We prove  the  convergence in probability of Wilson loops under  the Yang-Mills measure on any closed, orientable surface for $U(N)$ or $SU(N)$ structure group as $N\to\infty,$ confirming a conjecture of \cite{DLII}.   Our 
 approach revisits  and refines  arguments   of \cite{MageeI,MageeII}  for average Wilson loops under the Atiyah-Bott-Goldman measure and of \cite{DLII} for the Yang-Mills measure.   It consists in first expanding  Wilson loops moments as sums over highest weights, then in representing unitary integrals  as surface sums using Koike-Schur-Weyl duality for traceless tensors, estimating their Euler characteristic with a fine interplay with Dehn's algorithm for surface groups. The first part applies a new formula due to  T. L\'evy, for which we give an alternative proof. The second  starts by finding Schur-Weyl duals  for projection on traceless tensors suitable for large $N$ limits.  We also apply the arguments of the second part   to prove formulas discovered in theoretical physics  by Gross and Taylor \cite{GTI,GTII} and  later Kimura and Ramgoolam \cite{RamKim}  for dimensions of unitary irreps and non-chiral expansions of the Yang-Mills partition function, as well as a relation between rational Schur functions and Wick ordering of power sums functions. 
    }\end{abstract}
\maketitle
\tableofcontents
 
\section{Introduction}

The two-dimensional Yang-Mills measure is a probability measure associated to any pair formed by a compact Lie group $G$ and a surface $\Sigma $ endowed with a volume form $vol$.  It can be defined \cite{LevMarkovHol,GrossKingSengupta,DriversLassos} as a $G$-valued stochastic process $(h_\ell)_{\ell\in\mathrm{L}(\Sigma)}$ indexed by loops of the surface, whose marginals are closely related to the Brownian motion on $G.$  Its heuristic motivation stems  from the physics and geometry of the standard model for interactions of elementary particles, where $\Sigma$ plays the role of space-time and $G$ accounts for the symmetry of the type of interaction.   The Yang-Mills measure is from this heuristic point of view an instance of a Euclidean Quantum Field Theory  where  fields are connexions on $G$-principal bundles over $\Sigma$.

In \cite{Singer},  informed  by the physics literature \cite{KostovKazakov,DaulKazakov,Rusakov}  following the landmark work \cite{HOOFT} of 't Hooft,   I.M.  Singer proposed to study the Yang-Mills measure on a fixed surface,  when $G=U(N)$ or $SU(N)$ and $N\to\infty.$ This limit, a.k.a  large-N limit, aims to simplify the model while retaining its essential features. The main observables are  Wilson loops $W_\ell=\frac1 N\Tr(h_\ell)$,  where $\Tr$ stands for the standard trace on $N\times N$ matrices.  Formulas for their possible limit $\tau_{\Sigma}(\ell)=\lim_N W_\ell$ under the Yang-Mills measure,  appeared  in the physics literature when $\Sigma$ is the plane, an open disc, or the sphere and was doubt master field. As the Yang-Mills measure, the master field is invariant by area preserving diffeomorphisms and  values on homeomorphic loops   are identical as long as   the areas of the domains they bound are preserved. Whereas  convergence in probability towards  the master field  on the plane and the sphere has  been proved  rigorously in  \cite{LevMF,AS,DN,HallMF}, the one on other surfaces  remained open.   In the joint works \cite{DLI,DLII}, we proved it for the torus and conjectured the following  for any closed surface $\Sigma:$ 
\begin{equation}
\tau_{\Sigma}(\ell)=\left\{\begin{array}{ll}\tau_{\tilde\Sigma}(\tilde \ell) &\text{when}\,\ell \,\text{is contractible},\\
&\\
0&\text{otherwise,}\end{array}\right.
\end{equation}
where $\tilde\ell$ denotes the lift of a loop $\ell\in \Ld(\Sigma)$ for the fundamental cover   $\tilde \Sigma\to\Sigma$  of $\Sigma$ endowed with the volume form identified with the one of $\Sigma$. The main result of the current text is to prove this conjecture for closed orientable surface of genus $\ge 2.$  

One part of our proof  adapts a  reduction argument  \cite{DLII} using the so-called Makeenko-Migdal equations, almost allowing to focus on any given homotopy representative for each loop.  

Another one  builds on recent technical progress \cite{MageeI,MageeII} on the large $N$ limit of $\EE[ W_\ell]$ under the Atiyah-Bott-Goldman measure when $G=SU(N)$. The latter is the weak limit of the Yang-Mills measure    as $vol(\Sigma)\to 0$  and can be identified    \cite{SenguptaSC}    as the normalised  Liouville volume form for the symplectic form on 
the moduli space of flat connexions over $\Sigma$ \cite{GoldmanFunda,AB}.   A technical improvement of the current work consists in showing convergence of $\EE[|W_\ell|^2]$ under the Atiyah-Bott-Goldman measure and to generalise it 
the Yang-Mills measure.   Following\footnote{the first steps of this approach based on weighted surface sums, have also been revisited  in \cite{CaoCovering} for the  closely related  lattice Yang-Mills measure with Wilson action.  }  \cite{MPFree},  the strategy of \cite{MageeII}  we build on here,  is to write explicit formulas for Wilson loops moments as  $N^{\chi}$-weighted sums over surfaces with boundaries, with Euler characteristic $\chi$ which can be  suitably 
upper-bounded.   As in \cite{MageePuder} and \cite{MageeII}, the latter bound   depends crucially on the geometry of the loops, a refinement of Dehn's algorithm \cite{BirmanSeries}  and the realisation of $SU(N)$-irreducible representations as traceless tensors.

Revisiting the physics literature,  we   give two alternative approaches to   \cite{MageeI,MageeII}  to get explicit formulas for moments of Wilson loops: one using the so-called IRF 
formulas due to T. L\'evy, the other using walled Brauer algebras and explicit formulas for projection on traceless tensors.    We  use  IRF formulas to 
write and truncate Wilson loop expectations as sums over irreducible representations. We find and apply  duality relations  for traceless tensors to  evaluate each term of such sums, as 
well as to generalise $\frac 1N$-expansions for the partition function of the Yang-Mills measure \cite{LemoineMaida,NovakYM} revisiting the pioneering work of Gross and Taylor \cite{GTI,GTII} and recovering some formulas of \cite{RamKim}.  The latter expressions start from  combinatorial expressions for the inverse of dimensions of an  $SU(N)$-irreducible representation built from mixed tensors, generalising 
the Weingarten functions. At last we observe a relation, appearing implicitly in \cite{GTII},  between  irreducible characters for $SU(N)$ or $U(N)$ and Wick products   in   $(\Tr(U^n),\overline{\Tr(U^n)})_{n\ge 0}$ under the Haar measure.

\section{Definitions and main results}

\subsection{Heat kernel on (special) unitary groups} In this text, $N\ge 1$  is an integer,  $G_N$ or simply $G$ stands  for the compact Lie group given either  by $U(N)=\{U\in M_N(\C):UU^*=\mathrm{Id}_N\}$ or by $SU(N)=\{U\in U(N): \det(U)=1\}$.  Its Lie algebra $\mathfrak{g}$ is respectively  $\mfu(N)=\{X\in M_N(\C): X+X^*=0\}$ or $\mathfrak{su}(N)=\{X\in \mfu(N):\Tr(X)=0\}.$  It is endowed with its unique Haar probability measure $dg$ and the inner product 

\begin{equation}
\<X,Y\>_N= N\Tr(XY^*)\quad\forall X,Y\in \mfu(N).\label{eq---InnerProdLieA}
\end{equation}
The Laplacian \index{$\Delta_{\mathfrak{g}}$ Laplacian on a compact Lie group $G$ with Lie algebra $\mathfrak{g}$}  associated to this inner product is the operator 
\[\Delta_{\mathfrak{g}}=\sum_{i=1}^{\dim(\mathfrak{g})}\mathcal{L}_{X_i}^2 \]
where $(X_i)_i$ is an arbitrary orthonormal basis of $(\mathfrak{g},\<\cdot,\cdot\>_N)$ and for any $X\in \mathfrak{g},$ $\mathcal{L}_X$ is the Lie derivative operator mapping $f\in C^1(G)$ to
\[\mathcal{L}_X(f) (g)=\left.\frac{d}{dt}\right|_{t=0} f(ge^{tX})\quad \forall  g\in G. \]
The heat-kernel \index{$p_T$, Heat-kernel on a compact group $G$} on $G$ is the fundamental solution $(p_T)_{T>0}$ of the heat equation  
\begin{equation}
\left\{\begin{array}{ll}\partial_T p_T = \frac 1 2\Delta_{\mathfrak{g}}p_T&\text{for}\quad T>0,\\& \\ p_T(g)dg\rightharpoonup \delta_{\Id_N}  &\text{for}\quad T\downarrow 0.\end{array}\right.\label{def---HK}
\end{equation}
\subsection{Two dimensional topological maps}\label{sec---2DMAPS}

For a  finite graph $\mathcal{G}=(V,E),$   let $E_o$ be its set of oriented edges, for $e\in E_o,$ $e^{-1}\in E_o$ is the  edge with opposite orientation, $\underline{e} $ and $ \overline{e}\in V$ are respectively  the source and target vertices  of $e.$  A path of $\mathcal{G}$ is a sequence of edges $(e_1,\ldots,e_n)\in E_o^n$ with $\overline{e}_i=\underline{e}_{i+1}$ for all  $i<n$; we then denote it by $\g= e_1e_2\ldots e_n $  and say it is a based loop when its endpoints   $\underline{\g}=\underline{e}_1$ and $\overline{\g}=\overline{e}_n$ are the same.  An unrooted loop is an equivalence class of loops for the cyclic action   on based loops shifting a loop  $e_1\ldots e_n$ to $e_2\ldots e_ne_1.$ Paths, loops and unrooted loops of $\mathcal{G}$ are denoted $\mathrm{P}(\mathcal{G}),\mathrm{L}(\mathcal{G}), \mathrm{L}_c(\mathcal{G}).$

 A  \emph{map} is a triple   $\Gbb=(V,E,F)$   of three countable sets such that $\mathcal{G}=(V,E)$   is a finite connected graph together with
 
\begin{itemize}
 \item a bijection $\theta :E_0\to E_o$ whose cycles provide a cyclic ordering of $Out_v=\{e\in E_o:\underline{e}=v\}$ for each $v\in V,$ 
\item    a bijection $\partial:F\to \mathcal{C}_\mathcal{S}$ where  $\mathcal{C}_\mathcal{S}$ are the orbits of the map $\mathcal{S}:E_0\to E_0, e\mapsto \theta(e^{-1}).$
 \end{itemize}
%\item  corners $F$ is a finite set  such that $\mathcal{G}$ can be embedded in a closed, orientable surface $\Sigma$ in which  the complement of the embedding of $\mathcal{G}$ have   components indexed by $F$.
Elements of $F$ are called faces of the map. Elements of $\mathcal{C}_\mathcal{S}$ are identified with unrooted loops of $\mathcal{G}$ and we say $\partial f$ is  the boundary of the face $f\in F.$ The left coboundary    $L:E_0\to F$ maps $e$ to $f\in F$ with $e\in \partial f.$   Setting $R(e)=L(e^{-1}),$  as $L\circ \mathcal{S}=L,$
\begin{equation}
R(e)=L(\theta(e))\quad \forall e\in E_0\label{eq---OrderVertex/FaceCompatMap}.
\end{equation}

Identifying $\mathcal{G}$ with a  $1$-CW complex, a closed orientable surface $\Sigma$  can be built  gluing for each  $f\in F$   a disc  along the unrooted loop $\partial f.$   The orientation of $\Sigma$ can then be chosen  so that $\theta $ is given by the clockwise order around vertices. We then say the map $\mathbb{G}$ is \emph{embedded} in $\Sigma.$ With this convention for each edge $e\in E_0,$ the disc attached for the face $L(e)$
appears   on the 
left of the image of $e$ in $\Sigma.$  Call \emph{genus} $g$ of the map $\Gbb$ the genus of $\Sigma$ so that $\#V-\#E+\#F=2-2g.$  Say  $\Gbb$  is finite when $V,E,F$ are \emph{finite} and infinite otherwise.  Except when explicitly mentioned, all maps will be assumed finite. 

A \emph{covering} of $\Gbb$  is a  finite or infinite map  $\tilde \Gbb=(\tilde V,\tilde E,\tilde F)$ with cyclic order $\tilde\th$ together with an application $p:\tilde  E^o\to E^o$ with  \[p(e^{-1})=p(e)^{-1}\quad\text{and}\quad \th\circ p=p\circ \tilde \th\quad\forall e \in \tilde E^o\]
such that  $p_{|Out_v}$  is injective for all $v\in \tilde V.$  A path $\tilde \g=\tilde e_1\ldots \tilde e_n\in\Pd(\tilde \Gbb)$   is a lift of $\g\in \mathrm{P}(\Gbb)$ if  $\g=e_1\ldots e_n\in\Pd(\Gbb)$ with $p(\tilde e_i)=e_i.$ When $a\in \R_+^F$ its lift is the area vector  $\tilde a\in \R_+^{\tilde F}$ with $\tilde a(L( e))=a(L(p(e)))$ for all $e\in \tilde E^o.$

  A \emph{planar map} is a finite map of genus $0$ with one marked face $f_*$, we call it the unbounded face,   set $F^b=F\setminus\{f_*\}$ and for $a\in \R_+^F$ denote $a_{b}\in \R_+^{F_b}$ the restriction to $F_b.$  When $\Gbb$ is embedded in a surface $\Sigma$ and $U $ is a relatively compact, open topological disc of $\Sigma,$  $\Gbb_U=(V_U,E_U,F_U)$ is the planar map obtained by keeping only cells of $\Gbb$ whose embedding is included  in $U$ and identifying all faces with embedding intersecting $\Sigma\setminus U.$ Then, whenever $a\in (\R_+)^F,$ $a_U\in (\R_+^*)^{F^b_U}$  denotes  the \emph{restriction} of $a$ to $F_U^b.$  
\subsection{Lattice gauge configurations and character variety}

When $\mathcal{G}=(V,E)$ is a finite graph,   denote $G^{E_o}_0$ or $\mathcal{M}(\mathcal{G},G)$  the set of  functions $h\in G^{E_o}$ satisfying $h(e^{-1})=h(e)^{-1}.$ 
As a compact group, $G^{E_o}_0$  has identity element  $\id_{\mathcal{G},G}=(\Id_N)_{e\in E_o}\in G^{E_o}_0$ and we denote its Haar measure   $U_{\mathcal{G}}(dh)$ or simply $dh$, writing $\EE_{\mathcal{G},G}$ and $\PP_{\mathcal{\mathcal{G},G}}$ for expectations and probabilities on the probability space $(\mathcal{M}(\mathcal{G},G),\mathcal{B}_{\mathcal{G},G},U_\mathcal{G})$ where $\mathcal{B}_{\mathcal{G},G}$ is the Borel sigma field on $\mathcal{M}(\mathcal{G},G).$  For any path 
$\g=e_1\ldots e_n\in \mathrm{P}(\Sigma),$ $h\in\mathcal{M}(\mathcal{G},G),$ we set 
\begin{equation}
h_\g=h_{e_1}\ldots h_{e_n}.\label{eq---Transport}
\end{equation}
Whenever $\chi:G\to \C$ is a class function or $\mathscr{C}$ is a conjugacy class  $G$ and $\mfl\in \mathrm{C}_c(\mathcal{G}),$ $\chi(h_{\mfl})$ or $h_{\mfl}\in\mathscr{L}$ stand for $\chi(h_{\mfl'}) $ or $h_{\mfl'}\in\mathscr{C}$ where $\mfl'$ is any based loop in the equivalence class  $\mfl.$

When $\Gbb=(V,E,F)$ is a map, we consider the compact subset  \[\mathcal{M}_0(\mathcal{G},G)=\{h\in \mathcal{M}(\mathcal{G},G): h_{\partial f}=\Id_N \,\forall f\in F\}\]
$\mathcal{M}(\mathcal{G},G).$ 
The group $G^V$  acts on both $\mathcal{M}(\mathcal{G},G) $ and $\mathcal{M}_0(\mathcal{G},G)$ with 
\begin{equation}
j.h(e)=j(\underline{e})^{-1}h(e)j(\overline{e}) \quad\forall e\in E_o,j\in G^V.\label{eq---GaugeAction}
\end{equation}
The character variety is the quotient $\mathcal{X}_g(G) =\mathcal{M}_0(\mathcal{G},G)/G^V. $ Restricting \eqref{eq---Transport} to loops based at a same vertex yields an identification of $\mathcal{X}_g(G)$ with $\Hom(\pi_1(\Sigma),G)/G$ and therefore, more concretely, with   
 \[H_g(G)/G\]
 where $H_g(G)=\{h\in G^{2g}: [h_1,h_2]\ldots[h_{2g-1},h_{2g}]=\mathrm{Id}_N\},$  the action of $G$ is by diagonal conjugation and for $a,b\in G,$ $[a,b]=aba^{-1}b^{-1}.$  This quotient space is not a manifold but a manifold with corners which can be removed as follows.  For   a subgroup $H< G,$ denote  $Z(H)=\{g\in G: hg=gh \quad \forall h\in H \} $   its centraliser in $G$. When $G=SU(N)$ and $g\ge 2,$ 
 \begin{equation}
 H^{reg}_g(G)=\{h\in G^{2g}:[h_1,h_2]\ldots[h_{2g-1},h_{2g}]=\mathrm{Id}_N,  Z(\<h_i, 1\le i\le 2g\>)= Z(G) \}
 \end{equation}is  a dense open subset  in $H_g(G)$ and  a manifold of dimension $(2g-1)\dim(G)$ on which $G$ acts freely properly  so that  the quotient 
 \begin{equation}
\mathcal{X}^{reg}_g(G)= H^{reg}_g(G)/G
 \end{equation}
is an open manifold of dimension $(2g-2)\dim(G).$ 

\subsection{Atiyah-Bott-Goldman  and   Yang-Mills measure}
When  $\Gbb=(V,E,F)$ is a map and  $a\in {\R^*_+}^F$, the discrete Yang-Mills measure is the  measure on $\mathcal{M}(\mathcal{G},G)$  defined by 
\begin{equation}
\YM_{\Gbb,a}(dh)= \prod_{f\in F} p_{a(f)}(h_{\partial f}) dh.
\end{equation} 
For any $\Gbb,$ its total volume 
\begin{equation}
 Z_{G,g,|a|}=\YM_{\Gbb,a}(1)<+\infty
\end{equation}
is finite and depends only (see \cite{LevMarkovHol}) on  the group $G,$ the genus $g$ of $\Gbb$ and its total area 
\begin{equation}
|a|=\sum_{f\in F}a(f).
\end{equation}
When $G=SU(N)$ and $g\ge 2,$ we shall recall later this measure extents to $\R_+^F.$  In particular, the following is known as the semi-classical limit of the Yang-Mills measure.
   
\begin{thm} When $G=SU(N),g\ge 2,$ as $a\to0,$ 
\begin{equation}
 \YM_{\Gbb,a}\wto    \mu_{\Gbb,ABG} 
\end{equation}
for some measure $\mu_{ABG}$ with $\mathrm{supp}(\mu_{ABG})=\mathcal{M}_0(\mathcal{G},G).$  
\end{thm}
 
We  call  $\mu_{\Gbb,ABG}$  the Atiyah-Bott-Goldman\footnote{The  above holds for any compact, semi-simple Lie group $G$. The manifold $\chi_g(G)$ can be endowed with  a symplectic form $\omega$   with symplectic volume form  $vol_\omega=\frac{1}{(g-1)!}\wedge^{g-1} \omega$  satisfying   \begin{equation}
\pi_*\mu_{ABG}=c vol_\omega\label{eq---LiouvilleMeasure}
\end{equation} where $\pi: \mathcal{M}_0(\mathcal{G},G)\to \mathcal{X}_g(G)$ is the quotient map and
\begin{equation}
c= \#Z(G)^{-1} (2\pi)^{(2g-2)\dim(G)} vol(G)^{2-2g}. 
\end{equation}
Here  $vol(G)$ is the Riemannian volume of $G$ for $\<\cdot,\cdot\>_N$ and $\omega$ is defined \cite[p.209]{GoldmanFunda} with respect to $\<\cdot,\cdot\>_N$ as bilinear form $\mathcal{B}$ on $\mathfrak{g}.$ In particular 
\begin{equation}
vol_\omega(\mathcal{X}_g(G))= \#Z(G)  (2\pi)^{(2-2g)\dim(G)} vol(G)^{2g-2} \mu_{ABG}(1). 
\end{equation}} measure.   Its volume 
\begin{equation}
Z_{N,g}=\mu_{\Gbb,ABG}(1)<\infty
\end{equation}
 is finite and depends only of $g\ge 2$ and $N\ge 2.$ In what follows $\EE_{\Gbb,a},\EE_{\Gbb}$ or  $\PP_{\Gbb,a},\PP_{\Gbb}$ stand for the expectation or probability for the probability measures $Z_{G,g,|a|}^{-1}\YM_{\Gbb,a}$ or $Z_{N,g}^{-1}\mu_{\Gbb,ABG}.$
\subsection{Continuum Yang-Mills measure} The discrete Yang-Mills measures for different maps are compatible and restriction of a single measure in the following sense.   Assume that  $\Sigma$ is a closed orientable surface endowed with a Riemannian metric.  We then say that a map $\Gbb$ is embedded in $\Sigma$ if $\Gbb$ is embedded as two dimensional maps, in the sense of section \ref{sec---2DMAPS}, and such that each edge viewed as path of $\Sigma$ has finite length.  Then for each face $f\in F,$  $vol_\Gbb(f)$ denote the Riemannian volume of the image of $f$ in $\Sigma.$ 

Consider   the set $\mathrm{P}(\Sigma)$ of finite length path parametrized by $[0,1]$ at constant speed. It is endowed with a concatenation operation $\g_1\g_2$ defined whenever $\g_1,\g_2\in\mathrm{P}(\Sigma)$ and the endpoint $\overline{\g}_1$ of $\g_1$  equals the starting point $\underline{\g}_2$ of $\g_2.$  Denote $\mathrm{L}(\Sigma)$ the subset of finite length  loops, that is paths  $\ell\in \mathrm{P}(\Sigma)$ with $\underline{\ell}=\overline{\ell}.$ The  length metric on   $\mathrm{P}(\Sigma)$ is defined setting
\begin{equation}
d(\g_1,\g_2)=|\mathscr{L}_{\g_1}-\mathscr{L}_{\g_2}|+\inf_{h}\sup_{0\le t\le 1} d_\Sigma(\g_1(h(t)),\tilde \g_2(t))\quad \forall \g_1,\g_2\in\mathrm{P}(\Sigma),
\end{equation}
where $\mathscr{L}_{\g}$ denote the length of a path $\g\in \mathrm{P}(\Sigma)$ and the infimum is taken over all  homeomorphism $h:[0,1]\to[0,1]$.

We then endow
\begin{equation}
\mathcal{M}(\mathrm{P}(\Sigma),G)=\{h\in G^{\mathrm{P}(\Sigma)}: h_{\g_1\g_2}=h_{\g_1}h_{\g_2}\quad \forall \g_1,\g_2\,\text{with}\,\overline{\g}_1=\underline{\g}_2 \}.
\end{equation}  
with cylinder sigma-field $\mathcal{C}_{\Sigma}$, that is, the smallest sigma field  such that $h\in \mathcal{M}(\mathrm{P}(\Sigma),G)\mapsto h_\g$ is measurable for all $\g\in\mathrm{P}(\Sigma).$ Whenever $\Gbb$ is embedded in $\Sigma,$ there is a canonical map $R^{\Sigma}_{\Gbb}: \mathcal{M}(\mathrm{P}(\Sigma),G) \to \mathcal{M}(\mathcal{G},G).$

 For any $a\in (\R_{+}^*)^F,$ $\YM_{\Gbb,a}(1)<+\infty$ and 

\begin{thm}[\cite{LevMarkovHol}]  Assume $\Sigma$ is a closed orientable surface of genus   endowed with a Riemannian metric. There is a unique measure $\YM_{\Sigma}$  on $(\mathcal{M}(\mathrm{P}(\Sigma),G), \mathcal{C}_\Sigma)$ such that for any map $\Gbb$ embedded in $\Sigma$,
\begin{equation}
{\mathcal{R}^{\Sigma}_{\Gbb}}_*(\YM_{\Sigma})=\YM_{\Gbb,vol_\Gbb}
\end{equation} 
such that $h_{\g_n}$ converges in distribution to $h_{\g}$ under $\YM_{\Sigma}$ whenever  endpoints of $\g_n$ and $\g$ match for all $n\ge 1$ and  $\lim_{n\to\infty}d(\g_n,\g)=0.$
\end{thm}
\subsection{Large N limit of Wilson loops and Master field} Whenever $\g$ is a path in  $\mathrm{P}(\Gbb)$ or $\mathrm{P}(\Sigma),$  set 
\begin{equation}
W^N_{\g}=\frac{1}{N}\Tr(h_{\g}).
\end{equation}
where $\Tr$ stands for the standard trace of $h\in U(N)$ acting on $\C^N.$

\begin{thm}[\cite{DLI}] Assume   $\Gbb$ is a map of genus $g\ge 1$ embedded in a surface $\Sigma$ and $U$ is an open topological disc of $\Sigma.$ Then for any loop $\mfl\in \Ld(\Gbb)$  whose embedding in $\Sigma$ is included in $U,$  
\begin{equation}
\EE_{\Gbb,a}[|W_{\mfl}^N-\tau_{\Gbb_{U},a_U}(\mfl)|]\underset{N\to\infty}{\to}0.
\end{equation}
where $\tau_{\Gbb_U, a_U}(\mfl)\in[-1,1]$ is a constant depending only on  the planar map $\Gbb_U=(V_U,E_U,F_U)$ and $a_U\in (\R_+^*)^{F_U^b}.$

\end{thm}

 A continuous version of the above statement also holds. 

\begin{thm}[\cite{DLI}] Assume $\Sigma$ is closed orientable surface $\Sigma$ of genus $g\ge 1$endowed with a Riemannian metric and $U$ an open topological disc of $\Sigma$. Then for any loop $\Ld(\Sigma)$ with range included in $U,$ 
\begin{equation}
\EE_{\Sigma}[|W_{\ell}^N-\tau_U(\ell)|]\underset{N\to\infty}{\to}0.
\end{equation}
where $\tau_{U}(\ell)\in[-1,1]$ is a constant depending only on  $U$ and $\ell$ up to area-preserving diffeomorphisms. \end{thm}
 When $\Gbb=(V,E,F)$ is a planar map,  $a\in (\R_+^*)^{F^b}$ and  $\mfl\in \Ld(\Gbb),$ or when $\ell$ is a finite length loop in a disc $U $ endowed with Riemannian metric,  $\tau_{\Gbb}(\mfl)$ or $\tau_U(\ell)$  is called the \emph{planar master field} and is known to be the limit of   Wilson loops under the Yang-Mills measure on the plane \cite{LevMF,AS}. By convention, we set
 \begin{equation}
 \tau_{\Gbb}(\g)=0\quad\text{or}\quad\tau_{U}(\g)=0\quad\text{when}\,\g\in \Pd(\Gbb)\setminus\Ld(\Gbb)\,\text{or}\, \Pd(U)\setminus\Ld(U).
 \end{equation}
 \begin{rmk} \label{Rmk---Paths} This convention is coherent with the following remark. Using gauge invariance,   for any path $\g$ which is not a loop, under the Yang-Mills measure on any surface, $h_{\g}$ has same distribution as a  Haar distributed random variable $U$, so that by orthogonality of characters    \begin{equation}
 \EE[|W_\g^N|^2]= N^{-2}\EE[|\Tr(U)|^2]=\frac 1 {N^2}.
 \end{equation}
 \end{rmk}
 \begin{rmk} The above theorems do not hold when $\Sigma$ is the sphere where  Wilson  loops are known to converge to a different limit \cite{DN} that also depends on the total volume $|a|.$ 
 \end{rmk}

{  \begin{ex}  When $\ell$ is a simple loop in $\R^2$ bounding an disc of area $t>0,$   $\tau_{\R^2(\ell)}=e^{-\frac t 2}.$
 \end{ex}}
 
 When $g=1,$  \cite{DLII}   also proved the restriction on loops to  be contained in a disc can be omitted, with the limit Wilson loop being either vanishing when the loop is not contractible, or  the planar master field at the lift of the loop otherwise.  It was further  conjectured therein  that the same holds for any $g\ge2.$ The following theorems confirm this conjecture and are the main result and motivation for the current work.

 \begin{thm} \label{THM---MFConv}Assume $\Gbb$ is a map of genus $g\ge 2 $ and $\mfl\in \Ld(\Gbb).$ Consider a covering $\tilde \Gbb$ embedded in a surface $\tilde \Sigma$  and   $\tilde\mfl$  a lift of $\mfl$  with embedding included  an open disc $U$ of $\tilde \Sigma$.   Then for any $a\in (\R_+^*)^F,$ 
 \begin{equation}
\lim_{N\to\infty}\EE_{\Gbb,a}[|W^N_\mfl-\tau_{\tilde \Gbb_U,\tilde a_U}(\tilde \mfl)|]=0. \label{def---MasterField}
\end{equation} 
 \end{thm}
 \begin{rmk}  For any  loop $\mfl,$ there is always  a  lift as above choosing for instance $\tilde \Gbb$ with embedding isomorphic to the fundamental cover  of $\Sigma,$ and the value $\tau_{\tilde \Gbb_U,\tilde a_U}(\tilde \mfl)$ only depends on $\Gbb, \mfl$ and  $a$.  We shall denote the latter by  $\Phi_\mfl(a).$ 
 \end{rmk}
It has the following continuous version.

 \begin{thm}  \label{THM---ContMFConv}  Assume $\Sigma$ is a closed surface of genus $g\ge 2 $ endowed with a Riemannian metric.  Denote by $\tilde \Sigma$ its fundamental cover endowed with the lifted metric. Assume $\ell\in \Ld(\Sigma),$ then for any lift $\ell\in \Ld(\tilde\Sigma)$ and open disc $U$ of $\tilde \Sigma$ including $\tilde \ell,$     
 \begin{equation}
\lim_{N\to\infty}\EE_{\Sigma}[|W^N_\ell-\tau_{U}(\tilde\ell)|]=0.
\end{equation} 
 \end{thm}
 {  \begin{ex}  When $\ell$ is a  contractible loop in $\Sigma$ such that its lift $\tilde\ell$ in $\tilde\Sigma$ is simple and bounds a topological disc of area $t$, then     $\tau_{\Sigma}(\ell)=e^{-\frac t 2}.$
 \end{ex}}
 
Our argument builds on two  independent set of results \cite{DLI,DLII} and \cite{MageeI,MageeII} as well as a new formula for Wilson loop expectations due to T. L\'evy \cite{IRFThierry}.

In \cite{DLII} it was proved    how  to reduce  the above convergence,  for any closed surface, from all loops  to loops  of the form  $a$ or $ab$ where $a$ is included in a disc and of a specific form, and $b$ is a geodesic loop.   When $g=1, $ since any geodesic loop is a power of a simple loop,  combined with an  argument   \cite[Cor 4.3. and 4.6.]{DLI}  alike\footnote{using  additionally Pieri's formula} remark \ref{Rmk---Paths}  allowed to conclude.   When $g\ge 2,$ geodesic loops are generically not simple, and the analysis of the representation theory  involved in computating expected Wilson loops is significantly more involved which prevented  further progress so far.

Independently from \cite{DLI,DLII}, building on former works \cite{MPFree} and  in parallel to  \cite{MageePuder} where the problem was\footnote{The analog of the Atiyah-Bott-Goldman measure is then the uniform measure on $\Hom(\pi_1(\Sigma),S_N)$.  In this setting, the authors of \cite{MageePuder} are able to prove  a more general version of  \eqref{eq---Vanishing AverageABG}  but also to compute explicitly $\lim_{N\to\infty} N\EE[W_{\mfl}^N]$ for any non contractitble loop $\mfl$. } considered for $G=S_N$, much technical progress was made   for the Atiyah-Bott-Goldman measure leading to the following result of the two-paper series \cite{MageeI,MageeII}  we slightly reformulate below.

\begin{thm}[\cite{MageeII}]  For $G=SU(N),$ for any map $\Gbb$ of genus $g\ge 2$ and any loop $\mfl\in\Ld(\Gbb),$
\begin{equation}
\lim_{N\to\infty}\EE_{ABG}[W_{\mfl}^N]=0.\label{eq---Vanishing AverageABG}
\end{equation}
\end{thm}

Simply put, the strategy of \cite{MageeI,MageeII} and \cite{MageePuder} can be  broken into two parts. 

A first one is to expand $\mu_{ABG}$ into sums over   characters  \cite{MageeI} and to show $\EE_{ABG}[W_{\mfl}^N]$ can be approximated by finite sums to arbitrary precision in powers of $\frac{1}{N}.$ This part built on estimates of \cite{LarsenWZeta} that where used to show $\lim_N \mu_{ABG}(1)=1$ and the  technical innovation of \cite{MageeI,MageePuder} is to estimate each term of the sum using  classical branching rules.\footnote{Using a map with a single polygonal face, the idea of the  branching rule  approach of \cite{MageeI,MageePuder} can be summarised as follows.  First use the invariance by 
permutation to write  the  average standard trace of a word $w$ of length $k$  in Haar random variables  as a  sum of monomials  involving only matrix coefficients of the last  $k$ rows and columns. When $\la\in \hat{G}_N$ and $w_g$ is a surface word, $|\int_{G_N^{2g}} \Tr(w(h))\chi_\la(w_g(h))dh | $ is therefore bounded by $N^k\int_{G_{N}^{2g}}| \int_{G_{N-k}^{2g}}  \chi_\la(w_g(hh'))dh'|dh.$ A delicate part consists then in  expanding the character into a basis of the irrep $\la$ compatible with the branching, to explicitly compute the integral over $G_{N-k}$ using  Peter-Weyl isometry,  and at last to   bound it adapting estimate of \cite{LarsenWZeta} to bound  ratios of dimension of  representation of $G_N$ and $G_{N-k}$ related by branching, leading ultimately to a bound of the form $C N^{k'}\chi_\la(1)^{1-2g}$ for some constant $C$ and $k'.$   }

A second one  shows  that each term of this sum  asymptotically vanishes as $N\to\infty$. Each term is then the integral of a character evaluated in a surface word in independent Haar elements.  As is expected for other 
random unitary matrix ensembles, following  \cite{MPFree} together with some representation theory arguments,   these integrals are expressed in \cite{MageeII}  as   sums indexed by surfaces with boundaries, weighted by a factor $N^{\chi}$ where $\chi$ is the Euler characteristic of the index.  
A first key step and technical innovation of \cite{MageeII} is to study irreps in the large $N$ regime when realised as  traceless tensors. 
%and to perform integration over $G_N$ of representation coefficients as some projections within traceless tensors. 
This  significantly  simplifies more naive approaches using Frobenius formula and standard Weingarten calculus \cite{Collins,CollinsSn}.      A second one is then to upper bound the Euler characteristic of these surfaces. This second step crucially relies on the minimal length assumption on the loop.

\subsection*{Outline of the paper} Our main technical contribution is to argue the strategy of \cite{MageeI,MageeII} can be improved to show $L^2$ convergence instead of convergence in expectation and moreover generalised from the Atiyah-Bott-Goldman measure to the Yang-Mills measure.   As in \cite{MageeI,MageeII}, it is divided in two parts respectively sections \ref{sec---Truncation} and \ref{sec---Integration  using Koike-Schur-Weyl duality}.

 The argument of section \ref{sec---Truncation}  addresses the first step of this strategy. It is independent  from \cite{MageeI}  using  as a starting point Pieri's rule instead of branching rules for unitary groups, and rely then on a formula of T. L\'evy \cite{IRFThierry}.  We hope this argument is more easily adaptable to different group series.  
 
  Section \ref{sec---Integration  using Koike-Schur-Weyl duality} revisits the second step with a different approach to write characters in terms of traceless tensors. It is devoted to integration of characters as surface sums formulas using walled Brauer algebras thanks to the Koike-Schur-Weyl duality.  This allows us to give complete arguments for several formulas discovered   in the physics literature by  Gross-Taylor \cite{GTI,GTII} and  Ramgoolam and Kimura \cite{RamKim} for combinatorial interpretations for  $\frac 1N$- expansions  
for  irreps dimensions and  partition functions $Z_{G_N,g,T}$ for $g\ge 2$.  Similar expansions have been recently found independently in \cite{LemoineMaida} for $g=1$ and in \cite{NovakYM} for the so-called chiral regime, involving partial sums over $\hat{G}_N.$   
 From subsection \ref{sec---Integration walled Brauer algebra} onwards, the rest of the section focuses on application to Wilson loop integrations against a character.   We then carefully revisit and slightly improve arguments of  \cite{MageeII} within our framework to  bound the Euler characteristic of surfaces appearing when integrating the square modulus of a Wilson loop for a shortest length loop, leading ultimately to Theorem \ref{THM---CValmostshortestRep}.

Section \ref{sec---MasterFieldAllLoops}   adapts the argument of \cite{DLII} to prove Theorem \ref{THM---CValmostshortestRep} implies $L^2$ convergence for all loops and conclude the proof of Theorems \ref{THM---MFConv} and \ref{THM---ContMFConv}.

In addition, we add for completeness four appendices. Appendix \ref{sec---Witten-Zeta function truncation} recalls a short proof of a version of arguments of \cite{LarsenWZeta} we need 
here to truncate sums over irreps. Appendix \ref{sec---IRF} gives a proof the  IRF formula due to T. L\'evy \cite{IRFThierry} that allows to write Wilson loop expectations as sum 
over irreps.    Appendix \ref{Section---RepBrauer} gives an expression of the  projection onto traceless tensors in terms of walled Brauer algebra and a generalisation of Jucys-Murphy elements.  Appendix \ref{sec---GrossTaylor} gives a proof and new interpretation of a formula of \cite{GTII} for the integration of Newton polynomials, relating Wick product of random variables to projection onto traceless tensors. We prove there another combinatorial expression for the $\frac 1N$ expansion of irreps dimension due to \cite{GTII}.

\section{Decomposition of the discrete Yang-Mills measure into characters}

\label{sec---Truncation}

\subsection{Witten-zeta function truncation}

For any $\a\in\Z^{\downarrow N},$  set 
\begin{equation}
d_\a=\frac{V(\a+\rho)}{V(\rho)}\label{eq---dim}
\end{equation}
where $V(x)=\prod_{1\le i<j\le N}(x_i-x_j)$ and 
\[2\rho=(N-1,N-3,\ldots,-N+3,-N+1).\] 
Let us consider the action of $\Z$ on $\Z^{\downarrow N}$ by translation by the constant vector $1_N=(1,\ldots,1).$ The function \eqref{eq---dim} is   invariant and when $\a\in \Z^{\downarrow N}/\Z$   we also write $d_\a$ for $d_{\a_*}$ where $\a_*$ is any vector in the class of $\a.$ 
 
The \emph{Witten-zeta function}  \footnote{it was introduced in \cite{WittenOnQGT} appearing while computing the volume of the moduli space flat irreducible connections  and doubt Witten-zeta function by  \cite{ZagierWZeta}.}  is defined for  $N\ge1,$ $s>0$ by  
\begin{equation}
\zeta_{SU(N)}(s)=\sum_{\a\in \mathbb{Z}^{\downarrow N}/\Z} \frac{1}{d_\a^s}\in \R\cup\{\infty\}.\label{eq---WittenZeta}
\end{equation}
We shall encounter the following deformation.  For $q<1,$ 
\begin{equation}
\zeta_{SU(N)}(s,q)=\sum_{\a\in \mathbb{Z}^{\downarrow N}/\Z} \frac{1}{d_\a^s}q^{\mfc^*_\a} \quad\text{and}\quad\zeta_{U(N)}(s,q)=\sum_{\a\in \mathbb{Z}^{\downarrow N}} \frac{1}{d_\a^s}q^{\mfc_\a} \label{eq---WZCasimir}
\end{equation}
where 
\begin{equation}
N \mathfrak{c}_\a=\|\a+\rho\|^2-\|\rho\|^2\quad\text{and}\quad N \mfc^*_\a=\|\overline{\a+\rho}\|^2-\|\rho\|^2,  \label{def---Casimir}
\end{equation}
noting that the second function is invariant by translation.
Here for any vector $x\in\R,$
\[\|x\|^2=\sum_ix_i^2\quad\text{and}\quad \overline x= x- \frac1 N \left(\sum_{i}x_i\right)1_N. \]

Let us reparametrise $\ZD.$ An integer partition is a non-decreasing sequence of non-negative integers $\la=(\la_1,\la_2,\ldots)$ with  finitely many non zero terms. We write $\la\vdash n$ when $\sum_i\la_i=n$ and set $\ell(\la)=\#\{i\ge 1:\la_i>0\}.$  
When $\la,\mu$ are two integer partitions with $\ell(\mu)+\ell(\mu)\le N,$  define the vector
\[[\la,\mu]_N=(\la_1,\ldots,\la_{\ell(\la)}, 0,\ldots,0, -\mu_{\ell(\mu)},\ldots,-\mu_1 )\in \ZD\]
where, if there is any, the $N-\ell(\la)-\ell(\mu)$ middle coefficients vanish.   Any vector $\ZD$ is uniquely written in this form where for $\a\in \ZD,$  $\la$ and $-\mu$ are the possibly reordered sequences of 
positive and negative coefficients of $\a$.  This parametrisation of $\ZD$ is adapted to the large $N$  limit of  \eqref{eq---WittenZeta} and \eqref{eq---WZCasimir} in the following 
sense. Let us slightly change this parametrisation of $\ZD$ to make it compatible with the quotient $\ZD/\Z.$ When $\a\in \ZD,$ set 
\begin{equation}
c(\a)=\a_{\lfloor\frac{N-1}{2}\rfloor+1}
\end{equation}
and
\begin{equation}
\la(\a)= (\a_1-c(\a),\ldots,\a_{\lfloor \frac{N-1}{2}\rfloor}-c(\a)) \quad\text{and}\quad \mu(\a)=(c(\a)-\a_N ,\ldots, c(\a)-\a_{\lfloor \frac{N-1}{2}\rfloor+2}).
\end{equation}
Then  $\la(\a),\mu(\a)$ are two integer partitions  whose lengths sum is less than $N$ and 
\begin{equation}
\a= [\la(\a),\mu(\a)]_N+c(\a)1_N.
\end{equation}
Call size\footnote{The above definitions of translation by $\a_{\lfloor \frac{N-1}{2}\rfloor}$ or of size $|\a|$  is not canonical. E.g. a translation by $\a_{k_N}$ with $k_N\to\infty$ and $\underline{\lim}_{N}\frac {k_N}N,\underline{\lim}_{N}\frac {N-k_N}N>0$, or considering $\a_1-\a_N$ in place of $|\a|$ should only require minor changes to the following arguments.} of $\a\in \ZD$ the integer
\begin{equation}
|\a|=|\la(\a)|+|\mu(\a)|=\sum_{i=1}^N(\a_i-\a_N). 
\end{equation}

The functions $\la(\a),\mu(\a),|\a|$ are invariant by translation on $\ZD$ and we denote abusively by the same symbol the associated functions on $\ZD/\Z.$  Define then for  $k\ge 0,$
\begin{equation}
\zeta_{SU(N)}^{(> k)}(s)=\sum_{\a\in \mathbb{Z}^{\downarrow N}/\Z: |\a|>  k} \frac{1}{d_\a^s} \label{eq---WittenZetaT}
\end{equation}
and  for $q<1,$ 
\begin{equation}
\zeta_{SU(N)}^{(> k)}(s,q)=\sum_{\a\in \mathbb{Z}^{\downarrow N}:|\a|> k} \frac{1}{d_\a^s}q^{\mfc^*_\a} \quad\text{and}\quad\zeta_{U(N)}^{(> k)}(s,q)=\sum_{\a\in \mathbb{Z}^{\downarrow N}:|\a|> k} \frac{1}{d_\a^s}q^{\mfc_\a}. \label{eq---WZCasimirT}
\end{equation}
For $c\in \N,$ set furthermore
\begin{equation}
\zeta_{U(N)}^{(>k),(>c)}(s,q)=\sum_{\a\in\ZD: |\a|>k\,\text{and}\,|c(\a)|>c} \frac{1}{d_\a^s}q^{\mfc_\a}. 
\end{equation}

\begin{prop} \label{Prop---ConvWittenZeta}  
For any $d\ge1, s>0$ fixed, there are $k'>0$ and $K>0$ such that for all $k\ge k',$  $N\ge1, q< 1,$
\begin{equation}
 \quad \zeta_{SU(N)}^{(> k)}(s,q)\le \zeta_{SU(N)}^{(> k)}(s) \le  \frac{K}{N^d}. \label{eq---BoundWZT}
\end{equation}
For any $s>0,0<q<1$    and $k\ge 0, c\ge 1$
\begin{equation}
\zeta_{U(N)}^{(> k)}(s,q) \le   \theta(q)  \zeta_{SU(N)}^{(> k)}(s,q)
\end{equation}
and for $k<N,$
\begin{equation}
\zeta_{U(N)}^{(> k),(c+2)}(s,q) \le    \frac{q^{c^2}}{c \log(1/q)}  \zeta_{SU(N)}^{(> k)}(s,q)\label{eq---TruncationCharge}
\end{equation}
where $\theta(q)=\sup_{u\in [0,1]} \sum_{n\in \Z} q^{(n+u)^2}.$
%\theta(\sqrt q) 
\end{prop}

\begin{proof} The second bound of \eqref{eq---BoundWZT} follows from \eqref{Coro----MassLargeSize} in corollary \ref{Coro----MassLargeSize}.         For the second display, note that for any $\a\in \ZD,$ 
\begin{equation}
\mfc_\a=\mfc_\a^*+m(\a+\rho)^2= \mfc_\a^*+ \left(c(\a)+\frac{|\la|-|\mu|}{N} \right)^2,
\end{equation}
where we write  for $x\in \R^N,$ $m(x)=\frac 1 N \sum_{i}x_i.$ We conclude that for $0<q<1,$
\begin{align}
\zeta_{U(N)}^{(> k)}(s,q)= \sum_{\a\in \mathbb{Z}^{\downarrow N}/\Z:|\a|> k} \frac{q^{\mfc_\a^*}}{d_\a^s} \sum_{n\in\Z} q^{ (n+\frac{|\la(\a)|-|\mu(\a)|}{N})^2}\le \theta(q) \zeta_{SU(N)}^{(> k)}(s,q)  \label{eq---BoundSupTheta}
\end{align}
and likewise  comparison with an integral of the truncated second sum yields \eqref{eq---TruncationCharge}.
\end{proof}

\begin{rmk} Mind that since $d_\a$ is invariant by translation, for any $s>0,$  $\lim_{q \uparrow 1} \zeta_{U(N)}(s,q)=+\infty.$
\end{rmk}

\subsection{Decomposition of the  discrete Yang-Mills measure}

\subsubsection{Generalised Maps on surfaces} \label{sec---GeneralisedMap} A generalised map is a triple   $\Gbb=(V,E,F),$   of three finite sets such that $\mathcal{G}=(V,E)$   is a finite graph, together with
\begin{itemize}
 \item a bijection $\theta :E_0\to E_o$ whose cycles induces a cyclic ordering of $Out_v=\{e\in E_o:\underline{e}=v\}$ for each $v\in V,$ 
\item    a left boundary  map $L:E_0\to F$ such that setting $R(e)=L(e^{-1}),$ 
\begin{equation}
L(e)=R(\theta(e))\quad \forall e\in E_0\label{eq---OrderVertex/FaceCompat}
\end{equation}
\item and a genus map  $g:F\to \mathbb{N}$.
 \end{itemize}
 
A pair $(e,\theta(e))$ with $e\in E_0$ is called a \emph{corner} at the vertex $\underline{e}$ within the face $L(e).$  By \eqref{eq---OrderVertex/FaceCompat},  $L=L\circ \mathcal{S}$ where   $\mathcal{S}:E_0\to E_0, e\mapsto \theta(e^{-1}).$  Orbits of $\mathcal{S}$ are called face boundaries. For  each $f\in F,$ denote  $\partial_1f,\ldots,\partial_{k_f}f $ the orbits of $\mathcal{S}$ in $L^{-1}(f)$ and set
\begin{equation}
\chi (f)= 2-2g(f)-k_f.
\end{equation}

 For each generalised map $\Gbb=(V,E,F)$ a closed orientable surface $\Sigma$  can be built  gluing for each  $f\in F$   a surface $\Sigma_f$ of genus $g_f,$ with $k_f$ boundary components along the $k_f$ loops of 
$X_\mathcal{G}$ associated to $\partial f.$ The orientation of $\Sigma$ can then be chosen  so that the cyclic order of edges around vertices is the clockwise one and for each edge $e\in E_0,$ the  image of $\Sigma_{L(e)}$ in $\Sigma$ 
appears  on the 
left of the embedding of $e.$   We then say the map $\mathbb{G}$ is \emph{embedded} in $\Sigma.$ With this convention for each edge $e\in E_0,$ the surface attached for the face $L(e)$
appears   on the 
left of the image of $e$ in $\Sigma.$  Call \emph{genus} $g$ of the map $\Gbb$ the genus of $\Sigma$ so that 
{\begin{equation}
\#V-\#E+\#F-\sum_{f\in F} 2 g(f)=2-2g.
\end{equation}}

Our main motivation for introducing this definition is the following example.  Assume $\Gbb$ is a { map} embedded in a surface $\Sigma_\Gbb$ and $\mathcal{S}=(\mfl_1,\ldots,\mfl_p)$  is a tuple of  loops of $\Gbb$ having only {simple transverse intersections}.  Consider  
an   open set $X_\mathcal{S}$  of $\Sigma_\Gbb$  that deforms retract to the union of loops' ranges in $\Sigma_\Gbb.$    Let $V$ and $E_+$ (resp. $E$) be the vertices and oriented (resp. non-oriented) edges  of $\Gbb$ visited by loops of $\mathcal{S}$ and let $F$ be the set of connected 
components  of $\Sigma_{\Gbb}\setminus X_\mathcal{S}$.  Each $f\in F$ is a compact, orientable surface with signature $(g_f,k_f)$ and boundary components that deform retract to the range of  loops $\partial f_1, \ldots,\partial f_{k_f}$ of $\mathcal{G}=(V,E)$.   Along each loop $\partial f_i$ the orientation   $E_+$ changes an even number of times  $m_{f,i}$  and we set  
\begin{equation}
m_f=\sum_{i=1}^{k_f}m_{f,i}.\label{eq---OSWAPF}
\end{equation}
The orientation of $\Sigma_\Gbb$  determines then   a left boundary map and  cyclic orderings of outgoing edges that make   the triple  $\Gbb_\mathcal{S}=(V,E,F)$  into a generalised map. We then denote ${{\R^*_+}^{F}}$ by $\Delta_\mathcal{S}$ and $\{a\in \Delta_\mathcal{S}: |a|=T\}$ by $\Delta_\mathcal{S}(T).$

\subsubsection{Discrete Yang-Mills measure on generalised maps}

Assume $\Lambda$ is a subset of $\hat{G}_N.$ For $t>0,s\in\R, $ set for $g_1,\ldots,g_k\in G_N,$
\begin{equation}
p_{t,s}^{(\Lambda)}(g_1,\ldots,g_k)=\sum_{\a\in \Lambda}  e^{-t c_\a}d_\a^s \prod_{i=1}^k\chi_\a(g_i)
\end{equation} 
and for any generalised map $\Gbb=(V,E,F)$ with  area vector $a\in {\R_+^*}^F$ and subset of faces $F_*,$

\begin{equation}
\YM_{\Gbb,a,G_N}^{(\Lambda,F_*)}(dh)=\prod_{f\in F\setminus F_*} p_{a(f),\chi(f)}(h_{\partial f})  \prod_{f\in F_*} p_{a(f),\chi(f)}^{(\Lambda)}(h_{\partial f}) U_{\Gbb}(dh),\label{eq---DefTruncatedYM}
\end{equation}
where for any $f\in F$ with boundary loops\footnote{since $p_{t,s}^{\Lambda}(g_1,\ldots,g_k)$ is invariant by conjugation of $G_N^k,$ the density \eqref{eq---DefTruncatedYM} does not depend on the choice of base points for $\partial f_1,\ldots, \partial f_{k_f}.$} $ \ell_1,\ldots,\ell _{k_f},$ $h_{\partial f}= (h_{\ell_1},\ldots, h_{\ell_{k_f}}).$

\subsection{Refinement and compatibility} For two generalised maps $\Gbb=(V,E,F)$ and  $\Gbb'=(V',E',F')$,  say   $\Gbb$   is a refinement of, or finer than    $\Gbb'$ if  $(V',E')$ is  a connected  sub-graph of $(V,E),$   and the order $\th'$ and genus map $g:F'\to \N$ of $\Gbb'$ are determined by the embedding of $\Gbb.$ 
\begin{rmk} The oriented edges ${E'}^o$ is then a subset ${E}^o$  stable by orientation reversing and   the cyclic order of  ${E'}^{o}$ is the one induced by ${E}^{o}.$  The generalised maps $\Gbb'$ and $\Gbb$ have same genus. 
\end{rmk} 
For such a pair, for any face  $f\in F$ there is a unique face $\iota(f)\in F'$ whose embedding includes the one of $f.$ For $a\in \R_+^F,$  define $a'\in \R_+^{F'}$  setting $a'(x)=\sum_{f\in F:\iota(f)=x} a(f)$ for any $x\in F'.$ The Yang-Mills measures are compatible with refinement in the following sense. Denote  $\mathcal{R}^{\Gbb}_{\Gbb'}: (G^{E^o})_o\to (G^{{E'}^o})_o $ the restriction map.

\begin{prop}\label{Prop--Comp} Assume $\Gbb,\Gbb'$ are two generalised map with $\Gbb$ finer than $\Gbb',$   $a\in {\R^*_+}^{F},$    $F_*\subset F$ and $F'_*=\iota(F_*)$. Then 
\begin{equation}
\mathcal{R}^{\Gbb}_{\Gbb'}(\YM^{(\Lambda,F_*)}_{\Gbb,a,G_N})=\YM^{(\Lambda,F'_*)}_{\Gbb',a',G_N}.\label{eq---CompaYM}
\end{equation}
\end{prop}

The proof is similar to the restriction property of the  Yang-Mills measure (\cite{LevMarkovHol}). We  sketch here the argument. 

\begin{proof} Without loss of generality, $(\# V,\# E)= (\# V'+1,\# E'+1) $  or  $(\# V,\# E)= (\# V',\# E'+1). $  In both cases all but at most two edges of $E$ can be identified with an edge of $E'.$ The identity \eqref{eq---CompaYM}   is equivalent the equating the density of  $\YM^{(\Lambda,F'_*)}_{\Gbb',a',G_N}$ with the integral of   $\YM^{(\Lambda,F'_*)}_{\Gbb,a,G_N}$ along the edge coordinate of $\mathcal{M}(\mathcal{G},G)$ associated to one these two edges edge.  In the first case, a vertex is added to the interior of an edge $e$ of $E'$ dividing it into two edges $e_1,e_2$  and  $F$ can be identified with $F'$. The identify follows integrating with respect to $h_{e_1}$  from the invariance by multiplication of the Haar measure.     Consider the second case. Whether  $\#F=\#F'+1,$ the identity then follows by Peter-Weyl isometry and 
\begin{equation}
\int_{G} p_{t_1,s_1}^{(\Lambda)}(u g, h_1,\ldots,h_k )  p_{t_2,s_2}(g^{-1}v, h_{k+1},\ldots,h_n )  dg=  p^{(\Lambda)}_{t_1+t_2,s_1+s_2-1}(uv,h_1,\ldots,h_n) 
\end{equation}
for any $u,v,h_i\in G, t_i>0, s_i\in \R$. Or $\#F=\#F'$, it is then likewise a consequence of 
\begin{equation}
\int_{G} p_{t,s}^{(\Lambda)}(u gv g^{-1},h_1,\ldots,h_k)dg = p^{(\Lambda)}_{t,{s-1}}(u,v,h_1,\ldots ,h_k)
\end{equation}
or
\begin{equation}
 \int_{G} p_{t,s}^{(\Lambda)}(u g, g^{-1} v ,h_1,\ldots,h_k)dg = p^{(\Lambda)}_{t,s-1}(u v ,h_1,\ldots,h_k)
\end{equation}
for $u,v,h_i\in G, t>0, s\in \R$.
\end{proof}
 
When $\Gbb=(V,E,F)$ is a  map and $T>0,$  consider $\overline{\Delta_\Gbb(T)}=\{a\in \R_+^F: |a|=T\}$ and 
  \[\mathcal{M}_{a,0}(\mathcal{G},G)=\{h\in \mathcal{M}(\mathcal{G},G): h_{\partial f}=\Id_N \,\forall f\in F\,\text{with}\,a_f=0\}.\]
The  following compatibility relation follows from \eqref{def---HK}. Details of its proof are omitted here. 
 
\begin{prop} \label{Prop---Contract} Assume  $T>0,$   $a\in\overline{\Delta_\Gbb(T)} $ and  $F_*\subset F$ with    $a(f)>0$ for all $f\in F_*.$ Then  the following weak limit of signed measures exists
\[\YM^{(\Lambda,F_*)}_{\Gbb,a,G_N}= \lim_{a' \to a, a'\in\Delta_\Gbb(T)} \YM_{\Gbb,a',G_N}^{(\Lambda,F_*)}. \] 
Moreover,  $\YM^{(\Lambda,F_*)}_{\Gbb,a,G_N}$ is supported on $\mathcal{M}_{a,0}(\mathcal{G},G).$ In particular, when $f\in F\setminus F_*$ with  $a_f=0,$ identifying $f$ with the associated closed  $2$-cell of $\Sigma_\Gbb,$ for any loops $\mfl_1,\mfl'_1,\mfl_2,\ldots,\mfl_k,\mfl'_k\in \mathrm{L}(\Gbb)$ such that $\mfl_i$ is homotopic to $\mfl'_i$  relatively to $\Sigma_{\Gbb}\setminus \iota(f)$,   
\[\YM^{(\Lambda,F_*)}_{\Gbb,a,G_N}(\prod_{i=1}^kW_{\mfl'_i})=\YM^{(\Lambda,F_*)}_{\Gbb,a,G_N}(\prod_{i=1}^kW_{\mfl_i}).\]
\end{prop}

\subsection{Exact IRF-formulas for Wilson loops expectations with simple, transverse intersections} \label{sec---StatIRF} When  $\Gbb=(V,E,F)$ is a generalised map,   $\mathcal{R}$ a module and  $x:F\to \mathcal{R}$ a $2$-cycle, denote by $\partial x :E_o \to \mathcal{R}$ its boundary given by 
\[\partial x(e)= x(L(e))-x(R(e)).\]
 
For $N\ge 1$, denote $\mathcal{E}_N $ the set of $N$ elementary vectors  $\om_1=(1,0\ldots,0), \ldots,\om_N= (0,\ldots,0,1)$ of $\Z^N$ and by  $\mathcal{E}^*_N$ its projection to  $\Z^N/\Z.$ When $E_+$ is a fixed orientation of $E,$  introduce the space of Lipschitz functions 
\begin{equation}
\Omega_{N}=\{\a: F\to\ZD: \,\forall e\in E_+,\, -\partial \a(e)\in \mathcal{E}_N \}
\end{equation}
and 
\begin{equation}
\Omega_{N}^*=\{\a: F\to\ZD/\Z: \,\forall e,\, -\partial \a(e)\in \mathcal{E}_N^* \}.
\end{equation}

In the following statement, we fix   a generalised map  $\Gbb_\mathcal{S}=(V,E,F)$  with edges orientation $E_+$ associated to a   loops tuple $\mathcal{S}$ with simple and transverse intersections as in  section \ref{sec---GeneralisedMap}.

     For $\a:F\to\ZD,$  set  
\begin{equation} 
\mathcal{D}_{\a,\mathcal{S}}=\prod_{f\in F} d_{\a(f)}^{\chi(f)-\frac{m_f}2}  \quad\text{and}\quad \mfc_\a(f)= \mfc_{\a(f)} \quad\forall f\in F,
\end{equation} 
and for $a\in \Delta_{\Gbb_\mathcal{S}},$
\begin{equation}
 \<a,\mfc_\a\>=\sum_{f\in F} a(f)\mfc_{\a(f)}.
\end{equation}
When $\a\to\ZD /\Z$,  we use the same notations with  $\mfc^*_\a$ in place of $\mfc_\a.$
 
\begin{thm}[IRF-formula]  \label{thm---IRFformula}For any $T>0$ and any area vector $a\in \Delta_{\Gbb_\mathcal{S}}(T),$ 
\begin{equation}
 \zeta_{U(N)}(2g-2,q_T)\EE_{\YM_{U(N),\Gbb_\mathcal{S},a}}[W_{\mathcal{S}}]= \sum_{\a\in\Omega_N}  \mathcal{D}_{\a,\mathcal{S}} \mathcal{I}_{\a,\mathcal{S} } e^{- \frac 1 2 \langle a,\mfc_\a\rangle } \label{eq---IFRU}\end{equation}
 and 
 \begin{equation}
 \zeta_{SU(N)}(2g-2,q_T)\EE_{\YM_{SU(N),\Gbb_\mathcal{S},a}}[W_{\mathcal{S}}]= \sum_{\a\in\Omega^*_N}  \mathcal{D}_{\a,\mathcal{S}} {{ \mathcal{I}_{\a,\mathcal{S} } }}e^{- \frac 1 2 \langle a,\mfc^*_\a\rangle }, \label{eq---IFRSU}\end{equation}
 where for all $\a:F\to\ZD,$  $\mathcal{I}_{\a,\mathcal{S}}$ is invariant by translation of $\a$ and  $|\mathcal{I}_{\a,\mathcal{S}}|\le 1.$ 
More generally for any subset  $F_*$ of faces of $\Gbb,$  $\Lambda$ or $\Lambda^*$ of $ \ZD$ or $\ZD/\Z$  
\begin{equation}
\YM_{U(N),\Gbb_\mathcal{S},a}^{(\Lambda,F_*)}(W_\mathcal{S})=\sum_{\a\in \Omega_N :\, \forall f\in F_*, \a(f)\in\Lambda}  \mathcal{D}_{\a,\mathcal{S}} \mathcal{I}_{\a,\mathcal{S} } e^{- \frac 1 2 \langle a,\mfc_\a\rangle }\label{eq---IRFGenCut}
\end{equation}
and 
\begin{equation}
\YM_{SU(N),\Gbb_\mathcal{S},a}^{(\Lambda^*,F_*)}(W_\mathcal{S})=\sum_{\a\in \Omega_N^{*} :\, \forall f\in F_*, \a(f)\in\Lambda^*}  \mathcal{D}_{\a,\mathcal{S}} \mathcal{I}_{\a,\mathcal{S} } e^{- \frac 1 2 \langle a,\mfc^*_\a\rangle }.\label{eq---IRFGenCutSpecial}
\end{equation}
\end{thm}

Formulas similar but less explicit than the above were given by Witten \cite{WittenInt}; such sums over weights also occur in interaction-around-the-face (IRF) statistical mechanics 
models. 
The formula \eqref{eq---IFRU}, with the exact form of $\mathcal{I}_{\a,\mathcal{S}}$, was first proved by T. L\'evy. We give in the appendix an argument  different from the original one \cite{IRFThierry}, proving also  \eqref{eq---IRFGenCut}.

\subsection{Bound on truncated discrete Yang-Mills measure}

For $t>0,$ $k\ge 1,$ let us set $\Lambda_k=\{\a\in \ZD:|\a|\le k\}$ and
\begin{equation}
p_t^{(k)}= p_t^{(\Lambda_k)} .
\end{equation} 
 for any map $\Gbb=(V,E,F)$ with a subset of faces $F_*,$  set for any $k\ge 1$ and area vector $a,$
 
\begin{equation}
\YM_{\Gbb,a,G_N}^{(k,F_*)}(dh)=\YM_{\Gbb,a,G_N}^{(\Lambda_k,F_*)}(dh)\label{def---TruncYM}
\end{equation}
and 
\begin{equation}
\YM_{\Gbb,a,G_N}^{(> k,F_*)}=\YM_{\Gbb,a,G_N}-\YM_{\Gbb,a,G_N}^{(k,F_*)}.
\end{equation}

\begin{prop}\label{Prop---TruncationWLE}  Assume $\mfl_1,\ldots,\mfl_p$ are loops of a  map $\mathbb{G}$ of genus $g\ge2$ and  let $F_*$ be a non-empty subset of its 
faces.  There is an integer $d'\ge 1$ and $K>0$ such that for all $k\ge 1 ,$  $N\ge1, T>0$ and any area vector $a\in \Delta_\Gbb(T),$ 
\begin{equation}
|\YM^{(> k,F_*)}_{\Gbb,a,G_N}(W_{\mfl_1,\ldots,\mfl_p})|\le  K N^{d'} \zeta^{(> k)}_{G_N}(2g-2, q_{\frac T 2}).\label{eq---TruncationWittenBound}
\end{equation}
In particular,  for any $d\ge1,$ there is $k'>0$ and $K>0$ such that for all $k\ge k',$  $N\ge1, T>0$ and any area vector $a\in \Delta_\Gbb(T),$ 
 \begin{equation}
|\YM^{(> k,F_*)}_{\Gbb,a,U(N)}(W_{\mfl_1,\ldots,\mfl_p})|\le \frac{K}{N^d}\theta(q_{\frac T2})\quad\text{and}\quad |\YM^{(> k,F_*)}_{\Gbb,a,SU(N)}(W_{\mfl_1,\ldots,\mfl_p})|\le \frac{K}{N^d}.
\end{equation}
\end{prop}
Thanks to the Witten zeta functions  truncation of  proposition \ref{Prop---ConvWittenZeta} and to the IRF formula of Theorem  \ref{thm---IRFformula}, the above claim relies on the following classical bound.

\begin{lem} \label{Lem---RatioPieri} For any $\a,\b\in\ZD$ with $\b-\a\in\mathcal{E}_N,$
\[ \frac 1 N\le \frac{d_\b}{d_\a}\le N.\]
\end{lem}
\begin{proof}  It is equivalent to prove $\frac{d_\b}{d_\a}\le N$ for any $\a,\b\in\ZD$ with $\b-\a\in\mathcal{E}_N\cup-\mathcal{E}_N.$ Since $d_{\g}=d_{-\g}$ for any $\g\in\ZD,$ 
it is enough to prove this claim for $\b-\a\in\mathcal{E}_N.$   Now for any $\a\in\ZD,$ setting $x=\a+\rho,$
\[\sum_{\varepsilon\in\mathcal{E}_N}\frac{d_{\a+\varepsilon}}{d_\a}=\sum_{i=1}^N\prod_{j\not=i}\frac{x_i+1-x_j}{x_i-x_j}\]
is the opposite of the residue at infinity\footnote{Alternatively, the decomposition 
\[\prod_{i=1}^N\frac{z+1-x_i}{z-x_i}-1=\sum_{i=1}^N \prod_{j\not=i}\frac{x_i+1-x_j}{x_i-x_j} \frac{1}{z-x_i}\]
yields the identity at stake, multiplying both sides by $z$ and taking the limit as $z\to\infty.$} of $\prod_{i=1}^N\frac{z+1-x_i}{z-x_i},$ that is $N.$ In particular $\frac{d_{\a+\varepsilon}}{d_\a}\le N $ for all $\varepsilon\in\mathcal{E}_N.$
\end{proof}

\begin{proof}[Proof of \ref{Prop---TruncationWLE}] 
 
Thanks to Proposition \ref{Prop---ConvWittenZeta}, it is enough to prove \eqref{eq---TruncationWittenBound}.  By 
compatibility by refinement of {Proposition \ref{Prop--Comp} (with generalised map)} for both $\YM$ and $\YM^{(k,F_*)}$ for any $k\ge1$, we can assume that $\mfl_1,\ldots,\mfl_p$ have {simple, transverse intersections} and that {$\Gbb$ is the generalised map associated to $\{\mfl_1,\ldots,\mfl_p\}.$ } Now thanks to the {IRF formula of Theorem \ref{thm---IRFformula},}
\begin{equation}
|\YM^{(> k,F_*)}_{U(N),\Gbb,a}(W_{\mfl_1,\ldots,\mfl_p})|\le \sum_{\a: F\to\ZD:  \forall e \in E,  \partial \a(e)\in \mathcal{E}_N, \exists f\in F \,\text{with}\,|\a(f)|>k }  \mathcal{D}_{\a,\mathcal{S}} e^{- \frac 1 2 \langle a,\mfc_\a\rangle }.\label{eq---PartialBoundTruncWLE}
\end{equation}
Consider $\a:F\to\ZD$ as in the above sum and $f_*\in F$ with $|\a(f_*)|>k.$ For any $f\in F,$ thanks to Lemma \ref{Lem---RatioPieri}, 
\begin{equation}
\frac1{N^l}\le \frac{d_{\a(f)}}{d_{\a(f*)}}\le N^l,
\end{equation}
where $l$ is the length of the shortest path in the dual graph from $f_*$ to $f.$ Hence there is an integer $d'$ such that 
\begin{equation}
 \mathcal{D}_{\a,\mathcal{S}}\le N^{d'}  d_{\a(f_*)}^{\sum_{f\in F}\left(\chi(f)-\frac{m_f}{2}\right)}. 
\end{equation}
Denoting by $v$ and $e$ the number of vertices and edges of $\Gbb,$ since  $\Gbb$ is $4$-regular,  
\[\sum_{f\in F}\chi(f)=\chi(\Gbb)+3v-e=\chi(\Gbb)+v.\]
Around each vertex, using the orientation of $\mathcal{S},$ there are two incoming   and two outgoing edges, hence two corners where the edge orientation changes so that
\[\sum_{f}m_f=2 v. \]
All in all,
\begin{equation}
{\sum_{f\in F}\left(\chi(f)-\frac{m_f}{2}\right)}=\chi(\Gbb)=2-2g. \label{eq---DimensionIRF}
\end{equation}
Now for any $f\in F,$ $y_f=\a(f)-\a(f_*)$ is a sum of at most $e$  elementary vectors, it is has at most $e$ non-zero coefficients and  $\|y_f\|_2 \le e.$  Using Lemma \ref{lem---BoundCasimir} below, the summand of \eqref{eq---PartialBoundTruncWLE} is 
 
{upper bounded  by }
\[ e^{\frac{T}{2} (e+1)^3 }N^{d'}d_{\a(f_*)}^{2-2g}e^{-\frac{T}{4} \mfc_{\a(f*)}  }.\]

Setting $K'=e^{\frac{T}{2} (e+1)^3 },$ we  conclude that the right-hand side of \eqref{eq---PartialBoundTruncWLE} is upper-bounded by 
\begin{align*}
K' N^{d'}\sum_{\a_*\in\ZD,|\a|>k,f_*\in F} &d_{\a_*}^{2-2g}e^{-\frac{T}{4} \mfc_{\a(f*)}  } \#\{\a:F\to\ZD: \forall e \in E, \partial \a(e)\in \mathcal{E}_N, \a(f_*)=\a_* \} \\
&\le K' N^{d'+e}  \#F  \sum_{\a\in\ZD,|\a|>k }d_{\a}^{2-2g}e^{-\frac{T}{4} \mfc_{\a}  }\\
&\le K' \#F  N^{d'+e}\zeta_{U(N)}^{(>k)}(2-2g,q_{\frac T 2}).
\end{align*}
A simpler version of the same argument applies to the $SU(N)$ case and is left to the Reader.
 
  \end{proof}

\begin{lem} \label{lem---BoundCasimir} For all $N\ge3, s\ge 1$  and  $\a,\b\in \ZD$ such that $\a-\b$  has at most $s$ non-zero coefficients, 
\[\frac{|\|\b+\rho\|^2-\|\a+\rho\|^2|} N \le  \frac{\|\a+\rho\|^2-\|\rho\|^2}{2N}+  (\sqrt{s}+1)  \|\b-\a\|^2. \]
\end{lem}

\begin{proof} Consider $N,s,\a,\b$ as above.  Setting $y=\b-\a,$ 
\[|\<y,\rho\>|\le   \frac{N}{2}\sum_i|y_i|\le \frac{\sqrt{s}\|y\|}{2}N\le \frac{\sqrt{s}\|y\|^2}{2}N\]
and using GM inequality,
\[| \|\b+\rho\|^2-\|\a+\rho\|^2 - \|y\|^2|=2|\<y,\a+\rho\>|\le  2\|y\|^2+\frac {\|\a\|^2} 2+2|\<y,\rho\>|\le (\sqrt{s}N+2)\|y\|^2 +\frac{\|\a\|^2}2. \]
Since  $\<\a,\rho\>\ge 0,$
\[\|\a+\rho\|^2-\|\rho\|^2\ge \|\a\|^2\]
and the claim follows.
\end{proof}
 
\section{Integration  using Koike-Schur-Weyl duality}
\label{sec---Integration  using Koike-Schur-Weyl duality}

\subsection{Mixed tensors and Koike-Schur-Weyl duality}
Let us recall some standard facts about Schur-Weyl duality for tensors and its slightly less standard generalisation to traceless mixed tensors due to Koike \cite{Koike}. Our motivation here is to realise all 
irrep of $U(N)$  
as traceless mixed tensors.

Denote by $V=\C^N$ the natural representation of $U(N)$ and by $V^*$ its dual. When $n,m\ge 0$ are two integers,
 the space of complex mixed tensors\footnote{All tensor products will be taken in this section over the complex numbers. Moreover by convention, $V^{\ts 0}$ and ${V^*}^{\ts 0}$ are defined as the trivial representation $\C.$ } $T_{n,m}=V^{\otimes n }\otimes {V^*}^{\ts m}$  is endowed with the action of 
the unitary group and of the permutation group $S_{n,m}=S_n\tm S_m$.   When $v_1,\ldots, v_n\in V,$  $\om_1,\ldots, \om_m\in V^*$,  denote  
\begin{equation}
U. (v_1\ts\ldots v_n\ts \om_1\ts\ldots\ts \om_m)= U.v_1\ts\ldots\ts U. v_n\ts \om_1\circ U^{-1}\ts\ldots \om_m\circ U^{-1}. \label{eq---UnitAct}
\end{equation}
for $U\in U(N)$ and
\begin{equation}
\a\tm \b. (v_1\ts\ldots v_n\ts \om_1\ts\ldots\ts \om_m)= v_{\a^{-1}(1)}\ts\ldots\ts  v_{\a^{-1}(n)}\ts \om_{\b^{-1}(1)}\ts\ldots \om_{\b^{-1}(m)} \label{eq---Shuffle}
\end{equation}
for  $ \a\in S_n,\b\in S_m.$ Extending these definitions linearly defines   representations of $U(N)$ and $S_{n,m}$ which commute with one another.  When $n$ or $m$ vanish,  the standard 
Schur-Weyl duality asserts that any endomorphism of $T_{n,m}$ commuting with the unitary action \eqref{eq---UnitAct} is a linear combination of permutation endomorphisms of the form 
\eqref{eq---Shuffle}. \footnote{and vice-versa by bi-commutant theorem, any endomorphism of $T_{n,m}$ commuting with the shuffle action is a  linear combination of diagonal unitary 
endomorphisms of the form \eqref{eq---UnitAct}.}   By Schur's lemma,  the  decomposition of $V^{\ts n}$ into  irreducible components of $S_{n}\tm U(N)$ is multiplicity free, which defines a 
pairing between 
irrep of $S_{n}$ and $U(N)$ occuring in the isotypic decompositions of $T_{n}.$  More explicitly,  
\begin{equation}
V^{\ts n}=\bts_{\la \vdash n:\ell(\la)\le N}V^\la\ts V_\la
\end{equation} 
where\footnote{This is  H. Weyl's construction \cite{FultonHarris,Weyl} of an irrep of $U(N)$   for  each integer partition $\la$ with $\ell(\la)\le N$.  Identifying    $U(1)$ with scalar $U(N)$ matrices, for any 
$t\in V^{\ts n},$ $z. t= z^n t. $  All irreps built this way are therefore non isomorphic. It can further be shown that  any other irrep  is  given by the vector space of the irrep $(V_\la,\rho_\la)$ for some 
integer  
partition $\la$ endowed with the twisted action
 
\[\rho_{\la,c}(U). v = \det(U)^c \rho_\la(U).v,\quad\forall v\in V_\la,\forall U\in U(N) \] 
for some $c\in\Z.$ } $V^\la$ and $V_{\la}$ are the irreps respectively of $S_n$ and $U(N)$ associated to the integer partition $\la\vdash n.$

 When $n,m\ge 1,$ the above duality does not hold: the commutant of the unitary  action on $T_{n,m}$ is strictly larger than the algebra of endomorphism spanned by $S_{n,m}$ as we will recall in the 
 next section. However it is true if we restrict to traceless tensors. 
 
 Assume $n,m\ge 1. $ For $1\le i\le n, 1\le j \le m,$  define $c_{i,j}\in \Hom(T_{n,m},T_{n-1,m-1})$ setting 
 \begin{equation}
 c_{i,j}(v_1\ts\ldots v_n\ts \om_1\ts\ldots\ts \om_m)=\om_j(v_i) v_1\ts\ldots \hat{v_i}\ldots \ts v_n\ts \om_1\ts \ldots \hat\om_j\ldots \ts \om_m
 \end{equation}
 whenever $v_1,\ldots, v_n\in V,$  $\om_1,\ldots, \om_m\in V^*,$ where $\hat{v_i}$ means that the $i$th tensor is omitted. The space of \emph{traceless tensors} of order $(n,m)$ is
\begin{equation}
\overset{\circ}{T}_{n,m}=\bigcap_{i\le n,j\le m} \ker(c_{i,j}).\label{eq---TracelessTensor}
\end{equation}
Since  $c_{i,j}\in \Hom_{U(N)}(T_{n,m},T_{n-1,m-1})$ and $\ker(c_{\sigma(i),\sigma(j)})=\sigma.\ker(c_{i,j})$ for all $i,j$ and $\sigma\in S_{n,m},$    $\overset{\circ}{T}_{n,m}$ is a   
sub-representation of $T_{n,m}$ for both the $U(N)$ and $S_{n,m}$ action. For any $U\in U(N)$ and $\sigma\in S_{n,m}$ we shall write  $\overset{\circ}{\rho}_{n,m}(U)$ and 
$\overset{\circ}{\rho}_N(\sigma)$ the endomorphisms of $\overset{\circ}{T}_{n,m}$ of  these representations. The Koike-Schur-Weyl duality due to K. Koike  is the following 
statement.   

\begin{thm}[\cite{Koike}] The algebras spanned by $U(N)$ and $S_{n,m}$ are commutant of one another. That is, 
\begin{equation}
 \End_{U(N)}(\overset{\circ}{T}_{n,m})= \<\overset{\circ}{\rho}_{N}(S_{n,m})\>\quad\text{and}\quad  \End_{S_{n,m}}(\overset{\circ}{T}_{n,m})= \<\overset{\circ}{\rho}_{n,m}(U(N))\>
\end{equation}
 where the right-hand sides stand respectively for the algebras spanned  by  $\overset{\circ}{\rho}_{n,m}(S_{n,m})$ and $\overset{\circ}{\rho}_{N}(U(N)).$ Moreover, the isotopic decomposition 
 of $\overset{\circ}{T}_{n,m}$ for the action of $ S_{n,m}\times U(N)$ is
 \begin{equation}
 \overset{\circ}{T}_{n,m}=\bigoplus_{\substack{\la\vdash n,\mu\vdash m:\\ \ell(\la)+\ell(\mu)\le N}} V^{[\la,\mu]}\ts V_{[\la,\mu]_N}\label{eq---IsotTraceless}
 \end{equation}
 where $V^{[\la,\mu]}$ is the irrep $V^{\la}\ts V^{\mu}$ of $S_n\tm S_m$ and $V_{[\la,\mu]_N}$ is the irrep of $U(N)$ with highest weight $[\la,\mu]_N.$
\end{thm}

The above result can be used to generalise Weyl's  construction to all irreps of $U(N)$ using traceless tensors spaces in place of tensors.  Indeed  any highest weight of $U(N)$ is uniquely written as  
$[\la,\mu]_{N}$ with   $\ell(\la)+\ell(\mu)\le N.$  For our purpose, it is mainly useful to compute characters $\chi_{[\la,\mu]_N}.$ For any $n,m\ge1, \la\vdash n,\mu\vdash m,$ consider the $S_{n,m}$ irreducible character 
\begin{equation}
\chi^{[\la,\mu]}(\a\tm\b)=\chi^\la(\a)\chi^\mu(\b)\quad\forall \a\in S_n,\b\in S_m.
\end{equation}

{
 
\begin{coro} \label{coro---CharacTraceHarmoT} For any $n,m\ge1, \la\vdash n,\mu\vdash m$ and $U\in U(N),$ 
\begin{equation}
 n!m!\chi_{[\la,\mu]_N}(U)=\Tr_{\overset{\circ}{T}_{n,m}}\left(\rho_{N}(\chi^{[\la,\mu]})\rho_{n,m}(U)\right)
\end{equation}
and for any  irreducible idempotent   $p=\sum_{\sigma\in S_{n,m}} p(\sigma)\sigma$ of $\C[S_{n,m}]$ with $p\chi^{[\la,\mu]}\not=0$,  
\begin{equation}
\chi_{[\la,\mu]_N}(U)=\Tr_{\overset{\circ}{T}_{n,m}}\left(\rho_{N}(p)\rho_{n,m}(U)\right).
\end{equation}
 
 \end{coro}}
\begin{proof} Recall the action of $\frac{\chi^{[\la,\mu]}(1)}{n!m!}\chi^{[\la,\mu]}$ on  any $S^{n,m}$-module is the $S^{n,m}$-equivariant projection onto its  $V^{\la,\mu}$-isotypic component. Consider the module $T_{n,m},$ using the isotypic decomposition  \eqref{eq---IsotTraceless}, we conclude that
\begin{align*}
\frac{\chi^{[\la,\mu]}(1)}{n!m!}\Tr_{\overset{\circ}{T}_{n,m}}\left(\rho_{N}(\chi^{[\la,\mu]})\rho_{n,m}(U)\right)&=\Tr_{V^{[\la,\mu]}\ts V_{[\la,\mu]_N}}(\rho^{[\la,\mu]}({\id_{n,m}})\ts\rho_{[\la,\mu]_N}(U))\\
&=\chi^{[\la,\mu]}(1)\chi_{[\la,\mu]_N}(U).
\end{align*}
The second identity follows in a similar way.
\end{proof}
Consider the $S_{n,m}\tm U(N)$-equivariant projection\footnote{It is the orthogonal projection $\overset{\circ}{T}_{n,m}$ associated to any $S_{n,m}\tm U(N)$-invariant inner product on $T_{n,m}$. Denoting by the same symbol $\<,\>$  the standard inner product on $V=\C^N$ and its dual, setting
\begin{equation}
\<v_{1}\ts\ldots v_n\ts \om_1\ldots\ts \om_m,v'_{1}\ts\ldots v'_n\ts \om'_1\ldots\ts \om'_m\>=\prod_i \<v_i,v_{i}'\>\prod_j \<\om_j,\om_{j}'\>\quad \forall v_i\in V, \om_j\in V^{*} \label{eq---IPTensors}
\end{equation}
defines  an equivariant inner product on $T_{n,m}.$
} 
$P_{N}^{n,m}$ from $T_{n,m}$ onto $\overset{\circ}{T}_{n,m}$ associated to the isotypic decomposition of $T_{n,m}.$ The above formula implies
\begin{equation}
 \chi_{[\la,\mu]_N}(U)=\frac{1}{n!m!}\Tr_{T_{n,m}}\left(P_{N}^{n,m}\rho_{N}(\chi^{[\la,\mu]})\rho_{n,m}(U)\right).\label{eq---CharacterTraceless}
\end{equation}

\begin{rmk} Alternatively, $\chi_{[\la,\mu]_N}$ can be expressed using Littlewood-Richardson's coefficients \cite[Thm 2.3]{Koike} in terms of  Schur polynomials.  However, we shall not attempt here to analyse this formula as $N\to\infty$ and will  prefer instead \eqref{eq---CharacterTraceless} for that purpose.
\end{rmk}

Since the action of  $P_N^{n,m}$ on mixed tensors commutes with the one of $U(N),$ we shall next look for an expression for the Schur-Weyl dual of   $P_{N}^{n,m}.$ 

\begin{rmk}  As $P^{n,m}_{N}$ is also in the bi-commutant of the action of $U(N),$ it is also natural to look for an expression of $P_{n,m}^N$ as a weighted integral of $\rho_{n,m}(U)$.  From the $U(N)$-isotypic 
decomposition of $T_{n,m},$ it follows that 
\begin{equation}
P_{n,m}^N=\int_{U(N)} f_{n,m}(U^{-1})\rho_{n,m}(U)dU
\end{equation}
where 
\begin{equation}
f_{n,m}=\sum_{\la\vdash n,\mu\vdash m:\ell(\la)+\ell(\mu)\le N} d_{[\la,\mu]_N}\chi_{[\la,\mu]_N}.
\end{equation}
\end{rmk}
\begin{rmk}As argued in  \cite{BenkartCO} and \cite{MageeII}, an alternative to \eqref{eq---CharacterTraceless} is to consider a highest weight vector $v_{[\la,\mu]_N}$ of weight $[\la,\mu]_N$ in $ \overset{\circ}{T}_{n,m}$ and,    denoting by $P_{v_{[\la,\mu]_N}}$  the orthogonal projection from $T_{n,m}$  onto $v_{[\la,\mu]_N}$ for the  standard inner product \eqref{eq---IPTensors},  to notice 
{
\begin{equation}
P_{[\la,\mu]_N}=d_{[\la,\mu]_N}  \int_{U(N)} \rho_{n,m}(U)P_{v_{[\la,\mu]_N}} \rho_{n,m}(U^{-1})dU  \label{eq---CharacterWithHW}
\end{equation}}
is the equivariant projection from $T_{n,m}$ to the irreducible component of $\overset{\circ}{T}_{n,m}$ containing $v_{[\la,\mu]_N}.$ Hence
 
\begin{equation}
\chi_{[\la,\mu]_N}(U)=\Tr_{T_{n,m}}(P_{[\la,\mu]_N}\rho_{n,m}(U)).
\end{equation}
In \cite{MageeII}, to bound expected Wilson loops,  the author bounds the Schur-Weyl dual of $\int_{U(N)} \rho_{n,m}(U)P_{v_{[\la,\mu]_N}} \rho_{n,m}(U^{-1})dU$  and bounds      $d_{[\la,\mu]_N}$ independently using  \cite[Thm 2.3]{Koike}.    Using instead the Schur-Weyl dual of $P_{n,m}$ has the  advantage to recover this bound on the dimension but also to yield  combinatorial $\frac 1N$-expansions of the dimension and in turn of Yang-Mills partition function and of  Wilson loops moments. 
  \end{rmk}

\subsection{Walled Brauer algebras and traceless mixed tensors}

The commutant of the action of $U(N)$ on $T_{n,m}$ is strictly larger than the algebra spanned by permutations but can also be described explicitly as follows thanks to the so-called walled Brauer algebra \cite{Nikitin,HalversonWBrauer,DBrauer}.  Let us recall what we need from this duality. 

Assume $n,m\ge 1$. A walled Brauer diagram  is  an involution $\pi$ without fixed point  on $\{-m,\ldots,-2,-1,1,2,\ldots,n\}\times \{-1,1\}$ satisfying 
\[ \mathrm{sign}(v')\eta' =-\mathrm{sign}(v)\eta \]
whenever $\pi(v,\eta)=(v',\eta').$ In other words, $\pi$ is a perfect matching of the complete graph with vertices $\{-m,\ldots,-2,-1,1,2,\ldots,n\}\times \{-1,1\}$ with no matching crossing the $y$-coordinate axis.
Denote the set of walled Brauer diagrams by $\mathcal{D}_{n,m}$ and  by $\mathcal{B}_{n,m}(N)$ the formal complex linear combinations in $\mathcal{D}_{n,m}.$ When $\pi,\mu\in \mathcal{D}_{n,m}$ concatenating $\pi$ on top of $\mu$ yields a diagram $c(\pi,\mu)$  matching top and bottom vertices as well as  $l(\pi,\mu)$ loops that are  disconnected from the top and the bottom vertices. Setting
\[\pi.\mu=N^{l(\pi,\mu)} c(\pi,\mu)\]
defines an algebra structure on $\mathcal{B}_{n,m}(N)$ called the walled Brauer algebra. For any $\pi\in\mathcal{D}_{n,m},$ denote by $\rho_N(\pi)$ the endomorphism of $T_{n,m}$ with   
 
\begin{align}
&\<\rho_N(\pi)u_{1,-1}\ts\ldots u_{n,-1}\ts u_{-1,-1}\ts\ldots u_{-m,-1} ,u_{1,1}\ts\ldots u_{n,1}\ts  u_{-1,1}\ts\ldots u_{-m,1} \>\nonumber\\
&=\prod_{v,v':\pi(v)=v'} \<u_v,u_{v'}\> \quad { \forall u_{j,-1}, u_{-i',1} \in V,  u_{i,1}, u_{-j',-1}\in V^*, 1\le i,j\le n, 1\le i',j'\le m.}   \label{def---BrauerAction}
\end{align}
It is elementary to check that this defines an algebra morphism $\rho_N:\mathcal{B}_{n,m}(N)\to \End_{U(N)}(T_{n,m}).$ It can be further be shown this map is surjective for all $N\ge 1$ and injective  if 
$N\ge n+m$ \cite{HalversonWBrauer,BenkartCO}.  
 
Hence the action of   $U(N)$ and  $\mathcal{B}_{n,m}(N)$ on $T_{n,m}$ are mutually commutant. Let us use this point of view to express elements of $\End_{U(N)}(T_{n,m}).$

For each $\pi\in\mathcal{D}_{n,m}$,   there are as many top as bottom horizontal strings.  Denote  half the number of horizontal strings by $h(\pi).$ The linear span  $\mathcal{J}_{1}$ of $\{\pi\in\mathcal{D}_{n,m}:h(\pi)>0\}$ is a left and right ideal of $\mathcal{B}_{n,m}(N).$  When  $h(\pi)=0$, we call $\pi$ a permutation diagram; 
permutation diagrams can and will be identified below  with permutations in $S_{n,m}.$  
 
 \begin{lem} \label{Lem---HarmoProjectorIdCompo} For any $N\ge n+m,$ 
 For $n,m\ge 1, N\ge n+m$, there is a unique 
\begin{equation}
q_{n,m}= \sum_{\tau\in \mathcal{D}_{n,m}} \kappa_N(\tau)\tau \in \mathcal{B}_{n,m}(N)\label{eq-RepTracelessNonExp}
\end{equation}
with
\begin{equation}
 P^{n,m}_N= \rho_N(q_{n,m}).
\end{equation}
It further satisfies
\begin{equation}
q_{n,m}-\id_{n,m}\in \mathcal{J}_1\quad\text{and}\quad { q_{n,m}\mathcal{J}_1=\mathcal{J}_1q_{n,m}=\{0\}}. \label{eq---RelationBrauerHarmo}
\end{equation}
\end{lem}
\begin{proof} Since $P_{n,m}$ is $U(N)$-equivariant, existence  and uniqueness follows from  the invertibility of  $\rho_N: \mathcal{B}_{n,m}(N)\to \End_{U(N)}(\mathcal{T}_{n,m})$ for $n+m\ge  N.$   Let us write $q_{n,m}=\sum_{\pi\in\mathcal{D}_{n,m}}\kappa(\pi) \pi,$  $(e_i)_i$ a basis of $V$ and $(\om^{N-i})_i$ the dual basis of $V^*$. For $N\ge n+m,$  $e_1\ts \ldots \ts e_n \ts \om^1\ts \ldots \ts \om^m\in \overset{\circ}{T}_{n,m}$ and applying the projection $P_{n,m}^N=\rho_{N}(q_{n,m}),$
\[e_1\ts \ldots \ts e_n \ts \om^1\ts \ldots \ts \om^m=\sum_{\a\times \b\in S_{n,m}} \kappa(\a^{-1}\times \b^{-1})e_{\a(1)}\ts \ldots e_{\a(n)}\ts \om^{\b(1)}\ldots \ts \om^{\b(m)}.\]
Since the vectors of the  summand are linearly independent for $N\ge n+m,$ $\kappa(\sigma)=0$ for $\sigma\in S_{n,m}\setminus\{\id\}$ and $\kappa(\id)=1$ so that $q_{n,m}-\id_{n,m}\in\mathcal{J}_1.$ At last, for any $t\in T_{n,m} $ and $1\le i\le n, 1\le j\le m,$  $c_{i,j}(t)=0$ if and only if $\rho_N({\<i\,-j\>})(t)=0,$ where $\<i\,-j\>$ is the diagram\footnote{following \cite{LevMF}, we call such diagrams Weyl's contractions.}  with two horizontal strings from $\{i,\pm1\}$ to $\{-j,\pm1\}$ and only vertical strings otherwise. The second identity follows.
\end{proof}
The following can be proved similarly.  
\begin{lem}\label{lem---CharacProjB} For $n,m\ge 1,$  $P_{n,m}=\rho_N(q_{n,m})$ whenever $q_{n,m}\in\mathcal{B}_{n,m}$ satisfies
\begin{enumerate}
\item $q_{n,m}^2=q_{n,m},$
\item $q_{n,m}x=x q_{n,m}=0$ for all $x\in  \mathcal{J}_{1}.$
\end{enumerate}
Equivalently,  $\Id_{\mathcal{T}_{n,m}}-P_{n,m}=\rho_N (p_{n,m}),$   whenever $p_{n,m}\in\mathcal{J}_{1}$ with 
\begin{equation}
p_{n,m}x=x p_{n,m}=x \quad \forall x\in  \mathcal{J}_{1}.
\end{equation}
Furthermore $q_{n,m}$ and $p_{n,m}$ are unique when  $N\ge n+m$ and then satisfy
\begin{equation}
q_{n,m}y=yq_{n,m}\quad\text{and}\quad p_{n,m}y=yp_{n,m}\quad \forall  y\in \C[S_{n,m}].
\end{equation}
\end{lem}

\begin{ex} This lemma can be used to compute explicitly $q_{n,m}.$ For $n,m\ge 1,$ denote by $e_1$  the Brauer diagram with only vertical strands except for two horizontal strings matching  points with first coordinate $-1,1$. When $n=m=1,$  since the equivariant projection on traceless matrices $M_N(\C)$ is the map $f\mapsto f-\frac1 N\Tr(f),$ for any $N\ge 1,$
\begin{equation}
q_{1,1}= \id_{1,1}- \frac{1}{N}e_1.
\end{equation}
When $n\ge 1, m=1,N\ge n,$  $N+X_n$ is invertible in $\C[S_n]$ and it easily follows from Lemma \ref{lem---CharacProjB}    that
\begin{equation}
q_{n,1}=\sum_{i\le n} (i\, n) (N+X_n)^{-1} e_1 (i\, n),
\end{equation}
where $X_n=(1\,n)+\ldots+(n-1\,n )$ is the $n$-th Jucys-Murphy element.
\end{ex}

The next lemma is the starting point of our bound on Wilson loop expectations.
\begin{lem}\label{lem---ProjTraceLessBound} For $n,m\ge 1, N\ge n+m$ and  for any $\tau\in \mathcal{D}_{n,m},$  
\begin{equation}
\kappa_N(\tau)=O(N^{-h(\tau)}).\label{eq---BoundHarmonicProj}
\end{equation}
\end{lem}

   Its more explicit versions and a proof are given in the next sub-section and in appendix \ref{Section---RepBrauer}.  The  
 expression of the next sub-section  allows us to prove  
 $\frac 1N$-expansions of Yang-Mills partition function of a surface of genus $g\ge 2$ that appeared  in the physics literature \cite{RamKim}.

When $\la\vdash n,\mu\vdash m,$ we shall also consider    the composition 
\begin{equation}
P^{\la,\mu}_N=\frac{d^\la d^\mu}{n! m!} \rho_N(\chi^{[\la,\mu]})P_{N}^{n,m}\label{eq---EquiProjIrrepSym}
\end{equation}
of $P_{N}^{n,m}$ with the $S_{n,m}$-equivariant projection on 
the $V^{[\la,\mu]}$-component of $\overset{\circ}{T}_{n,m}.$ Since the multiplication by $S_{n,m}$ does not change the number of horizontal strings of a diagram,  Lemma \ref{lem---ProjTraceLessBound} implies the following.

\begin{coro}\label{lem---ProjMixedTIrrep}Assume $\lambda\vdash n,\mu\vdash m,$ 
\begin{equation}
 P^{\lambda,\mu}_N= \sum_{\tau\in \mathcal{D}_{n,m}} \kappa_N^{\la,\mu}(\tau)\tau
\end{equation}
where for any $\tau\in \mathcal{D}_{n,m},$ $\kappa^{\la,\mu}_N(\tau)=O(N^{-h(\tau)}).$
\end{coro}

\subsection{$\Omega$-functions and combinatorial expression for the expansion of the partition function}  Let us give here an explicit expression for the representation \eqref{eq-RepTracelessNonExp} of the 
projection on traceless tensors as an element of the walled Brauer algebra.  This part is independent from the rest of the text and is not strictly necessary to understand the main argument for the convergence of Wilson loop moments of Theorem 
\ref{THM---MFConv}.

A first  motivation is to give a generalisation to the traceless tensor setting of the by now standard unitary integration tools of tensors pioneered by D. Weingarten  and S. Samuel and  known as Weingarten 
calculus, which has been rigorously proved and further investigated in the regime where $N\to\infty$  see  \cite{Collins,CollinsSn} or \cite[Sect. 3]{LevMF}.   These formulas for non-mixed tensors can be formulated thanks to a specific element  $\mathrm{Wg}_n$ of $\C[S_n]$ called Weingarten function.

Another motivation is to obtain combinatorial expressions for coefficients of $\frac{1}{N}$-expansions for Wilson loop expectations or for the Yang-Mills partition function.  We provide below such an expansion  with  a  rigorous interpretation  for the formulas 
occurring in the physics literature due Gross and Taylor    \cite{GTI,GTII} (see \cite{CMR} for a later review),  in particular clarifying the notion of $\Omega^{-1}$-factor appearing therein that we define here as an element of $\C[S_{n,m}]$ generalising 
the Weingarten function.   We  recover first some formulas found by S. Ramgoolam and Y.  Kimura in \cite[equations (2.7) and (2.8)]{RamKim}  giving independent arguments.

When $n,N\ge 1,$ consider $\Omega_n^N\in \C[S_n]$ the character associated to the representation $V^{\otimes n}$ of $S_n,$ that is\footnote{mind that $\#\sigma=\#\sigma^{-1}$ for any permutation.} 
\begin{equation}
\Omega_n^N= \sum_{\sigma\in S_n} \Tr_{V^{\ts n}}(\rho_N(\sigma^{-1}))\sigma=\sum_{\sigma\in S_n} N^{\# \sigma^{-1}} \sigma.
\end{equation}   
When the context is clear we shall drop the superscript $N$. When $N\ge n,$ $\Omega_n$ is invertible (see e.g. \cite{LevSW,DBrauer} or  Appendix \ref{Section---RepBrauer}) in the algebra $\C[S_n]$. Its inverse 
\begin{equation}
\mathrm{Wg}_{n}^N=\Omega_n^{-1}
\end{equation} 
is called the \emph{Weingarten function}. Its usefulness stems in part from allowing to compute the Schur-Weyl dual of an equivariant invariant tensor from its traces. 

\begin{lem}[\cite{LevMF,CollinsSn,FranckII}] \label{Lem---SchurWeylCumu} Assume $T\in \End_{U(N)}(V^{\ts n})$ given by Schur-Weyl duality by $T=\rho_N(\kappa_T)$ where $\kappa_T\in \C[S_n]$ and consider $m_T=\sum_{\sigma\in S_n}m_T(\sigma)\sigma$ with 
\[m_T(\sigma)=\Tr_{V^{\ts n}}(T\rho_N(\sigma^{-1}))\quad\forall \sigma\in S_n.\] 
Then for $n\ge N,$
\[\kappa_T=\mathrm{Wg}_{n}^N*m_T.\]
\end{lem}
\begin{proof}  For any $n,N\ge1,$
\[m_T=\sum_{\alpha,\beta\in S_n}\kappa_T(\a)\Tr_{V^{\ts n}}(\rho_N(\a \b^{-1}))  \beta=\sum_{\alpha,\beta\in S_n}\kappa_T(\a) \Omega_{n}(\beta\a^{-1})  \beta=\Omega_n*\kappa_T.\]
When $N\ge n,$ $\Omega_n$ is invertible and the claim follows.
\end{proof}
 
\begin{rmk}Denote by $(e_i)_{1\le i\le N}$ the canonical basis of $\C^N$, set $e_I=e_{i_1}\ts\ldots \ts e_{i_n}$ for $I\in [N]^n$ and $T_{I,J}\in\End(V^{\ts n})$  the tensor with 
\[\<T_{I,J}e_{J'_1,\ldots,J'_n},e_{I'_1}\ts \ldots e_{I'_n}\>=\delta_{I,I'}\delta_{J,J'}\quad\forall I',J'\in[N]^n.\]
  Considering a Haar unitary matrix $U$ on $U(N),$ applying this lemma  to $\hat{T}_{J,J'}=\mathbb{E}[\rho_n(U)T_{J,J'}\rho_n(U^{-1})]$ with $J,J'\in [N]^n,$ it is simple to recover the integration formula 
  \begin{equation}
  \EE[U_{i_1,j_1}\ldots U_{i_n,j_n}\overline{U}_{i'_1,j'_1}\ldots \overline{U}_{i'_n,j'_n}]=\< \hat{T}_{J,J'}, T_{I,I'}\>=\sum_{\a,\b\in S_n:i_{\a(k)}=i'_k,j_{\b(k)}=j'_k} \mathrm{Wg}_{n,N}(\a\b^{-1})\label{eq---WeingartenFormulaUnit}
  \end{equation}
  for $I,I',J,J'\in[N].$
 \end{rmk}

In view of the isotypic decomposition of  $V^{\ts n},$ the $\Omega$ function also allows to compute the dimension of $U(N)$ irreps in terms of characters of the symmetric group.

\begin{lem}[\cite{LevSW}] For any $n\ge 1$ and  any partition $\la\vdash n$ with $\ell(\la)\le N,$
\begin{equation}
d_{[\la,\emptyset]_N}=\frac{\chi^\lambda(\Omega_{n}^N)} {n!}.
\end{equation}

\end{lem} 
\begin{proof} For any $\la\vdash n,$ $\rho_N(\frac{d^{\la}}{n!}\chi^\la)$ is the equivariant projection on the $V^\la$-isotypic component of $V^{\ts n}.$ By Schur-Weyl duality it is null  for $\ell(\la)>N$ and isomorphic to $V^{\la}\ts V_{[\la,\emptyset]_N}$ when $\ell(\la)\le N.$ In the latter case, we can conclude since the dimension of this component is
\[d^\la d_{[\la,\emptyset]_N}= \frac {d^\la} {n!}\Tr_{V^{\ts n}} (\rho_N( \chi^\la))=\frac {d^\la} {n!} \chi^\la(\Omega^N_n).\]

\end{proof}

\begin{rmk}In particular, \[\frac{d_{[\la,\emptyset]_N}}{N^n}=1+\sum_{k=1}^{n-1}a_k(\la) N^{-k}  \]
with $a_k(\la)=\frac{1}{n!} \sum_{\sigma\in S_n: n-\#\sigma=k}\chi^\la(\sigma)$ and the above expression gives a Taylor expansion of $d_{[\la,\emptyset]_N}$ for $\la$ fixed as $N\to\infty.$ For instance, when  $1\le h\le n$, for $\la=(1^n),$ $V^\la$ is the sign representation of $S_n,$ $V_{[\la,\emptyset]_N}= \Lambda^n(V)$ is the $n$-th exterior product of $V$ and
\begin{equation}
d_{[1^n,\emptyset]}= \frac{N(N-1)\ldots(N-n+1)}{n!}=\sum_{k=0}^{n-1} s(n,k)N^k,
\end{equation}
where $s(n,k)$ are Stirling numbers of the first kind.
\end{rmk}

\begin{rmk} For any $n,m\ge 1,$ $\la\vdash n,\mu\vdash m$ partitions, the above formula applies as well to compute $d_{[\la,\mu]_N}=d_{[\la,\mu]_N+(\mu_1,\ldots,\mu_1)}=d_{[\tilde \la,\emptyset]_N}$ with $\tilde\la\vdash n+N\mu_1-m.$  Though as 
$N\to\infty,$ it is a priori harder to use   since $|\tilde \la|\to\infty$ .
\end{rmk}

\label{subsec---OmegaPoint}

Let us now consider traceless tensors in place of tensors. For $n,m\ge 1,$ consider $\Omega_{n,m}\in\C[S_{n,m}]$ the character associated to the representation $\overset{\circ}{T}_{n,m}$ of $S_{n,m},$ that is 
\begin{equation}
\Omega_{n,m}^N=\sum_{\sigma\in S_{n,m}}\Tr_{\overset{\circ}{T}_{n,m}}(\rho_N(\sigma^{-1}))\sigma.
\label{eq---OmegaMixedDef}\end{equation}

Just as in the non-mixed tensor case, it can be motivated by the following proposition.  Let $\th: \mathcal{D}_{n,m}\to S_{n+m}$ be the  bijection mapping a walled diagram $\pi$ to a permutation diagram by transposing the $m$ last top and bottom vertices. Denote by  $\theta: \mathcal{B}_{n,m}(N)\to \C[S_{n+m}]$ its linear extension.

\begin{prop}\label{Prop---HarmoWeingarten} Assume $T\in \End_{U(N)}(\overset{\circ}{T}_{n,m})$ given by Koike-Schur-Weyl duality by $T=\rho_N(\kappa_T)$ where $\kappa_T\in \C[S_{n,m}]$ and consider $m_T=\sum_{\sigma\in S_n}m_T(\sigma)\sigma$ with 
\[m_T(\sigma)=\Tr_{\overset{\circ}{T}_{n,m}}(T\rho_N(\sigma^{-1}))\quad\forall \sigma\in S_{n,m}.\] 
For  $N\ge n+m,$   $\Omega_{n,m}^N$ is central and invertible in $\C[S_{n,m}]$ with inverse 
\begin{equation}
\Wg_{n,m}^N=(\Omega_{n,m}^N)^{-1}=\sum_{\a\times \b\in S_{n,m}} \Wg_{n+m}^N(\a\times\b)\a\times\b\label{eq---HarmoWeingarten}
\end{equation}
and 
\begin{equation}
\kappa_T=\Wg_{n,m}^{N}*m_T.\label{eq--CenteredCumu}
\end{equation}
Moreover,
\begin{equation}
q_{n,m}=\Omega_{n,m}\theta^{-1}(\Wg^N_{n+m})=\theta^{-1}(\Wg_{n+m}^N)\Omega_{n,m}.\label{eq---FormulaProHarmoWeingarten}
\end{equation}
\end{prop}
 Expressions for $\Omega_{n,m}$ seem to have first appeared in the physics literature \cite{GTII,RamKim}.  In the appendix  \ref{sec---GrossTaylor}, we  prove the first  expression\footnote{in a slightly  different form.} \eqref{eq---OmegaFunctionGT}  discovered by   \cite{GTII} for $\Omega_{n,m}$. The above alternative  identity \eqref{eq---HarmoWeingarten}   is due to \cite{RamKim} with an argument relying on \cite[Thm 2.3]{Koike}.  We give now an independent 
 detailed argument using Lemma \ref{Lem---SchurWeylCumu}.  We shall need first the following observation.

\begin{proof}[Proof of Prop. \ref{Prop---HarmoWeingarten}]  
Identifying $ \End(T_{n,m})$ and $ \End(V^{\ts n+m})$ respectively with $V^{\ts n}\ts {V^*}^{\ts m} \ts {V^{*}}^{\ts n}\ts V^{\ts m}$ and $V^{\ts n}\ts {V}^{\ts m} \ts {V^{*}}^{\ts n}\ts {V^*}^{\ts m},$ consider the  $U(N)\times S_{n,m}$-equivariant application 
$\Theta: \End(T_{n,m})\to \End(V^{\ts n+m})$ given by the partial transposition,  mapping $v_1\ts\omega_1 \ts \omega_2 \ts v_2$ to $v_1\ts v_2 \ts \omega_2 \ts \om_1$   for $v_1\in V^{\ts n},v_2\in V^{\ts m},\om_1\in {V^{*}}^{\ts m},\om_2\in {V^{*}}^{\ts n}.$  Let us note two properties of
 $\Theta:$  
\begin{equation}
\Tr_{V^{\ts n+m}}(\Theta(T))=\Tr_{T_{n,m}}(T)\quad \forall T\in  \End(T_{n,m})
\end{equation}
and 
\begin{equation}
\Theta(\rho_N(x))=\rho_N(\theta(x))\quad\forall x\in\mathcal{B}_n(N).
\end{equation}
Now for $T\in\End_{U(N)}(\overset{\circ}{T}_{n,m})$ and $\kappa_T\in \C[S_{n,m}]$ as above,   consider  $\tilde T\in \End_{U(N)}(T_{n,m})$ with $\tilde T P_{n,m}^N= \tilde T $ and  $\tilde T v= Tv$ for $v \in \overset{\circ}{T}_{n,m}$ so that $\tilde T=\rho_N (\kappa_T)P_{n,m}^N.$ Since for any $\sigma\in S_{n+m}\setminus S_{n,m}, $  $P_{n,m}^N\Theta^{-1}(\sigma)=0,$
\begin{align*}
m_{\Theta(\tilde T)}&= \sum_{\sigma\in S_{n+m}} \Tr_{V^{\ts n+m}}(\Theta(\tilde T)\rho_N(\sigma^{-1}))\sigma=\sum_{\sigma\in S_{n+m}} \Tr_{T_{n,m}}(\tilde T\Theta^{-1}\left[\rho_N(\sigma^{-1})\right])\sigma \\
&= \sum_{\a\times\b\in S_{n,m}} \Tr_{\overset{\circ}{T}_{n,m}}(T\Theta^{-1}(\a^{-1}\times \b^{-1}))\a\times \b\\
&= \sum_{\a\times\b\in S_{n,m}} \Tr_{\overset{\circ}{T}_{n,m}}(T\rho_N(\a^{-1}\times \b))\a\times \b=\th(m_T).
\end{align*}
From Lemma \ref{Lem---SchurWeylCumu}, it follows that for $N\ge n+m,$
\begin{equation}
\Theta(\tilde T)=\rho_N(\Wg_{n+m,N}*\th(m_T)).
\end{equation}
Since   $\Theta(\tilde T)=\rho_N (\theta(\kappa_T q_{n,m}))$ and  $\rho_N:\C[S_{n+m}]\to\End(V^{\ts n+m})$ is injective for $N\ge n+m,$  
\begin{equation}
\kappa_T q_{n,m}=\theta^{-1}(\Wg_{n+m,N}*\th(m_T)).\label{eq---WeingartenHarmo}
\end{equation}
For any $x\in\mathcal{B}_{n+m}(N),$  denote by $[x]_0$ its projection onto $\mathcal{B}_{n,m}(N)/\mathcal{J}_1.$ Thanks to Lemma \ref{Lem---HarmoProjectorIdCompo}, we conclude that 
\begin{align*}
[\kappa_T]_0&= \sum_{\sigma\in S_{n+m}} \Wg_{n+m,N}(\sigma) [\theta^{-1}(\sigma*\th(m_T))]_0\\
&= \sum_{\sigma\in S_{n,m}} \Wg_{n+m,N}(\sigma) [\theta^{-1}(\sigma*\th(m_T))]_0\\
&= \sum_{\a\times \b,\g\times\delta\in S_{n,m}} \Wg_{n+m,N}(\a\times\b) m_T({\g}^{-1}\times{\delta}^{-1}) [\theta^{-1}(\a\times\b*\g\times {\delta}^{-1})]_0\\
&= \sum_{\a\times \b,\g\times\delta\in S_{n,m}} \Wg_{n+m,N}(\a\times\b) m_T({\g}^{-1}\times{\delta}^{-1}) [\a\g\times \delta\b]_0\\
&=\sum_{\a\times \b,\g\times\delta\in S_{n,m}} \Wg_{n+m,N}(\a\times\b) m_T({\g}^{-1}\times{\delta}^{-1}) [\a\g\times \b\delta]_0= [\Wg_{n+m,N}]_0 [m_T]_0,
\end{align*}
where we used in the penultimate and last identities that $\Wg_{n+m,N}$ is central.  Identifying $\mathcal{B}_{n,m}(N)/\mathcal{J}_1$ with $\C[S_{n,m}]$ the identity \eqref{eq--CenteredCumu} follows.

Let us now   deduce \eqref{eq---FormulaProHarmoWeingarten} using the previous argument applied  to $T=\Id_{\overset{\circ}{T}_{n,m}}. $ Then $\kappa_T=\id_{n,m},$  $m_T=\Omega_{n,m}$ and    \eqref{eq---WeingartenHarmo} together with centrality of $\Wg_{n+m}^N$ yields
\begin{align*}
 q_{n,m}&=\sum_{\sigma\in S_{n+m},\a\times\b \in S_{n,m}} \Wg^N_{n+m}(\sigma) \Omega_{n,m}(\a\times\b)\th^{-1}(\sigma . (\a \times \b))\\
 &=\sum_{\sigma\in S_{n+m},\a\times\b \in S_{n,m}}\Wg^N_{n+m}(\sigma) \Omega_{n,m}(\a\times\b)\th^{-1}(  (\a\times \id_m)\sigma (  \id_{n}\times \b))\\
 &=\Omega_{n,m}\th^{-1}(\Wg_{n+m}^N)=\th^{-1}(\Wg_{n+m}^N)\Omega_{n,m}.
\end{align*}

Since $\Omega_{n,m}$ is associated to a character, it is central. To conclude\footnote{alternatively, projecting  the last identity onto $\mathcal{B}_{n,m}(N)/\mathcal{J}_1\simeq \C[S_{n,m}]$ and using Lemma \ref{Lem---HarmoProjectorIdCompo} implies that $\Wg_{n,m}^N$ is the inverse of $\Omega_{n,m}^N$, minding that since $\Wg_{n+m}^N$ is central,  $[\Wg_{n+m}^N]_0=[\th^{-1}(\Wg_{n+m}^N)]_0.$ } and prove \eqref{eq---HarmoWeingarten},  observe that for any $n,m,N\ge1,$  $T\in \End(\overset{\circ}{T}_{n,m})$ and $\kappa_T$ as above, just as for non-mixed tensors, 
\[m_T=\sum_{\alpha,\beta\in S_{n,m}}\kappa_T(\a)\Tr_{\overset{\circ}{T}_{n,m}}(\rho_N(\a \b^{-1}))  \beta=\sum_{\alpha,\beta\in S_{n,m}}\kappa_T(\a) \Omega_{n,m}(\beta\a^{-1})  \beta=\Omega_{n,m}*\kappa_T.\]
Using this last identify for $T=\rho_N(\id_{n,m})$ and $\kappa_T=\id_{n,m},$  using \eqref{eq--CenteredCumu} for $N\ge n+m$ implies
\[\id_{n,m}=\kappa_T=\Wg_{n,m}^N*m_T= \Wg_{n,m}^N*\Omega_{n,m}^N.\]
\end{proof}

 Given the isotypic decomposition \eqref{eq---IsotTraceless} of  $\overset{\circ}{T}_{n,m},$ the $\Omega_{n,m}$ function allows to compute  dimensions of $U(N)$-irreps in terms of characters of the symmetric group.

\begin{lem} For any $n,m\ge 1$ and  any partitions $\nu\vdash n,\mu\vdash m$ with $\ell(\nu)+\ell(\mu)\le N,$
\begin{equation}
d_{[\nu,\mu]_N}=\frac{\chi^{[\nu,\mu]}(\Omega_{n,m}^N)} {n!m!}.\label{eq---DimOmega}
\end{equation}

\end{lem} 
\begin{proof} For any $\la\vdash n,\mu\vdash m,$ $\rho_N(\frac{d^{\la}d^{\mu}}{n!m!}\chi^{[\la,\mu]})$ is the equivariant projection on the $V^{[\la,\mu]}$-isotypic component of $\overset{\circ}{T}_{n,m}.$ By Koike-Schur-Weyl duality it is null  for $\ell(\la)+\ell(\mu)>N$ and isomorphic to $V^{[\la,\mu]}\ts V_{[\la,\emptyset]_N}$ when $\ell(\la)+\ell(\mu)\le N.$ In the latter case, the dimension of this component is
\[d^\la d^{\mu} d_{[\la,\mu]_N}= \frac {d^\la d^{\mu}} {n!m!}\Tr_{\overset{\circ}{T}_{n,m}} (\rho_N( \chi^{[\la,\mu]}))=\frac {d^\la d^\mu} {n!m!} \chi^{[\la,\mu]}(\Omega^N_{n,m}).\]
\end{proof}

Let us now consider some  properties of $\Omega_{n,m}^N$ as  $N\to\infty.$ 

\begin{lem} \label{Lem---OmegaF}For $N\ge n+m$ and any $\sigma\in S_{n,m},$  
\begin{equation}
\Omega_{n,m}(\sigma)=Q_\sigma(N) \label{eq---rationalOm}
\end{equation}
 is a rational function  in $N$ with $Q_\sigma\in \C(X)$ satisfying
 
\begin{equation}
Q_\sigma(X)=X^{\# \sigma}+ \sum_{g\ge 1} h^{n,m}_g(\sigma) X^{\# \sigma -2 g} \label{eq---ExpOm}
\end{equation}
where $h^{n,m}_g(\sigma)\in \Z$ for all $g\ge 1.$ 
\end{lem}
\begin{proof}  Consider $\tilde \Omega_{n+m}=N^{-n-m}\Omega_{n+m}$ and  $\tilde\Omega_{n,m}= N^{-n-m}\Omega_{n,m}.$ For $N\ge 2 (n+m)!,$ $\Omega_{n+m}$ is invertible  in $\C[S_{n+m}]$ and denoting by $[x]_{n,m}$ the projection onto $\C[S_{n,m}]$ of an element $x\in \C[S_{n+m}],$ $[\tilde\Omega^{-1}_{n,m}]_{n,m}$ is invertible in $\C[S_{n,m}].$  Moreover the following series converge then absolutely  for  the norm  $\|\cdot\|_{1,1}$  defined by \eqref{eq----NormGroupAlgS}, 
\[ \tilde\Omega_{n+m}^{-1}=1+ \sum_{k\ge 1,\sigma_1,\ldots,\sigma_k\in S_{n+m}\setminus \{\id\} } (-1)^{-1}N^{-|\sigma_1|-\ldots-|\sigma_k|}\sigma_1\ldots \sigma_k  \]
and 
\begin{equation}
 [\tilde\Omega_{n+m}^{-1}]_{n,m}^{-1}=1+ \sum_{l\ge 1}\sum_{k_1,\ldots, k_l\ge 1 }\sum_{\sigma} (-1)^{k_1+\ldots +k_l+l} \left( \prod_{1\le i\le l, 1\le j\le k_i} N^{-|\sigma_{i,j}|}\right)P_\sigma, \label{eq---OmegaHarmoPreExpan}
\end{equation}
where the third sum is over partial arrays $(\sigma_{i,j})$ of permutations  in $S_{n+m}$ with $1\le i\le l, 1\le j\le k_i$ and $\sigma_{i,1}\ldots \sigma_{i,k_i}\in S_{n,m}$ for all $i,$ and 
\[P_{\sigma}=\sigma_{1,1}\ldots\sigma_{1,k_1}\sigma_{2,1}\ldots \sigma_{l,1}\ldots \sigma_{l,k_l}\]
is the product of all its terms in lexicographic order.     Recall that for any tuple of permutation $\sigma_1,\ldots, \sigma_L\in S_{n+m},$ 
\[|\sigma_1|+\ldots+|\sigma_L|-|\sigma_1\ldots\sigma_L|\in 2\N. \]
Define\footnote{The following tuples are known as  constellation of  genus $g;$ they can be associated to a degree $n+m$, genus $g$ ramified covering of the sphere with  one labeled fiber, $|\la|+1$ labeled punctures,  the last one with $\g$ and the other with non trivial monodromy \cite{LandoZvonkin}.  }  for all $L\ge 1, g\ge 0$ and $\g\in S_{n+m},$
\begin{equation}
H_{L,g}(\g)=\{\sigma \in (S_{n+m}\setminus\{\id\})^L:  \sigma_1\ldots \sigma_L=\g \quad\text{and}\quad |\sigma_1|+\ldots+|\sigma_L|= |\g|+2g \}
\end{equation} 
and when $\g\in S_{n,m},l\ge 1$ and  $\la\in{\N^*}^{l},$
\begin{equation}
H^{n,m}_{\la,g}(\g)=\left\{\sigma \in H_{|\la|,g}(\g): \forall 1\le i\le l, \quad \sigma_{\la_{1}+\ldots +\la_{i-1}+1}\ldots \sigma_{\la_{1}+\ldots +\la_{i}}\in S_{n,m} \right\},
\end{equation}
where  $\la_0=0.$ 

 Using  \eqref{eq---HarmoWeingarten},  since the  sum \eqref{eq---OmegaHarmoPreExpan} is absolutely converging, so is the following
\begin{align}
\tilde \Omega_{n,m}(\g)&=\delta_{\g,\id}+\sum_{g\ge 0,l\ge 1,\la\in{\N^*}^{l} } (-1)^l N^{-|\g|-2 g}  \# H_{\la,g}^{n,m}(\g).\label{eq---OmegaTComposedHurwitz}
\end{align}
Multiplying by $N^{n+m},$ we conclude that for all $\g\in S_{n,m},$ for $N\ge (n+m)!,$
\begin{equation}
 \Omega_{n,m}(\g)=N^{n+m} \delta_{\g,\id}+\sum_{g\ge 0  } N^{\#\g-2 g}    h_g^{n,m}(\g) \label{eq---OmegaComposedHurwitz}
\end{equation}
with 
\begin{equation}
h_g^{n,m}(\g)=\sum_{1\le l\le |\g|+2g,\la\in{\N^*}^{l} }   (-1)^l      \# H_{\la,g}^{n,m}(\g)\in \Z.\label{eq---HurwitzOmegaP}
\end{equation}
 To prove  \eqref{eq---rationalOm} and \eqref{eq---ExpOm}, it is enough to prove that $\Omega_{n,m}(\g)$ is a rational function of $N$ with no pole in $[0,n+m-1]$ and that
 \begin{equation}
 \delta_{\g,\id}+   h_0^{n,m}(\g)=1. \label{eq---CompoHGenZ}
 \end{equation}
Performing the same computation as above simply with $\tilde\Omega_{n+m}^{-1}$ in place of $[\tilde\Omega_{n+m}^{-1}]_{n,m},$ we find that for any $\g\in S_{n+m}$
for $N\ge (n+m)!,$ 
\begin{equation}
 \Omega_{n+m}(\g)=N^{n+m} \delta_{\g,\id}+\sum_{g\ge 0  } N^{\#\g-2 g}    h_g(\g) \label{eq---ExpWeingarten}
\end{equation}
with 
\begin{equation}
h_g(\g)=\sum_{1\le l\le |\g|+2g,\la\in{\N^*}^{l} }   (-1)^l      \# H_{|\la|,g}(\g).\label{eq---ComposedHurwitz}
\end{equation}
Since $\Omega_{n+m}(\g)=N^{\#\g},$  $\delta_{\g,\id}+h_0(\g)=1.$   Now for any $\sigma\in H_{L,0}(\sigma)$  each cycle of  $\sigma_1,\ldots,\sigma_L$  permutes terms of cycle of $\g$ and in particular when $\g\in S_{n,m},$ $\sigma_1,
\ldots, \sigma_L\in S_{n,m}.$ Therefore for all $l\ge 1, \la\in  {\N^*}^{l} ,$  $ H_{\la,g}^{n,m}(\g)= H_{|\la|,g}^{n,m}(\g),$  $h_0^{n,m}(\g)=h_0(\g)$ and \eqref{eq---CompoHGenZ} follows.  

 At last, for any $N\ge n+m$ and $\la\vdash n+m,$  setting $E_\la(X)=\prod_{\square\in \la} (X+c(\square)),$ $\frac{1}{d^\la}\chi^\la(\Omega_{n+m})=E_{\la}(N) >0$ and the branching rule from $S_{n+m}$ to $S_{n,m}$  yields

 \begin{equation}
[\chi^\la]_{n,m}= \sum_{\mu\vdash n,\nu\vdash m} c_{\mu,\nu}^{\la} \chi^{[\mu,\nu]},   
 \end{equation}
 where\footnote{Recall that $ c_{\mu,\nu}^{\la}$ are the well-known Cauchy-Littlewood coefficients.} $c_{\mu,\nu}^{\la}= \dim(\Hom_{S_{n,m}}(V^\mu\ts V^\nu, R, V^\la) ).$ 
Therefore for $N\ge n+m,$ 
\[\Omega_{n+m} =\sum_{\la \vdash n+m} E_\la(N) \frac{ d^\la \chi^\la}{(n+m)!}  \quad\text{and}\quad\Omega_{n+m}^{-1}= \sum_{\la \vdash n+m} E_\la(N) ^{-1} \frac{ d^\la \chi^\la}{(n+m)!},    \]
while 
\begin{equation}
 [\Omega_{n+m}^{-1}]_{n,m}=\sum_{\nu \vdash n,\mu \vdash m} Q_{\nu,\mu}(N) \frac{ d^\nu d^\mu \chi^{[\nu,\mu]}}{n!m!}  \quad\text{and}\quad \Omega_{n,m}= \sum_{\nu \vdash n,\mu \vdash m} Q_{\nu,\mu}(N)^{-1} \frac{ d^\nu d^\mu \chi^{[\nu,\mu]}}{n!m!},  \label{eq---}
\end{equation}
with 
\begin{equation}
Q_{\nu,\mu}(X)=\frac{n!m! }{(n+m)!}\sum_{\la\vdash n+m}  \frac{d^{\la} c_{\nu,\mu}^{\la}}{d^\nu d^\mu} E_\la(X) ^{-1}. 
\end{equation}
 
\end{proof}

{\begin{rmk} Mind that  sets in the last sum of \eqref{eq---ComposedHurwitz} over compositions are not disjoint. For any composition $\la\in {\mathbb{N}^*}^l$,  define $ \overline{H}_{\la,g}^{n,m}(\g)$ as  the set of $ \sigma\in H_{\la,g}^{n,m}(\g)$  with  no composition $\mu=(\mu_1,\ldots,\mu_{s_1+\ldots+s_l})$ such that  $\sigma\in H_{\mu,g}^{n,m}(\g)$ and 
\[\mu_{1}+\ldots \mu_{s_1}=\la_1, \ldots, \mu_{s_1+\ldots+s_{l-1}+1}+\ldots +\mu_{s_1+\ldots+s_{l}}=\la_l\]
for some  $s_1,\ldots,s_l\ge1.$ 

Since there are $2^{l-1}$ ways to merge a composition $\la\in {\mathbb{N}^*}^l$,  for $N\ge 2 (n+m)!,$ \eqref{eq---OmegaTComposedHurwitz} can be rewritten as the absolutely converging sum
\begin{equation}
\Omega_{n,m}(\g)=N^{n+m}\delta_{\g,\id}+\sum_{g\ge 0,l\ge 1,\la\in{\N^*}^{l} }   (-1)^l  N^{n+m-2g}  2^{l-1} \# \overline{H}_{\la,g}^{n,m}(\g)
\end{equation}
so that for any $\g\in S_{n,m}$
\begin{equation}
h_g^{n,m}(\g)=\sum_{l\ge 1,\la\in{\N^*}^{l} }   (-1)^l      2^{l-1} \# \overline{H}_{\la,g}^{n,m}(\g).
\end{equation}
\end{rmk}

Expressions for the dimension of irreducible dimensions as $S_{n,m}$-characters evaluation like \eqref{eq---DimOmega}, when put together with similar representations for the Casimir term $\mfc_\a$ as defined in  \eqref{def---Casimir},  yield  in turn   combinatorial expressions for Yang-Mills partition functions. 

\begin{lem} For $N\ge2 ,$  $g\in  \N, q\in \C$ with  $0<q\le  1$ and any integer $0\le k<\frac N2,$ 
\label{Lem---PreTopologicalExpansion}
\begin{equation}
\zeta_{SU(N)}(2-2g,q)= \sum_{n,m:n+m\le k} \frac{q^{n+m-\left(\frac{n-m}N\right)^2}}{n!m!}    \tau_{n,m}\left(\omega_g\Omega_{n,m}^{2-2g} q^{\frac{2 {C}_{n,m}}{N} } \right)   +\zeta_{SU(N)}^{(>k)}(s,q)\label{eq---SpecialStringZeta}
\end{equation}
and 
\begin{equation}
\zeta_{U(N)}(2-2g,q)= \sum_{n,m:n+m\le k} \frac{q^{n+m}  \theta(q,\frac{n-m}{N}) }{n!m!} \tau_{n,m}\left(\omega_g\Omega_{n,m}^{2-2g} q^{\frac{2 {C}_{n,m}}{N} } \right)   +\zeta_{U(N)}^{(>k)}(s,q) \label{eq---StringZeta}
\end{equation}
where  $ \tau_{n,m}$ stands for the trace of the regular  representation  $\C[S_{n,m}]$ defined by $\tau_{n,m}(\sigma)=\delta_{\sigma,1_{n,m}}$ for $\sigma\in S_{n,m}$ and  
\[C_{n,m}=\sum_{1\le i<j\le n+m: \, j\le n \,\text{or}\,n<i} (i\,j),\]
\[\omega_g=\sum_{a_1,b_1,\ldots ,a_g,b_g\in S_{n,m}} [a_1,b_1 ]\ldots [a_g,b_g],\]
while for any $x\in \R$
\begin{equation}
\theta(q,x)=\sum_{c\in \Z}  q^{c^2 +2 x c}.
\end{equation}

\end{lem}
\begin{proof} On the one hand, for $N\ge n+m,$ $\Omega_{n,m}$ is invertible in $\C[S_{n,m}]$ and for any $\la\vdash n,\mu\vdash m,$  for any $s\in \Z,$ since $\Omega_{n,m}$ is central, using first Schur's lemma and then \eqref{eq---DimOmega},
\begin{equation}
\frac{\chi^{[\la,\mu]}(\Omega_{n,m}^s) }{d^{[\la,\mu]}}= \left(\frac{\chi^{[\la,\mu]}(\Omega_{n,m}) }{d^{[\la,\mu]}}\right)^s= \left(\frac{n! m! d_{[\la,\mu]_N}}{ d^{[\la,\mu]}}\right)^s.\label{eq---PowerOmegaF}
\end{equation}
On the other hand, for any $n,m\ge0,$ using an orthogonal basis $(X_i)_{1\le i\le N^2}$ adapted to orthogonal decomposition $\mathfrak{u}(N)= \mathfrak{su}(N)\oplus \C \id$ for the inner product \eqref{eq---InnerProdLieA}  leads to 
\begin{equation}
 \rho_{n,m}(\Delta_{U(N)})=\sum_{i=1}^{N^2}  \rho_{n,m}(X_i)^2=   \rho_{n,m}(\Delta_{SU(N)})-\left(\frac{n-m}{N}\right)^2
\end{equation}
and 
\begin{equation}
 \rho_{n,m}(\Delta_{U(N)})=  -n-m-2 \frac{C_{n,m}}{N}+\frac 2N\sum_{i<j: \<i\,j\> \in \mathcal{D}_{n,m}} \<i\, j\>.
\end{equation}
Since $\<i\, j\>$ vanishes on $\overset{\circ}{T}_{n,m}$ whenever $i\le n<j,$
\begin{equation}
\rho_{n,m}(\Delta_{SU(N)})_{| \overset{\circ}{T}_{n,m}}=\rho_N\left(-n-m -\frac {2 C_{n,m}}N \right) \label{eq---SWSpecialU}
\end{equation}
and for any $\la \vdash n,\mu\vdash m$ with $\ell(\la)+\ell(\mu)\le N,$ using Koike-Schur-Weyl decomposition \eqref{eq---IsotTraceless} and Schur's Lemma,  identifying $V^{[\la,\mu]}\ts V_{[\la,\mu]_N}$ as a subspace of $\overset{\circ}{T}_{n,m},$ 
\begin{align}
-\mfc^*_{[\la,\mu]_N}\Id_{V^{[\la,\mu]}\ts V_{[\la,\mu]_N}}&=\rho_{n,m}(\Delta_{SU(N)})_{| V^{[\la,\mu]}\ts V_{[\la,\mu]_N}} \nonumber\\
&= \left(-n-m+  \left(\frac {n-m}N\right)^2 -\frac{\chi^{[\la,\mu]}( \frac {2 C_{n,m}}N  )}{d^\la d^\mu}\right) \Id_{V^{[\la,\mu]}\ts V_{[\la,\mu]_N}}.\label{eq---CasimirPairPart}
\end{align}
Recall  the decomposition of the regular representation $\C[S_{n,m}]$ and the Frobenius formula yield 
\begin{equation}
\tau_{n,m}=\frac{1}{n!m!}\sum_{\la\vdash n,\mu\vdash m}  d^{[\la,\mu]}\chi^{[\la,\mu]}
\end{equation}
and 
\begin{equation}
\frac{1}{(n!m!)^{2g}}  d^{[\la,\mu]}\chi^{[\la,\mu]}(\omega_{g})= (d^{[\la,\mu]})^{2-2g}.
\end{equation}
Putting this together with \eqref{eq---PowerOmegaF} and \eqref{eq---SWSpecialU}, we find for $s=2-2g,$
\begin{align*}
\sum_{\la\vdash n,\mu\vdash m } d_{[\la,\mu]_N}^{s} q^{\mfc_{[\la,\mu]_N}^*}&= \frac{q^{n+m-\left(\frac{n-m}{N}\right)^2}}{(n!m!)^2} \sum_{\la\vdash n,\mu\vdash m }    d^{[\la,\mu]}\chi^{[\la,\mu]} (\omega_g)     \frac{\chi^{[\la,\mu]} (\Omega_{n,m}^s) }{d^{[\la,\mu]}}  \frac{\chi^{[\la,\mu]} (q^{\frac{2C_{n,m}}{N} }) }{d^{[\la,\mu]}} \\
&=  \frac{q^{n+m-\left(\frac{n-m}{N}\right)^2}}{(n!m!)^2} \sum_{\la\vdash n,\mu\vdash m }    d^{[\la,\mu]}\chi^{[\la,\mu]} (\omega_g \Omega_{n,m}^s q^{\frac{2C_{n,m}}{N} })   \\
&  = \frac{q^{n+m-\left(\frac{n-m}{N}\right)^2}}{n!m!}  \tau_{n,m}\left(\omega_g \Omega_{n,m}^s q^{\frac{2C_{n,m}}{N} }\right).
\end{align*}
Minding that $\{\a\in \ZD/Z: |\a|<\frac N 2\}$ is in bijection with $\{\la,\mu\in\Yb:|\la|+|\mu|<\frac{N}{2}\}$ concludes the proof of \eqref{eq---SpecialStringZeta}.  Using the  definition \eqref{def---Casimir},   for any $c\in \Z,$ 
\begin{equation}
\mfc_{[\la,\mu]_N+c}= \mfc^*_{[\la,\mu]_N}+\left(\frac{n-m}{N}+c\right)^2
\end{equation} which implies \eqref{eq---StringZeta}.
\end{proof}
Together with Proposition \eqref{Prop---ConvWittenZeta},    the above expressions  are  suitable to prove  asymptotic  expansion of their LHS in powers of $\frac 1N$ and to give combinatorial/topological interpretations for its coefficients.   We 
shall then say that these sequences in $N$  admit a topological expansion.

The formulas of  Lemma \ref{Lem---PreTopologicalExpansion} were discovered at a formal level  by Gross and Gross and Taylor {\cite{Gross,GTI,GTII}} in view of interpreting coefficients as counting continuous maps from surfaces to  surfaces, aka 
surface-maps,  up to some 
equivalence. These papers have been revisited by many authors in the physics \cite{CMR} and    mathematics literatures \cite{LevSW,LemoineMaida,NovakYM}, in particular in the so-
called chiral regime  where  in sums over  highest weight  parametrised  by  $[\la,\mu]_N,$ only terms with  $\mu=\emptyset$ are kept.

    For the combinatorial problem,  one of the difficulty of these expressions comes 
from definition and interpretation of the group algebra element  $\Omega_{n,m}^{2-2g}$. When $g=1,$ this term disappears and   Riemann-Hurwitz formula is enough to interpret the $\frac 1N$-expansion for the Witten-Zeta functions 
in terms of Hurwitz 
numbers  counting simple ramified covering of the torus.\footnote{This is derived with complete arguments in \cite{LemoineMaida,NovakYM} in the chiral case and in \cite{LemoineMaida2} in the non-chiral case, and also appeared  e.g. in \cite{CMR}. }
 In the chiral case and general $g\ge 1,$  since the $\frac 1N$ expansion of the standard Weingarten element $\Omega^{-1}_{n,0}$ is known \cite{Collins,CollinsSn,Novak,LevSW}, it is also possible to give such expressions considering for 
 instance mixed monotone 
 Hurwitz 
 numbers \cite{NovakYM}.

 Thanks to the expressions of $\Omega_{n,m}$ proven here, it is   also possible to attempt to prove topological expansions in the non-chiral case.  The following corollary follows for instance from  the Kimura-Ramgoolam formula \eqref{eq---HarmoWeingarten} and the expansion \eqref{eq---ExpWeingarten} of the Weingarten function. For integers $n\ge 1, g\ge 1,$ denote by $H^{n}_{L,g,r}$ the set of tuples  $\sigma\in S_{n}^{L+r+2g}$ such that 
 \begin{enumerate}
 \item $\sigma_i\not= \id_n$ for $1\le i\le L,$
 \item $\sigma_i$ is a transposition for $L<i\le L+r,$  
 \item  
  \begin{equation*}
\sigma_1\ldots  \sigma_{L+r} [\sigma_{r+L+1},\sigma_{r+L+2}]\ldots [\sigma_{r+L+2g-1},\sigma_{r+L+2g}]=1.
\end{equation*}
 \end{enumerate}
This space is in bijection with the set of labeled ramified coverings of degree $n$ over  a closed surface of genus $g$ with $L+r$ ramification points, the first   $L$ being non-trivial and the $r$ others simple, up to isomorphism of labeled ramified coverings.  By Riemann-Hurwitz theorem, for any $Y\in H^{n}_{L,g,r},$
\begin{equation}
\chi(Y)= (2-2g)n-  r -\sum_{i}^L|\sigma_i|\label{eq---RH}
\end{equation}
is the Euler characteristic of the total space of the ramified covering associated to $Y$. Since the latter is a closed orientable surface, $\chi(Y)=2-2g'$ for some $g'\ge0.$   For any $g'\ge 0,$ let us introduce  the subset  $H^{n,g'}_{L,g,r}$ of tuples 
of $H^n_{L,g,r}$ 
satisfying \eqref{eq---RH} and for $\la\in \N^{2g-2},$ denote by $H^{[n,m],g'}_{\la,g,r}$ the one  of $\sigma\in H^{n+m,g'}_{|\la|,g,r}$ such that 
\begin{enumerate}
\item  $\sigma_{i}\in S_{n,m}$ for all $i>|\la|,$
\item   $\forall 1\le i\le 2g-2$, with $\la_i>0,$ 
\begin{equation*}
\sigma_{\la_{1}+\ldots +\la_{i-1}+1}\ldots \sigma_{\la_{1}+\ldots +\la_{i}}\in S_{n,m}.
\end{equation*}
\end{enumerate}
The proof of the following corollary being very similar to the one of Lemma \ref{Lem---OmegaF} is left to the Reader.

\begin{coro} For $t>0$ and integers $n,m,N\ge 0$ and  $g\ge 2$ with $N\ge n+m,$ the following expansion is absolutely converging
\begin{equation}
 \tau_{n,m}\left(\omega_g\Omega_{n,m}^{2-2g} e^{-\frac{t {C}_{n,m}}{N} } \right) =\sum_{g'\ge 0,r\ge 0,\la\in \N^{2g-2}}  \frac{(-t)^{r+\# \{i: \la_i>0\}}}{r!} \# H_{\la,g,r}^{[n,m],g'}  N^{-2g'}.
\end{equation}

 \end{coro}

 \begin{rmk}In their  papers \cite{Gross,GTI,GTII},  based on another formula\footnote{equation \eqref{eq---OmegaFunctionGT} proven in the appendix. Mind that in their notation,  their $\Omega$ is  our $\Omega_{n,m}$ divided by $N^{n+m}.$  } for $\Omega_{n,m}$,   Gross and Taylor propose a topological expansion and suggest to 
enumerate   surface maps  with genus $g$ target, using the Kneser-Edmonds factorisation of   surface maps with non-zero degree into  "pinches" and ramified covering accounting  respectively  for the $\Omega^{2-2g}_{n,m}$ 
and  $q^{\frac{2 {C}_{n,m}}N}$  terms. We will not attempt to explore this approach here. 
 \end{rmk}

\newpage
\subsection{Integration of rational characters and walled Brauer algebras} 

\label{sec---Integration walled Brauer algebra}

\subsubsection{$\varepsilon$-Walled Brauer algebra} To smoothen our presentation of Wilson loops integration, we    use a slightly more general definition of  walled Brauer algebras.

Assume $s\ge 1$ is an integer and $\varepsilon:[s]\to\{-1,1\}$. We call $\varepsilon$-walled Brauer diagram   any involution $\pi$ without fixed point  on $[s]\times \{-1,1\}$ satisfying 
\[ \varepsilon(v')\eta' =-\varepsilon(v)\eta \]
whenever $\pi(v,\eta)=(v',\eta').$

We denote their set by \index{$\mathcal{D}_{\varepsilon}$  set of walled Brauer diagrams} $\mathcal{D}_{\varepsilon}$ and  by $\mathcal{B}_\varepsilon(N)$ the vector space of formal complex  linear combinations in $\mathcal{D}_\varepsilon.$  
Like $\mathcal{B}_{n,m}(N),$ $\mathcal{B}_\varepsilon(N)$  is endowed with an algebra structure that  is moreover isomorphic to $\mathcal{B}_{n,m}(N)$ with $n=\#\varepsilon^{-1}(1)$ and $m=\#\varepsilon^{-1}(-1)$ and  is the commutant of the action of $U(N)$ on 
\begin{equation}
\mathcal{T}_{\varepsilon,N}=V^{\varepsilon_1}\otimes\ldots \otimes V^{\varepsilon_s}
\end{equation}
where  $V^{\eta}$ is $V$ for $\eta=1$ and $V^{*}$ for $\eta=-1$, considering   the action  $\rho_N: \mathcal{B}_\varepsilon(N)\to \End(\mathcal{T}_{\varepsilon,N})$  defined similarly to \eqref{def---BrauerAction}. When $\la\vdash n ,\mu\vdash m,$ we denote by $P_\varepsilon^{\la,\mu}\in\mathcal{B}_{\varepsilon}(N)$  the image of $P^{\la,\mu}\in \mathcal{B}_{n,m}(N).$ Note that since $P^{\la,\mu}$ 
commutes with $S_n\times S_m$, $P_\varepsilon^{\la,\mu}$ does not depend on the enumeration choice defining the isomorphism between $\mathcal{B}_{n,m}(N)$ and $\mathcal{B}_{\varepsilon}(N).$  For each $\pi\in\mathcal{D}_{\varepsilon}, $  define a graph $\mathcal{G}_{\pi}$ with vertices $[s]\times\{-1,1\}$ and edges  given for all pairs  $\{v,\tau(v)\} $  with $v\in V,$ or  $\{(i,1),(i,-1)\}$ with  $i\in[s].$ Each vertex of $\mathcal{G}_{\pi}$ has degree two, so that each c.c. of $\mathcal{G}_{\pi}$ can be identified with a simple cycle. We call these \emph{cycles   of $\pi.$} By definition of $\rho_N,$
\begin{equation}
\Tr_{\mathcal{T}_{\varepsilon,N}} (\rho_N(\pi))=N^{\# c.c\, \mathcal{G}_{\pi}}=N^{\# \{\text{cycles of}\,\pi\}}.
\end{equation}

In what follows, to lighten the notation,  we shall sometimes drop 
the subscript $\varepsilon$ or write $\rho_N(\pi)$ in place of $\pi$  when the context is clear. We shall also denote simply by $\id_{s}\in\mathcal{D}_\varepsilon$ the involution mapping $(v,-1)$ to $(v,1)$ for any  $v\in [s].$ 

For each $\pi\in\mathcal{D}_{\varepsilon}$,    there are as many top as bottom horizontal strings, we denote this number by $h(\pi).$  When  $h(\pi)=0$, we call $\pi$ a permutation diagram; such a  diagram can be identified with the pair of permutations acting on   $\varepsilon^{-1}(-1)$
and $\varepsilon^{-1}(1).$

\subsubsection{Transposition} When $\varepsilon':[s] \to \{-1,1\},$ denoting also by $\varepsilon'$ the multiplication map on $[s]\times \{-1,1\},$  for any $\pi\in\mathcal{D}_\varepsilon,$
we set
\[\pi^{\varepsilon'}=\varepsilon' \pi\varepsilon'\in\mathcal{D}_{\varepsilon\varepsilon'}.\]
When $\varepsilon$ is constant equal to $1,$  identifying $S_s$ with $\mathcal{D}_\varepsilon$ we also write $\sigma^{\varepsilon'}\in\mathcal{D}_{\varepsilon'}$ for any permutation $\sigma\in S_s.$

\subsubsection{Words and colored Brauer diagrams}  Assume $\mathcal{A}=\{x_1,\ldots,x_k, x^{-1}_1,\ldots,  x^{-1}_k\}$ is an alphabet with $k$ letters and their formal inverse and $\mathcal{S}$ is a finite set   indexing finitely many words in $\mathcal{A}$.  

 We call colored walled Brauer diagram the data of a Brauer diagram $\pi\in\mathcal{D}_{\varepsilon}$ for some 
$\varepsilon:[s]\to\{-1,1\},$  together with two maps $ \mathfrak{a}:[s]\to \mathcal{A}$ and $\mathfrak{b}: [s]\to \mathcal{S}$ such that for all $i,$ $\mfa (i)=x_l^{\varepsilon_i}$ for some $l.$  When  $\mathcal{S}=\{W\}$ where $W$ is a symbol for the word $\omega=x_1^{\varepsilon_1}x_2^{\varepsilon_2}\ldots x_s^{\varepsilon_s} $ in 
$\mathcal{A},$ we consider the Brauer diagram $\pi_{\omega}= (1 2 \ldots s)^{\varepsilon}\in\mathcal{D}_{\varepsilon}$ coloured with the maps $\mfa:[s]\to\mathcal{A}$ given by $\mfa(i)=x_i^{\varepsilon_i}$  and the constant map $\mfb:[s]\to\mathcal{S}.$

 In general, we shall call $\mfa$ and $\mfb$ respectively the arc and boundary colourings and define   the subset $\mathcal{D}_{\mfa,\mfb}$ of \emph{compatible diagrams} $\mu\in\mathcal{D}_\varepsilon$ requiring   
\[\mfa (v')= \mfa(v) \quad\text{and}\quad \mfb (v')=\mfb(v)\]
when $\mu(v,-1)=(v',1),$ and
\[\mfa (v')= \mfa(v)^{-1} \quad\text{and}\quad \mfb (v')=\mfb(v)^{-1}\]
when $\mu(v,1)=(v',1)$ or $\mu(v,-1)=(v',-1).$

\begin{rmk}  When $\om$ has two letters $x,y$ with $x\not\in\{y,y^{-1}\}$,  $\pi_{\om}\not\in \mathcal{D}_{\mfa,\mfb}.$
\end{rmk}
 The arc colouring determines  the sign function  $\varepsilon:[s]\to\{-1,1\}$ (we call    sign function of the colouring).

When an arc colouring $\mfa$ is fixed, there is a natural left  action of  $U(N)^k$  on $\mathcal{T}_{\varepsilon,N}$ respecting the colouring, defined by  setting for every $U\in U(N)^k,$
\begin{equation}
U^{\ts \mfa}=\rho_{V^{\varepsilon_1}}(W_1)\ts \ldots \ts \rho_{V^{\varepsilon_s}}(W_s)
\end{equation}
with $W_i= U_l$ whenever $\mfa(i)=x_{l}^{\varepsilon_i}.$

\subsubsection{Juxtaposition} When $\alpha\in \mathcal{D}_{\varepsilon}(N),\beta\in \mathcal{D}_{\varepsilon'}(N)$ for some $\varepsilon:[s]\to\{-1,1\}$, $\varepsilon':[s']\to\{-1,1\}$,  we denote by $\alpha\otimes\beta\in \mathcal{D}_{\tilde \varepsilon}(N)$  the diagram obtained by juxtaposition of $\beta$ to the right of $ \alpha,$ where $(\tilde \varepsilon_1,\ldots,\tilde\varepsilon_{s+s'})=(\varepsilon_1,\ldots,\varepsilon_s,\varepsilon'_1,\ldots,\varepsilon'_{s'})$ and extend this definition bilinearly to elements of Brauer algebras yielding a  linear map from $\mathcal{B}_{\varepsilon}(N)\otimes \mathcal{B}_{\varepsilon'}(N) $ to $\mathcal{B}_{\tilde\varepsilon}(N).$ When $\alpha,\beta$ are coloured walled Brauer diagrams with colouring $\mfa:[s]\to\mathcal{A},\mfa':[s]\to\mathcal{A},$  and $\mfb:[s]\to\mathcal{S},\mfb':[s]\to\mathcal{S}'$, we define the colouring $\tilde \mfa:[s+s']\to\mathcal{A}$ and $\tilde\mfb:[s+s']\to \mathcal{S}\cup\mathcal{S}'$ of $\alpha\otimes \beta$ similarly as $\tilde\varepsilon.$

\subsubsection{Extension} When $n,m\ge 1,$  for any diagram $\pi\in\mathcal{D}_{\varepsilon}$,  consider the diagram $\pi^{n,m}$ obtained by replacing each string with $n+m$ copies, changing  signs of the vertices in the last $m$ copies. More precisely, for all $v\in[s(n+m)],$ writing $v= r+q(n+m)$  for some $q\in[s] $ and $1\le r\le n,$ set $\tilde\varepsilon(v)=\varepsilon(q)$ when  $r\le n$ and  $\tilde\varepsilon(v)=-\varepsilon(q)$ if $n<r \le n+m,$ and for $\eta\in\{-1,1\},$ set
\[\pi^{n,m}(v,\eta)= (r+q'(n+m),\eta')\]
when  $\pi(q,\eta)= (q',\eta').$ By construction, $\pi^{n,m}\in\mathcal{D}_{\tilde\varepsilon}.$\footnote{Though we shall not use it here, mind that extending this definition linearly  yields an algebra morphism from $\mathcal{B}_{\varepsilon}(N)$ to $\mathcal{B}_{\tilde\varepsilon}(N^{n+m}).$}   When $\alpha$ is a coloured walled Brauer diagram with colouring $\mfa:[s]\to\mathcal{A}$  and $\mfb:[s]\to\mathcal{S}$, we define the colouring $\tilde \mfa:[s+s']\to\mathcal{A}$ and $\tilde\mfb:[s+s']\to \mathcal{S}\cup\mathcal{S}^{-1}$ of $\alpha$ similarly to $\tilde\varepsilon:$    for all $v\in[s(n+m)],$ writing $v= r+q(n+m)$  for some $q\in[s] $ and $1\le r\le n,$ set 
\[\tilde\mfa (v)=\mfa(v) \quad\text{and}\quad \tilde\mfb (v)=\mfb(v)\] when  $r\le n,$ and 
 \[\tilde\mfa (v)=\mfa(v)^{-1} \quad\text{and}\quad \tilde\mfb (v)=\mfb(v)^{-1}\] when $n<r \le n+m.$ 

Recall that  the\index{$|\om|$ Length of a word $\om$} length $|\om|$ of a word  $\om$ is its number of letters and the definition \eqref{eq---EquiProjIrrepSym} for $S_{n,m}$-equivariant projections. With the above notations, \eqref{eq---CharacterTraceless} can be rewritten as follows.

 \begin{lem}\label{lem---characterTraceHarmo}  For $\la\vdash n,\mu\vdash m$ and any word $r$ in $\mathcal{A}$ with associated sign function $\varepsilon:[|r|]\to\{-1,1\}$ and arc-colouring $\mfa:[|r|]\to \mathcal{A}$,  for $U_1,\ldots,U_k\in U(N),$
 \[d^{\la}d^{\mu}\chi_{[\la,\mu]_N}(r(U_1,\ldots,U_k))=\Tr_{\mathcal{T}_{\tilde\varepsilon}}((P^{\la,\mu}_N)^{\otimes |r|} \pi_{r}^{n,m} U^{\otimes \tilde\mfa}), \]
 where $\tilde\varepsilon:[|r|(n+m)]\to \{-1,1\},\tilde\mfa:[|r|(n+m)]\to \mathcal{A}$  are the extension of the sign function and arc colouring  of $r.$  \end{lem}
 
 \subsubsection{Second moment diagram for one relator}\label{Sec---Second moment diagram for one relator} Let us now specify the main type of diagram we shall work with.  
 
When $\omega, r$ are two words in $\mathcal{A}$, with indexing set $\{W,R\}$, we consider the Brauer diagram
\[\pi^{n,m}_{r,\omega}=\pi_r^{n,m}\otimes \pi_\omega\otimes \pi_{\omega^{-1}}\]
and denote its colourings and sign function defined by juxtaposition    
 $\mfa: [s]\to \mathcal{A}, \mfb:[s]\to \{R,R^{-1},W,W^{-1}\} $ and  $\varepsilon:[s]\to\{-1,1\},$ where
  \[s=|r|(n+m)+2|\omega|.\]
The following is  elementary to check from Lemma \ref{lem---characterTraceHarmo}.
 \begin{lem} \label{lem---TraceWordCharacterRelatorForm} For $\la\vdash n,\mu\vdash m$ and any word $r,\om$ in $\mathcal{A}$ and $U_1,\ldots,U_k\in U(N),$
\begin{align*}
 d^\la d^\mu|\Tr(\omega(U_i))|^2\chi_{[\la,\mu]_N}(r(U_1,\ldots,U_k))&=\Tr_{\mathcal{T}_{\varepsilon,N}}(Q^{\la,\mu}\pi^{n,m}_{r,\omega} U^{\otimes\mfa})\\
&=\Tr_{\mathcal{T}_{\varepsilon,N}}(Q^{\la,\mu} \pi^{n,m}_{r,\omega}Q^{\la,\mu} U^{\otimes  \mfa}),
\end{align*}
 where  $ \varepsilon:[s]\to \{-1,1\}, \mfa:[s]\to \mathcal{A}$  are the juxtaposition of the sign function and arc colouring of $\om$ and $\om^{-1}$,  with the extension of the ones of  $r,$ while
 \begin{equation}
Q^{\la,\mu}=(P_N^{\la,\mu})^{\otimes |r|}\otimes \id_{2|\omega|}. \label{eq---ExtensionTensHarmoProj}
 \end{equation} 
 When $A,B\in U(N),$ 
\begin{align}
d^\la d^\mu &|\Tr(\omega(U_i)A)|^2\chi_{[\la,\mu]_N}(r(U_1,\ldots,U_k)B)\nonumber\\ 
&=\Tr_{\mathcal{T}_{\varepsilon,N}}\left[Q^{\la,\mu}\pi^{n,m}_{r,\omega} U^{\otimes \mfa}W^{\otimes\varepsilon}\right]\\
&=\Tr_{\mathcal{T}_{\varepsilon,N}}\left[Q^{\la,\mu}\pi^{n,m}_{r,\omega} Q^{\la,\mu}U^{\otimes \mfa}W^{\otimes\varepsilon}\right]
\end{align}
  { for some $W_1,\ldots,W_s \in U(N)$ with $W^{\ts \varepsilon} Q^{\la,\mu}=Q^{\la,\mu} W^{\ts \varepsilon}. $} 
 \end{lem}

 \subsubsection{Unitary integration and invariant ideal} \label{sec---Unitary integration and invariant ideal} When the sign function satisfies $\varepsilon:[s]\to \{-1,1\} $  and  $\#\varepsilon^{-1}(\{1\})=\#\varepsilon^{-1}(\{-1\}),$  denote by $S_\varepsilon$  the set of bijections from $ \varepsilon^{-1}(\{-1\})$ to $\varepsilon^{-1}(\{1\}).$   When $\mfa:[s]\to \mathcal{A}$ is an arc colouring,  $S_\mfa$ denotes the subset of bijections $\a\in S_{\varepsilon}$ with
 \[ \mfa(\a(v))=\mfa(v)^{-1}\]
 for all $v\in [s].$ 
 
   For any pair of 
 permutations $\alpha,\beta\in S_\varepsilon,$ we define an element $\pi_{\a,\beta}\in \mathcal{D}_\varepsilon$ setting
{
 
\[\pi_{\a,\b}(v,-1)=(\b(v),-1) \quad \text{and}\quad \pi_{\a,\b}(v,1)=(\a(v),1)\]  }
 for any $v\in\varepsilon^{-1}(\{-1\}).$ The standard unitary  Weingarten formula \eqref{eq---WeingartenFormulaUnit} and bounds on the Weingarten function (see e.g. \cite[Cor. 2.4. and Prop. 2.6.]{CollinsSn}) reads as follows in our notations. When $\sigma\in S_n$ is a permutation denote by $\# \sigma$ its number of cycles. 
 
 \begin{lem}[\cite{CollinsSn}] \label{lem---ColorWeing} For any function $\mfa:[s]\to\mathcal{A},$ when $U_1,\ldots,U_k$ are $k$ independent Haar unitary matrices,
 \begin{equation}
 \EE[U^{\ts \mfa}]=\sum_{\a,\b\in S_{\mfa}}W_{N,\mfa}(\a,\b)\pi_{\a,\b}\label{eq---ColorWeing}
 \end{equation}
where $W_{N,\mfa}(\a,\b)= O(N^{-{2}s+\# \a^{-1}\b})$ for all $\a,\b\in S_{\mfa}.$
 
 \end{lem}

  When $\mfb:[s]\to \{R,R^{-1},W,W^{-1}\}$ is a boundary colouring, $S_{\mfa,\mfb}$ is the subset of permutations $\a\in S_\mfa$ with\footnote{In other words, for a  permutation of $S_{\mfa,\mfb}$,  two matched points  of $[s]$    have inverse $\mfa$-label, but a point with $R$  $\mfb$-label cannot be matched to  a $R^{-1}$ point. }
\begin{equation}
{ \mfb(\a(v))\not=R^{-\varepsilon} \label{eq---AdmissibilityPerm}}
\end{equation}
 
 for any $\varepsilon\in\{1,-1\}$ and $v\in \mfb^{-1}(R^\varepsilon).$

With these notations, the second identity of \eqref{eq---RelationBrauerHarmo} implies the following. 

 \begin{lem} \label{lem---ProjectionIrrepNoMixing}For any ${ n,m\ge1}, \la\vdash n,\mu\vdash m,$   $\a,\b\in S_{\varepsilon}$ with $\a\not\in  S_{\mfa,\mfb},$
 \[(P_N^{\la,\mu})^{\otimes |r|}\otimes \id_{2|\omega|}. \pi_{\a,\b} = \pi_{\b,\a} .(P^{\la,\mu}_N)^{\otimes |r|}\otimes \id_{2|\omega|}=0.  \]
 \end{lem}

We can now deduce our main integral representation formula in terms of Brauer algebra using Corollary \ref{lem---ProjMixedTIrrep}.

\begin{coro} \label{Coro---BrauerRepWLE}  For $\la\vdash n,\mu\vdash m,$  any word $r,\om$ in $\mathcal{A}$ and $A,B\in U(N),$

\begin{align*}
d^\lambda d^\mu& \EE_U[ |\Tr(\omega(U_i)A)|^2\chi_{\la,\mu}(r(U_1,\ldots,U_k)B)]=\\
&\sum_{\substack{\tau_{u},\tau_b\in \mathcal{D}_{n,m}^{|r|}\\\a,\b\in S_{\mfa,\mfb}}} \kappa^{\lambda,\mu}_N( \tau_{u},\tau_b,\pi_{\a,\b}) \Tr_{\mathcal{T}_{\varepsilon,N}}\left[x_u \pi^{n,m}_{r,\omega}x_b \pi_{\a,\b} W^{\otimes \varepsilon}\right]
\end{align*}
 { for some $W_1,\ldots,W_s \in U(N),$}  where
 \[x_l=\tau_{l,1} \ts \ldots \ts \tau_{l,|r|} \otimes \id_{2|\omega|} \quad\text{for}\quad l\in\{u,b\},\]
  \begin{equation}
\kappa^{\lambda,\mu}_N(\tau_u,\tau_b,\pi_{\a,\b})= W_{N,\mfa}(\a,\b) \prod_{i=1}^{|r|}\left(\kappa_{N}^{\lambda,\mu}(\tau_{u,i})\kappa_{N}^{\lambda,\mu}(\tau_{b,i})\right) \label{eq---WeingIrrep}
\end{equation}
and $\varepsilon,\mfa,\mfb$ are the sign function and colourings associated to $\pi_{r,\omega}^{n,m}.$ 
\end{coro}
\begin{proof}  Thanks to Lemmas \ref{lem---TraceWordCharacterRelatorForm}, using \eqref{eq---ColorWeing} and that $W^{\ts\varepsilon}$  commutes with $Q^{\la,\mu},$ the left-hand side is
\[  \sum_{\a,\b\in S_{\mfa}} W_{N,\mfa}(\a,\b)\Tr_{\mathcal{T}_{\varepsilon,N}}\left[ \pi^{n,m}_{r,\omega}  Q^{\la,\mu} \pi_{\a,\b}Q^{\la,\mu}  W^{\otimes \varepsilon}\right]. \]
By Lemma \ref{lem---ProjectionIrrepNoMixing},  $Q^{\la,\mu} \pi_{\a,\b}Q^{\la,\mu} =0$ whenever  $\a \not\in S_{\mfa,\mfb}$ or $\b\not\in S_{\mfa,\mfb}.$  Since  $Q^{\la,\mu}W^{\ts\varepsilon}=W^{\ts\varepsilon}Q^{\la,\mu},$ the claim follows from expanding \eqref{eq---ExtensionTensHarmoProj} using Corollary \ref{lem---ProjMixedTIrrep}.

\end{proof}

We shall call  diagrams of the form $\pi_{\a,\b}$ with $\a,\b\in S_{\varepsilon}$ \emph{horizontals}.  When  $\a=\b,$ we call  them \emph{quadrangles-diagrams}. When  $\a,\b\in S_{\mfa,\mfb}$ we say  $\pi_{\a,\b}$ is  \emph{admissible.}  Denote the set of horizontal diagrams by $\mathcal{D}^h_\varepsilon$ and the subset of admissible ones by $\mathcal{D}^h_{\mfa,\mfb}.$  When $\a\in S_n$ is a permutation, recall  
\[|\sigma|=n-\#\sigma,\]
where $\#\sigma$ is the number of cycles of $\sigma,$ is the smallest $k$ such that $\sigma$ is a product of $k$ transpositions. When $\varphi=\pi_{\a,\b},$ we set 
\[\#\varphi= \#\alpha^{-1}\b\quad\text{and}\quad|\varphi|=|\a^{-1}\b|\]
so that $|\varphi|=0$ if and only if $\varphi$ is a quadrangle-diagram. \label{SubSec--HorizontalDiag}
%\color{orange} 
\begin{rmk}
{Note that horizontal diagrams span an ideal  $\mathcal{I}_\varepsilon$ of $\mathcal{B}_\varepsilon(N).$  When $n=m,$ $\mathcal{I}_{\varepsilon}\not=\{0\}$ and when  $2n\le N$, it follows from Koike-Schur-Weyl duality \cite{HalversonWBrauer}  that  $\mathcal{B}_{\varepsilon}(N)$ is semi-simple while  $\mathcal{I}_\varepsilon$,  as a sub-algebra, is simple and isomorphic to 
$ \mathrm{End}(V^{[\emptyset,\emptyset]})$ where $V^{[\emptyset,\emptyset]}=\Hom_{U(N)}(V_{\rho_N}, \mathcal{T}_{n,m}).$
}
\end{rmk}

\subsection{Surfaces  associated to products of Brauer diagrams}

Let us call \emph{Brauer map}    the data of a tuple $\mfm=(\tau_u,\pi,\tau_b,\varphi)$ formed of a coloured walled Brauer diagram $\pi\in\mathcal{D}_\varepsilon$  together with a horizontal  $\varphi\in\mathcal{D}^h_\varepsilon$ and 
two {compatible} diagrams $\tau_u,\tau_b\in \mathcal{D}_{\mfa,\mfb},$ where  $\mfa:[s]\to \mathcal{A}$ and $\mfb:[s]\to \mathcal{S}$ are the arc-boundary colourings and  $\varepsilon:[s]\to\{-1,1\}$ is 
the sign function of $\mfa.$ For such a tuple, consider the $2$-CW complex built in the following way.  Informally,  the complex is built from  the graph appearing in the product 
$\tau_u\pi\tau_b\varphi$ in $\mathcal{B}_{\varepsilon}$ (including possible inner loops) by  adding vertical edges between top   and bottom vertices of $\varphi,$  idenfitying top and bottom 
vertices of $\tau_u\pi\tau_b\varphi$ and  gluing  discs  along cycles associated to the trace of $\tau_u\pi\tau_b\varphi$ and  $\varphi$ in $\mathcal{B}_\varepsilon.$  See Figure \ref{FIG---DefBrauerMap} for an example.

\begin{figure} 
\includegraphics[width=6cm]{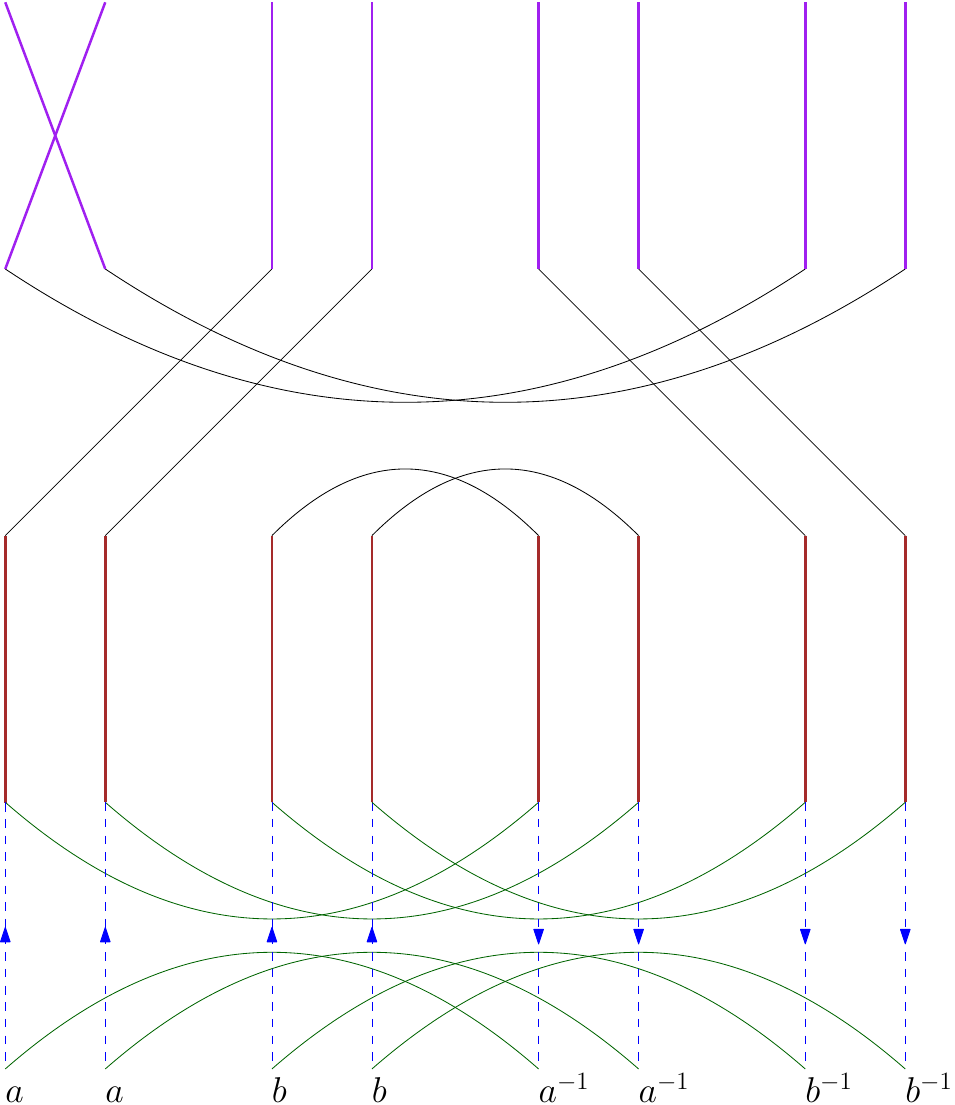}
 \qquad \includegraphics[width=6cm,height=7 cm]{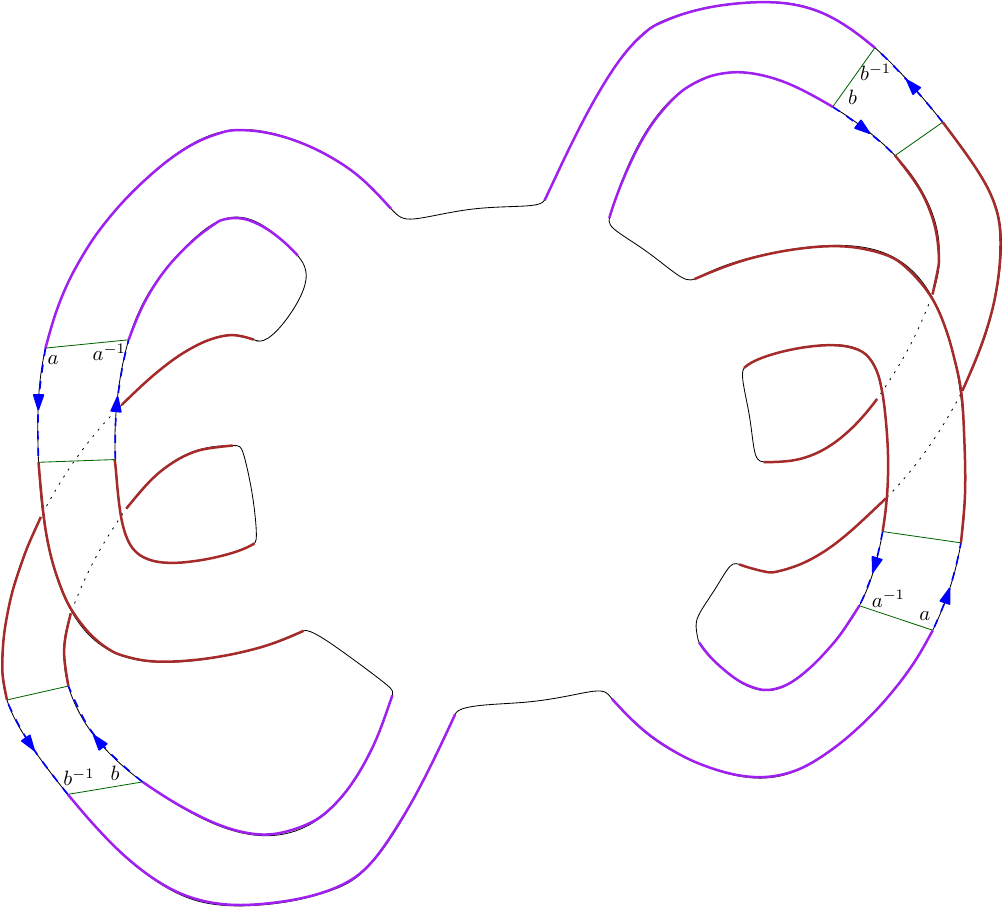}
\caption{A Brauer map\label{FIG---DefBrauerMap} and its associated surface}
\end{figure}

More precisely, set $V=[s]\times\{-2,-1,0,1,2\}.$   Consider  involutions  $(\tau_u',\pi',\tau'_b,\varphi')$ of $[s]\times \{1,2\},$ $[s]\times \{0,1\},$ $[s]\times \{-1,0\}$  and $[s]\times \{-2,-1\}$ 
obtained from $(\tau_u,\pi,\tau_b,\varphi)$ by conjugating them respectively by the bijections acting only the second coordinate and mapping $(-1,1)$ to 
$(1,2),(0,1),(-1,0)$ and $(-2,-1)$.    For each pairing of $(\tau_u',\pi',\tau'_b,\varphi')$, attach  an interval between its matched points; call the resulting interval \emph{$\tau,\pi$ or $\varphi$-interval} according to the involution used.  Let $V'$ be the quotient of $V$ identifying $(i,2)$ with $(i,-2)$ for any $i\in [s].$ Since each point of $V'$ belongs to exactly two such 
pairings, the resulting $1$-skeleton is a disjoint union {$\mathcal{C}_I$ of circles}. Attach additionally 
intervals between projections of  $(x,-2)$ and $(x,-1)$ for each $x\in[s]$ to yield a $1$-skeleton $\mathcal{G}_\mfm.$   Call this last type of intervals, \emph{$\iota$-intervals}.  Orient an $\iota$-interval upwards when its first coordinate $s$ satisfies $\varepsilon(s)=1$ and downwards otherwise.  Mark in both cases its source vertex by the letter $\mfa(s).$ Since  $\tau_u,\tau_b,\pi,\varphi\in \mathcal{D}_{\varepsilon}$ the orientation of $\iota$ induces an orientation of $\mathcal{C}_I$-circle.

Concatenating alternatively intervals associated to   $\varphi'$ and $\iota$ yields a disjoint \emph{union $\mathcal{C}_{II}$ of circles} in $\mathcal{G}_\mfm$. Since $\varphi\in \mathcal{D}_{\varepsilon},$ $\iota$ induces also an orientation of 
$\mathcal{C}_{II}$ circles.  For each connected component of 
$\mathcal{C}_I$ or of $\mathcal{C}_{II}$, attach an oriented disc gluing its positively oriented boundary along the given circle with  same   orientation as the one of $\iota$.  
We denote the resulting $2$-CW complex by $X_\mfm.$  { Since  exactly  two or three intervals are attached  to each point of $V'$, 
whereas  exactly one or two discs are attached along each interval, the total space of $X_\mfm$ 
is a surface. We denote it by $\Sigma_\mfm.$}  

\begin{rmk}[Orientation]\label{rmk---OrientBrauer} Orient each strand of  a Brauer diagram $\nu\in \mathcal{D}_\varepsilon$ from $v$ to $v'$ when $v=(s,\eta),v'=(s',\eta')$ with  $\varepsilon_s\eta=-1,\varepsilon'_s\eta'=1$.  It induces an orientation of its cycles. When $\nu$ is an element of a Brauer map $(\tau_u,\pi,\tau_b,\varphi)$, the orientation of $\nu$  is the same as the one induced by $\mathcal{C}_{II}$-circles when $\nu=\varphi$ and by $\mathcal{C}_I$-circles otherwise.
\end{rmk}

\begin{rmk} \label{Rmk---FaceDiagram}  When $\varphi=\pi_{\a,\b},$  connected components of $\mathcal{C}_{II}$ are in bijection with cycles of $\a^{-1}\b$ so that $\#\varphi$ is the number of connected  components of 
$\mathcal{C}_{II}.$ Moreover, for a given connected component,  its number of $\varphi$-intervals is twice the size of the associated cycle of $\a^{-1}\b$.  Orienting $\varphi$-intervals as in remark \ref{rmk---OrientBrauer}, top and bottom strands have  opposite orientation so that  the orientation of $\varphi$-intervals  alternates along any $\mathcal{C}_{II}$-circle.   \end{rmk}

\begin{rmk}[Boundary]

Each inner point of a $\varphi$-interval  belongs to exactly one $\mathcal{C}_I$-circle and one $\mathcal{C}_{II}$-circle, so that the image of an open $\varphi$-interval in $X_{\tau_u,\pi,\tau_b,\varphi}$ is included in 
exactly two $2$-cells and belongs to the interior of $\Sigma_\mfm$.  For any other interval, inner points belong to exactly one circle and hence to the boundary of  $\Sigma_\mfm$.  Lastly each point of $V'$ 
is exactly in two intervals of type $\iota,\pi$ or $\tau.$ The union of images of such intervals is a disjoint union of simple loops that is 
exactly the boundary $\partial\Sigma_\mfm.$   \label{Rmk---InnerORBound}
\end{rmk}

We therefore  call $\varphi$-intervals \emph{inner intervals}, and $\iota,\tau$ or $\pi$  intervals
\emph{boundary intervals}.

\begin{rmk} Since  the vertices  image of $[s]\times\{-2,-1\}$ in $V'$ have degree three, while all others have degree $2,$  
\begin{equation}
\chi(\mathcal{G}_\mfm)=-s.
\end{equation}
Together with the previous remark,
\begin{equation}
\chi(\Sigma_\mfm)=-s+\# \text{c.c.}\, \mathcal{C}_I+\#\alpha^{-1}\beta.
\end{equation}
\end{rmk}
By definition of the Brauer algebra and of $\rho_N,$   for any Brauer map $\mathfrak{m}=(\tau_u,\pi,\tau_b,\varphi),$ with sign function $\varepsilon:[s]\to \{-1,1\},$
\begin{equation}
\Tr_{\mathcal{T}_{\varepsilon,N}}(\rho_{N}(\tau_u\pi\tau_b \varphi))=N^{\# c.c.\, \mathcal{C}_I}. \label{eq---TraceBM}
\end{equation}
Putting this together with formulas of the previous subsection gives the following upper-bound.  For  a Brauer map $\mfm=(\tau_u,\pi,\tau_b,\varphi)$,  set 
\begin{equation}
h(\mfm)=h(\tau_u)+h(\tau_b).
\end{equation}

\begin{lem} \label{lem---BoundTraceWordCharactRel Euler} Assume $G_N$ is $U(N)$ or $SU(N).$ For $\la\vdash n,\mu\vdash m$ and  any words $r,\om$ in $\mathcal{A}$, define $\pi_{r,\omega}^{n,m}$  and  its  colouring and sign functions $\mfa,\mfb,\varepsilon$ as in section \ref{sec---Unitary integration and invariant ideal}. Then there is  
a constant  $K>0$ depending only on $n,m,r$ and $\omega,$ such that for all $N\ge1,$ $U_1,\ldots,U_k,A,B\in G_N,$
\begin{equation}
| \EE_U[ |\Tr(\omega(U_i)A)|^2\chi_{\la,\mu}(r(U_1,\ldots,U_k)B)]|\le K_{n,m}  N^{\max_{\mfm\in\mathcal{M}_{n,m,r,\omega}} \chi(\Sigma_{\mfm})-h(\mfm)}, \label{eq---BoundSquareTraceCharac}
\end{equation}
where $\mathcal{M}_{n,m,r,\omega}$ is the set of $(\varepsilon,\mfa,\mfb)$-admissible Brauer maps $\mfm=(\tau_u,\pi,\tau_b,\varphi)$ with $\pi=\pi_{r,\omega}^{n,m}, \varphi=\pi_{\a,\b}$ for some $\a,\b\in S_{\mfa,\mfb}$ and   $\tau_u,\tau_b$ are of the form $\tau_1\otimes \ldots \otimes \tau_r\otimes \id_{2|\omega|}$ for some $\tau_i\in\mathcal{D}_{n,m}.$ 
\end{lem}

\begin{proof}  When  $S,U,\xi$ are Haar distributed random variables on respectively $SU(N),$  $U(N)$ and $U(1),$  $U$ has same law $D_\xi S $ where $D_\xi$ is the diagonal matrix with diagonal coefficient $\xi.$ In particular when $A,B\in SU(N),$ the expectation in the left-hand side of \eqref{eq---BoundSquareTraceCharac} for $G_N=U(N)$ is the same as the one for $G_N=SU(N)$. It is therefore enough to prove the claim for $G_N=U(N).$

Recall \eqref{eq---WeingIrrep} and for $ \a,\b\in S_\varepsilon,\tau_u,\tau_b\in \mathcal{D}_\varepsilon,$ set \begin{equation}
\kappa^{\lambda,\mu}_N(\tau_u,\tau_b,\pi_{\a,\b}) =  \kappa^{\lambda,\mu}_N(Z_u,Z_b,\pi_{\a,\b})  
\end{equation}
if  for any $l\in\{u,b\},$  $\tau_l=Z_{u,i}\otimes\ldots\otimes Z_{l,|r|}\otimes \id_{2|\omega|}$ for some $Z_{l}\in\mathcal{D}_{\varepsilon}^{|r|}$ and $\kappa^{\lambda,\mu}_N(\tau_u,\tau_b,\pi_{\a,\b}) =0$ otherwise. Putting together Corollary \ref{Coro---BrauerRepWLE} and \eqref{eq---TraceBM},
\begin{equation}
\EE_U[ |\Tr(\omega(U_i)A)|^2\chi_{\la,\mu}(r(U_1,\ldots,U_k)B)]=\sum_{(\tau_u,\pi,\tau_b,\varphi)\in \mathcal{M}_{n,m,r,\omega}}  \kappa^{\lambda,\mu}_N(\tau_u,\tau_b, \varphi)N^{\# \text{c.c.}\,\mathcal{C}_I} \mathrm{tr}_{I}(A,B)
\end{equation}
where 
\[ \mathrm{tr}_{I}(A,B)= N^{-\#\text{c.c.}\, \mathcal{C}_I}\Tr_{\mathcal{T}_{\varepsilon,N}}(\tau_u\pi\tau_b\varphi V^{\otimes \varepsilon}) \]
and $V_1,\ldots ,V_{|r|+2|\omega|}\in U(N)$ are as in Corollary \ref{Coro---BrauerRepWLE}.    Since $\Tr_{\mathcal{T}_{\varepsilon,N}}(\tau_u\pi\tau_b\varphi V^{\otimes \varepsilon})$ is a product of $\#\mathcal{C}_I$ traces in words in $V_1,\ldots,V_s$ and their inverse,\footnote{More explicitely, for each $\mathcal{C}_I$-cycle $c$, consider the unitary matrix $h_c\in U(N)$ formed by multiplying resp. by  $B, B^{-1},A$ or $A^{-1}$ whenever $c$ visits $[n]\times \{1\},\{n+1,\ldots,n+m\}\times \{2\}, \{|r|(n+m)+1\}\times \{2\}$ or $\{|r|(n+m) +|\omega|+1\}\times \{2\}.$ Then \[ \mathrm{tr}_I(A,B)=\prod_{c\,\,\mathcal{C}_I-\text{cycle}} \frac 1 N\Tr(h_c).\]   } $| \mathrm{tr}_{I}(A,B)|\le 1.$ Now thanks to \eqref{lem---ColorWeing} and Lemma \ref{lem---ProjMixedTIrrep}, there is $K>0$ such that for any Brauer map  $(\tau_u,\pi,\tau_b,\varphi)\in \mathcal{M}_{n,m,r,\omega},$  
\begin{equation}
| \kappa^{\lambda,\mu}_N(\tau_u,\tau_b, \varphi)| N^{\#c.c.\,\mathcal{C}_I}\le  K N^{-s+\# c.c. \,\mathcal{C}_{II}+\# c.c.\,\mathcal{C}_{I}}=K N^{\chi(\Sigma_{\tau_u,\pi,\tau_b,\varphi})}. 
\end{equation}
\end{proof}

\begin{thm} \label{thm---BoundTraceWordCharactRel Size Irrep}Assume $\omega$ is a cyclically shortest length representative of an element of $\G_g\setminus\{1\}$ and $r$ is the surface word \eqref{eq---SurfaceWord}.  Assume $G_N$ is $U(N)$ or $SU(N).$ For any $k\ge1,$ there is $K_{k}>0$ such that for all $\a\in\mathbb{Z}^{\downarrow N}$ with $|\a|\le k,$ and  all $N\ge1,$ $A,B\in G_N,$
\[  | \EE_U[ |\Tr(\omega(U_1,\ldots,U_{2g})A)|^2\chi_{\a}(r(U_1,\ldots,U_{2g})B)]|\le  K_{k} N^{-|\a|}\]
where $U_1,\ldots,U_{2g}$ are independent and Haar distributed on $G_N.$
\end{thm}
\begin{proof}  For all $N\ge1,$ $U_1,\ldots,U_{2g},A,B\in G_N,$ $\la\vdash n,\mu\vdash m,$ and $c\in \mathbb Z,$ since $r(U_1,\ldots,U_{2g})\in SU(N),$
\[\chi_{[\la,\mu]_N+c1_N}(r(U_1,\ldots,U_{2g})B)= \det(B)^{c} \chi_{[\la,\mu]_N}(r(U_1,\ldots,U_{2g})B).\]
It is therefore enough to prove the bound for $\a=[\la,\mu]_N$ for  fixed  $\la\vdash n,\mu\vdash m.$    We can then apply Lemma \ref{lem---BoundTraceWordCharactRel Euler} and the claim  follows from Proposition  \ref{Prop---GeoBoundEuler} 
below.
\end{proof}

Following \cite{MageeII}, two types of upper-bounds on the Euler characteristic $\chi(\Sigma_\mfm)$ lead to the proof of Proposition  \ref{Prop---GeoBoundEuler}.

\subsubsection{Upper-bound I:  reducing $\tau$-diagrams to permutations and  $\varphi$-diagrams to quadrangles.}
 
 Let us argue that we can focus on a smaller set of Brauer maps.  The arguments are here relatively standard and variations of  the arguments in \cite[Sec. 4.2.]{MageeII}  
 adapted to our notations.

\begin{lem}  \label{Lem---CutORCompress} For each  Brauer map $\mfm=(\tau_u,\pi,\tau_b,\varphi)$ with arc-boundary colouring $\mfa,\mfb,$ there  are compatible permutation diagram $\tau_u',\tau_b'\in\mathcal{D}_{\mfa,\mfb}$ such that  
\[|\chi(\Sigma_{\tau_u',\pi,\tau_b',\varphi})-\chi(\Sigma_{\mfm})|\le h(\mfm).\]
\end{lem}
\begin{proof}  Consider a Brauer map $\mfm=(\tau_u,\pi,\tau_b,\varphi)$ with   arc-boundary colouring $\mfa:[s]\to\mathcal{A},\mfb:[s]\to\mathcal{S}$ and $h(\mfm)>0.$ Assume w.l.o.g. $h(\tau_u)>0.$ Consider one bottom and one top horizontal strand of $\tau_u$ {in the same cycle of 
$\tau_u.$} Denote their endpoints by  $v_b,v_b'\in [s]\times \{-1\}$ and $v_t,v_t'\in[s]\times \{1\}$  so that $\tau_u(v_b)=v_b',$  $\tau_u(v_t)=v'_t$ and $v'_bv_bv_tv'_t$  are cyclically ordered by the latter $\tau_u$-cycle. Since $\tau_u$ is compatible with $\mfa,\mfb$ {and that the two strands are in the same $\tau_u$-cycle}, we can assume
\begin{equation}
\mfa(v_t')=\mfa(v_b')=\mfa(v_b)^{-1}=\mfa(v_t)^{-1} \quad\text{and}\quad \mfb(v_t')=\mfb(v_b')=\mfb(v_b)^{-1}=\mfb(v_t)^{-1} .\label{Eq---CompRedPerm}
\end{equation}

 Consider the involution $\tau_u'$ of $[s]\times\{-1,1\}$ that agrees with $\tau_u$ away from $v_b,v_t,v'_b,v'_t$ and with
 \[{ \tau_u'(v'_b)=v'_t \quad\text{and}\quad \tau_u'(v_b)=v_t.}\]
 
Thanks to  \eqref{Eq---CompRedPerm}, $\tau_u'\in\mathcal{D}_{\mfa,\mfb}$ and by construction  $h(\tau_u')=h(\tau_u)-1.$  Setting $\mfm'=(\tau'_u,\pi,\tau_b,\varphi),$ it is enough to show  
\[|\chi(\Sigma_{\mfm'})- \chi(\Sigma_\mfm)|\le 1.\]
 
Denote  $I'_1,I_2'$ the intervals of $X_{\mfm'}$ with respective endpoints $(v'_b,v'_t)$ and $(v_b,v_t).$  All cells of $X_{\mfm}$ but $I_1,I_2$ can be identified  with a cell of   $X_{\mfm'}$ but $I'_1,$ $I'_2,$ with the exception of faces bounded respectively by $I_1$ or $I_2,$ and  by $I'_1$ or $I'_2$.   

with the exception of faces bounded respectively by $I_1$ or $I_2,$ and  by $I'_1$ or $I'_2$.

Either  $I_1,I_2$ belong to the same $\mathcal{C}_I$-circle  $I_2p_2 I_1p_1$  of $\mathcal{G}_\mfm$ and replacing them with $I'_1,I'_2$ cuts this cycle into the two $\mathcal{C}_I$-circles $I'_2p_2$ and $I'_1p_1$ of $\mathcal{G}_{\mfm'}$. Or vice-versa  they belong to two different $\mathcal{C}_I$-circles  $I_1p_1$ and $ I_2p_2$  and  replacing them with $I'_1,I'_2$ merges these  $\mathcal{C}_I$-circles of $\mathcal{G}_\mfm$ into the $\mathcal{C}_I$-circle  $I'_2p_2 I'_1p_1$ of $\mathcal{G}_{\mfm'}.$ In both cases, only the number of faces changes by $\pm 1$, so $\chi(\Sigma_{\mfm'})-\chi(\Sigma_{\mfm})$ is respectively  $1$ and $-1$.

\end{proof}

\begin{figure} 
\includegraphics[width=6cm]{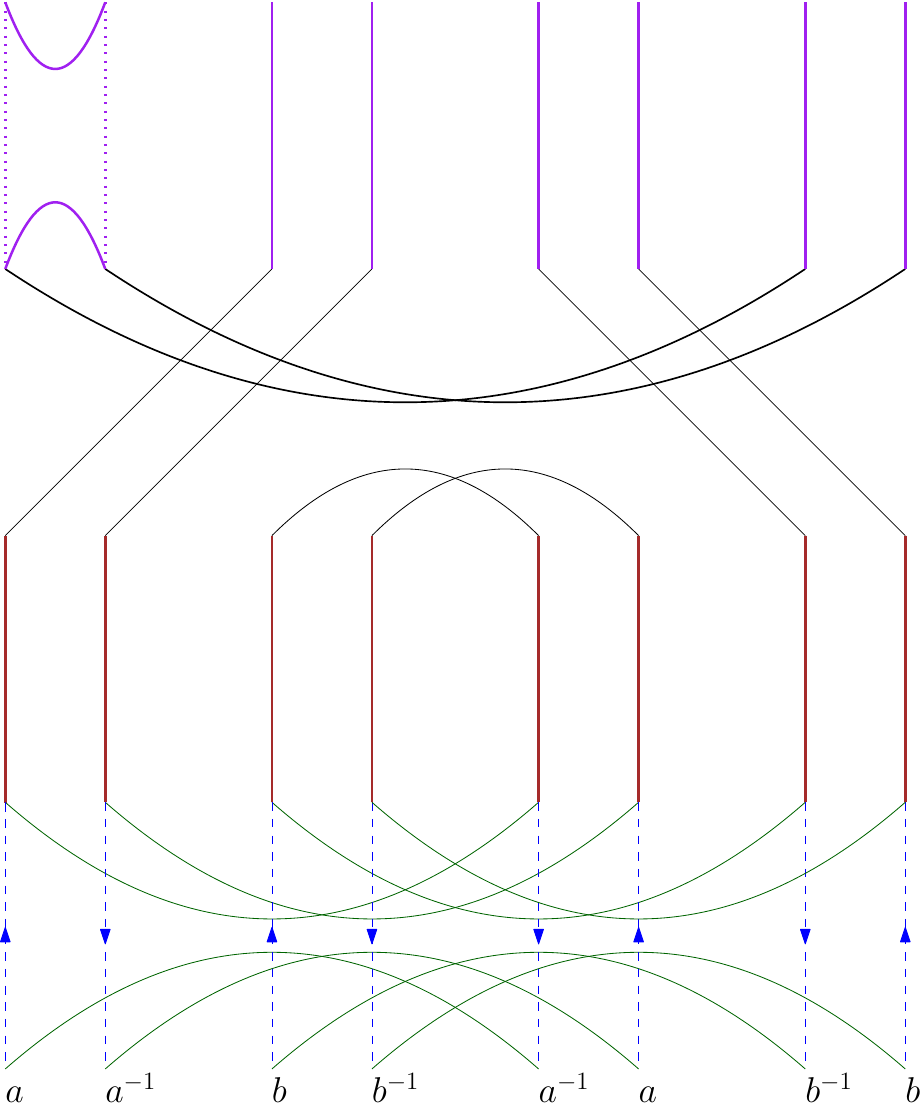}
 \qquad \includegraphics[width=6cm,height=7 cm]{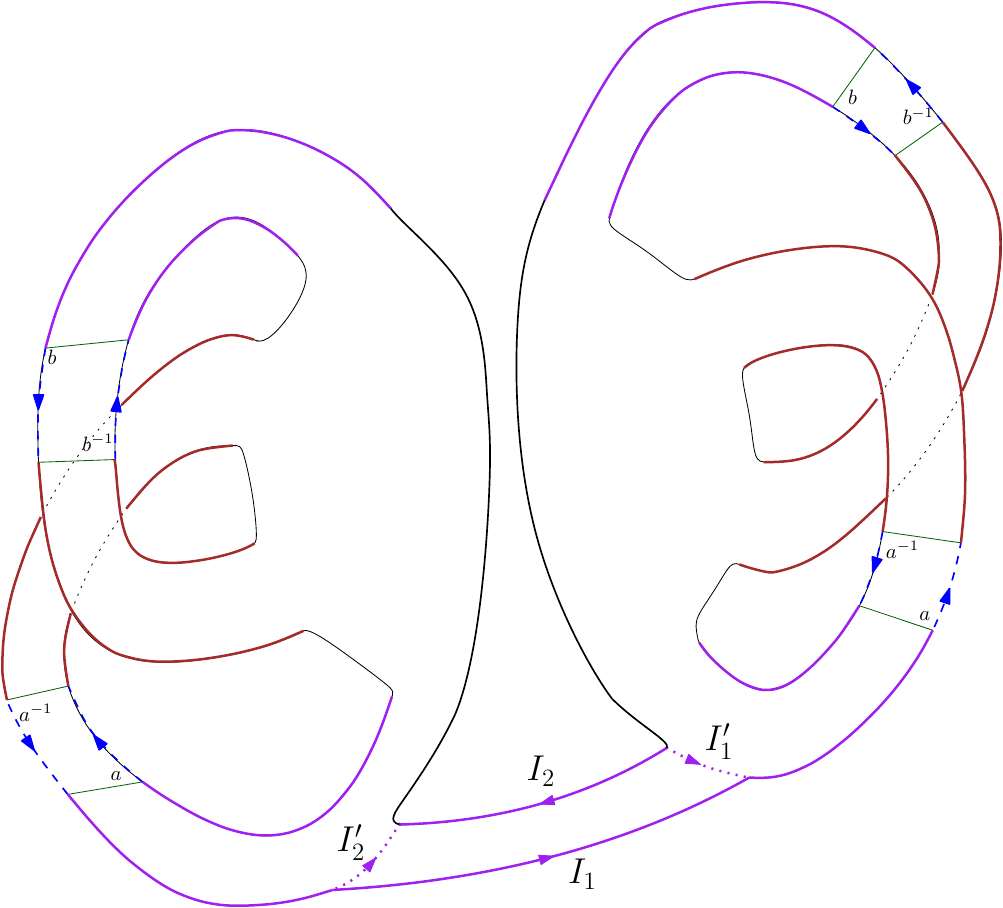}
\caption{A Brauer map \label{TwistedBrauerMap} $\mfm$ with $h(\mfm)=1$ and its associated surface.  The intervals $I_1,I_2,I'_2,I'_3$ displayed in the right-hand side are as defined  in the first case of  Lemma  \ref{Lem---CutORCompress}'s  proof. }
\end{figure}

Similarly to the previous lemma, let us argue that we can focus on Brauer maps $(\tau_u,\pi,\tau_b,\varphi)$ with $\varphi$ a quadrangle diagram.

\begin{lem}  \label{Lem---CutQuadrangles} For each  Brauer map $(\tau_u,\pi,\tau_b,\varphi)$ with arc-boundary colouring $\mfa,\mfb,$ there is a quadrangle diagram $\varphi'\in\mathcal{D}^h_{\mfa,\mfb}$ with 
\[\chi(\Sigma_{\tau_u,\pi,\tau_b,\varphi})\le \chi(\Sigma_{\tau_u,\pi,\tau_b,\varphi'}).\]
When   $\varphi$ is admissible\footnote{recall the definition of subsection \ref{SubSec--HorizontalDiag}} for $\mfa,\mfb$, so is $\varphi'.$  
\end{lem}

\begin{proof}  Consider a Brauer map $\mfm=(\tau_u,\pi,\tau_b,\varphi)$ with   arc-boundary colouring $\mfa:[s]\to\mathcal{A},\mfb:[s]\to\mathcal{S}$ and $|\varphi|>0.$ 
Assume $\varphi=\pi_{\a,\b},$ set {$\sigma=\beta\alpha^{-1}$ and 
$\tilde\sigma=\alpha^{-1}\beta$. }  Let us first argue that there is  $\a'\in S_{\mfa}$ such that\footnote{{With the choice below,  in this first step, $\a'\in S_{\mfa}$ but possibly $\a'\not\in S_\mfb,$  even when $\varphi\in D_{\mfa,\mfb}$.}} setting $\varphi'=\pi_{\a',\b} $ and $\mfm'=(\tau_u,\pi,\tau_b,\varphi'),$
\begin{equation}
|\varphi'|=|\varphi|-1\quad\text{and}\quad\chi(\Sigma_{\mfm'})-\chi(\Sigma_\mfm)\in\{0,2\}.\label{eq---InductionCUTVARPHI}
\end{equation}
Since $|\varphi|>0,$ there is  a cycle  of  $\varphi$ with length larger than $2.$  
Orienting  $\varphi$-intervals of the associated cycle as in remark  \ref{Rmk---FaceDiagram}, choose five consecutive intervals $I_1,\ldots,I_5$ of the associated $\mathcal{C}_{II}$-circle.
% As illustrated on figure \ref{Fig---CUT},
Assume w.l.o.g. the $\varphi$-intervals $I_1, I_3,I_5$  to have endpoints $((a,-1),(b,-1)),$ $((b,-2),(\sigma(a),-2))$  and $((\sigma(a),-1),(\tilde\sigma(b),-1))$ with $b=\a^{-1}(a)$ for some $a\in [s].$   Set 
\[\a'= (a \,\sigma(a)) \a \quad\text{and}\quad\varphi'=\pi_{\a',\b}.\]
With this choice  $\a'(b)=\sigma(a)=\b(b),$   ${\a'}^{-1}\b(b)=b $ and ${\a'}^{-1}\b$ has exactly one more cycle than $\a^{-1}\b$ with  $b$ as additional fixed point and  \begin{equation}
\#\varphi'=\#\varphi+1\label{eq---cylePlus}
\end{equation} so that  $|\varphi'|=|\varphi|-1.$ Moreover as  $\varepsilon\circ \sigma=\varepsilon$ and $\mfa\circ \sigma=\mfa,$ $\varphi'\in\mathcal{D}_{\varepsilon}$ and $\a'\in S_{\mfa},$ and $\mfm'= (\tau_u,\pi,\tau_b,\varphi')$ is a Brauer map.  
The cells of the complex $X_\mfm$ can be identified with the ones of $X'=X_{\mfm'}$  except for two $\varphi$-intervals,  $I_1,I_5$ for $X$, $I'_1,I'_5$ for $X'$, and the faces bounding them. Here the intervals $I'_1,I'_5$ of $X'$ have endpoints 
\[((a,1),(\tilde\sigma(b),1))\quad\text{and}\quad((\sigma(a),1),(b,1)).\]
By \eqref{eq---cylePlus}, $X'$ has exactly one $\mathcal{C}_{II}$ -circle more than $X.$ Similarly to the previous Lemma, depending on whether $I_1,I_5$ belong to the same $\mathcal{C}_I$-circle of $X,$ $X'$ has one more or one  less $\mathcal{C}_I$-circle  as $X$ and the second claim of \eqref{eq---InductionCUTVARPHI} follows.

We conclude by induction on $|\a^{-1}\b|$ that 
\[\chi(\Sigma_{\tau_u,\pi,\tau_b,\varphi})\le \chi(\Sigma_{\tau_u,\pi,\tau_b,\varphi'})\]
with $\varphi'=\pi_{\b,\b}.$  If $\varphi$ is admissible  then $\a,\b\in S_{\mfa,\mfb}$ and $\varphi'$ is also admissible.

\end{proof}

\subsubsection{Upper-bound II: bound for shortest words representatives}
\label{sec---bound for shortest words representatives}
Let us fix here $\mathcal{A}=\{a_1,b_1\ldots,a_g,b_g,a_1^{-1},b_1^{-1},\ldots,b_g^{-1}\}$ to be an alphabet with $2g$ letters and their inverse  and consider a surface word
\begin{equation}
r=[a_1 \,b_1]\ldots [a_g\,b_g]\label{eq---SurfaceWord}
\end{equation}
where for $x,y\in\mathcal{A},$ $[x \, y]=xyx^{-1}y^{-1}$ for $g\ge2$ and a word $w$ in $\mathcal{A}$  which is a \emph{cyclically shortest length representative} of an element of  surface group
\[\G_g=\left\langle a_1,b_1,\ldots,a_g,b_g|r\right\rangle\]
We then consider for $n,m\ge1$ the diagram $\pi^{n,m}_{r,\omega}$
defined in section \ref{Sec---Second moment diagram for one relator} together with the associated  colouring 
$\mfa:[s]\to\mathcal{A},\mfb:[s]\to \mathcal{S}=\{R,R^{-1},W,W^{-1}\}$ with 
\[s=4g(n+m)+2|w|.\]

The following is a refinement of the argument of [Magee II, section 4.3], allowing to consider surfaces with a $W $-boundary  and a $W^{-1}$-boundary in place of  only one  $W$  boundary.

\begin{prop} \label{Prop---GeoBoundEuler}
Assume   $\mfm=(\tau_u,\pi,\tau_b,\varphi)$ is an admissible Brauer map with $\pi=\pi^{n,m}_{r,\omega}$. 
 
Then 
\[\chi(\Sigma_\mfm)\le -n-m+h(\mfm). \]
\end{prop}

Thanks to the reduction lemmas  \ref{Lem---CutORCompress} and \ref{Lem---CutQuadrangles}, to prove Proposition \ref{Prop---GeoBoundEuler}, it  is enough to show  
\begin{equation}
\chi(\Sigma_{\tau_u,\pi,\tau_b,\varphi})\le -n-m,\label{eq---EulerBoundQBrauerMap}
\end{equation}
when $\tau_u,\tau_b$ and $\varphi=\pi_{\a,\a}$ are respectively  compatible permutation diagrams and an admissible 
quadrangle diagram, hence with $\tau_u,\tau_b\in\mathcal{D}^0_{\mfa,\mfb},\a\in S_{\mfa,\mfb}.$  Under this assumption, let us  rename  $\varphi$-arcs $\a$-arcs. For such a Brauer map each $\mathcal{C}_{II}$-face of $\Sigma_{\tau_u,\pi,\tau_b,\varphi}$ can 
be identified with a rectangle 
with two parallel $\a$-arcs and two parallel $\iota$-arcs labeled by a letter of $\{a_1,b_1,\ldots,b_g\} $ and its inverse. 
 Recall that each connected component of 
$\partial \Sigma_{\tau_u,\pi,\tau_b,\varphi}$ is the concatenation of $\tau,\pi$ and $\iota$ intervals. Since  $\tau_u,\tau_b,\pi,\id_{s}$ are compatible for $\mfb,$   vertices along each  boundary component have same $\mfb$-colouring say  $B\in\mathcal{S}$. In the latter case call the boundary component \emph{ $B$-boundary.} Since $\tau_u,\tau_b\in\mathcal{D}_{\mfa,\mfb}^0$, reading 
the $\mfa$-value  of  $\iota$ intervals along an 
oriented $B$-boundary  spells out  a power $p$ of a cyclic permutation of the word labeled $B$ for each cycle of length $p$ of $\tau.$ We then assign the \emph{multiplicity of the boundary component to be $p. $}   Note that  $p=1$ when 
$B\in\{W,W^{-1}\},$  since, as permutations,  $\tau_u,\tau_b$ act trivially on $\mfb^{-1}(\{W,W^{-1}\}).$ Counted with their multiplicity, there are $n$ boundaries  coloured  $R,$ $m$  coloured $R^{-1},$ one   $W$ and one
   $W^{-1}.$  Recalling the constraint \eqref{eq---AdmissibilityPerm},  since $\a\in S_{\mfa,\mfb},$ an $\a$-arc cannot have one endpoint in a  $R$-boundary  and the other in a $R^{-1}$-boundary.

Let us denote by $\Sigma_{\tau_u,\tau_b,\omega,\a}$  the surface  obtained  by retracting  each rectangle  in the directions of $\iota$-arcs to a  middle $\a$-parallel interval, thereby 
collapsing their top and bottom $\a$-arcs, labelling the middle arc with a transverse vector  and $\{a_1,\ldots,b_{2g}\}$-label  as illustrated in the example of  Figure 
\ref{Fig---ContractingRect}. 
  \begin{figure}\label{Fig---ContractingRect}

{\centering  \includegraphics[width=12cm,height=6cm]{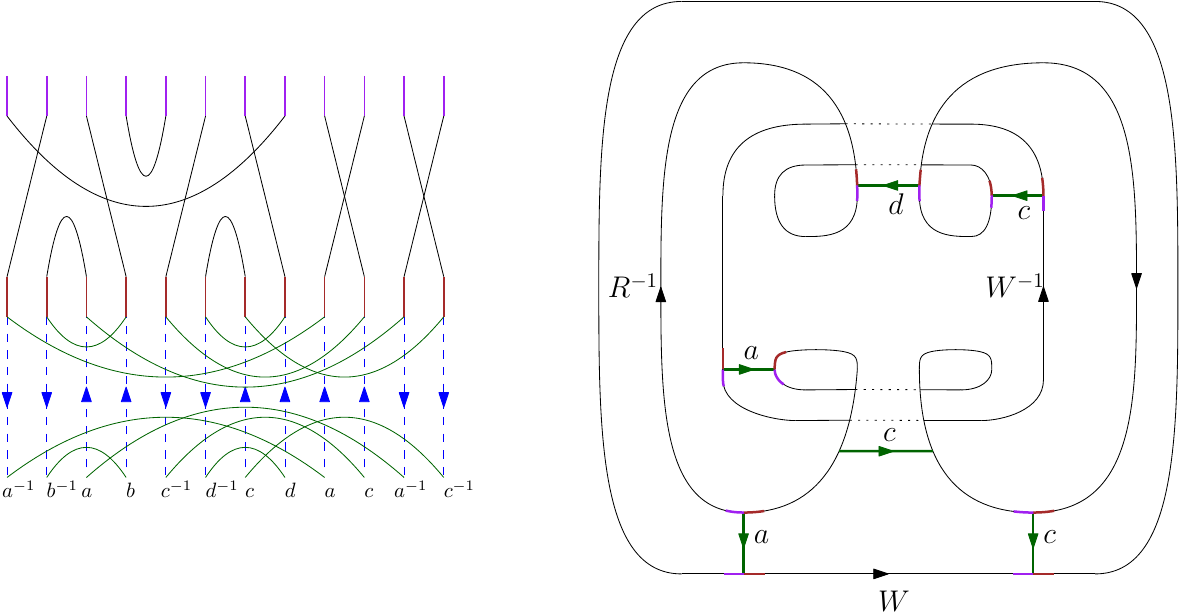}
}

 \caption{Left: a Brauer map associated to a quadrangular Brauer map $(\tau_u,\pi,\tau_b,\varphi)$ with $\pi=\pi^{0,1}_{r,ac}$ and its associated colouring, in particular,  
  $\mfb(s)$  is  $R^{-1}, W$ and $W^{-1}$   when $s$ is in  respectively $\{1,\ldots,8\}, \{9,10\}$ and $  \{11,12\}.$  Right: its associated surface with marked arcs in green. }
\end{figure}

The surface $\Sigma_{\tau_u,\tau_b,\omega,\a}$ being homeomorphic to $\Sigma_{\tau_u,\pi,\tau_b,\varphi},$ $\chi(\Sigma_{\tau_u,\tau_b,\omega,\a})=\chi(\Sigma_{\tau_u,\pi,\tau_b,\varphi})$ and we need to prove 
\begin{equation}
\chi(\Sigma_{\tau_u,\tau_b,\omega,\a})\le -n-m\label{eq---EulerBoundCullerS}.
\end{equation}
The new surface $\Sigma_{\tau_u,\tau_b,\omega,\a}$ is the total space of a $2$-CW complex  $X_{\tau_u,\tau_b,\omega,\a}$ with $2$-cells in one-to-one correspondence with $\mathcal{C}_I$ cycles of $X_{\tau_u,\pi,\tau_b,\varphi}$ and  its $1$-skeleton can be identified to the $1$-skeleton of $X_{\tau_u,\pi,\tau_b,\varphi}$ where $\iota$-intervals are collapsed to their middle point.  
The boundary of $\Sigma_{\tau_u,\tau_b,\omega,\a}$ is a disjoint union of circles  in one-to-one correspondence with boundary components of  $\Sigma_{\tau_u,\pi,\tau_b,\varphi}$,  each c.c. of  $\partial\Sigma_{\tau_u,\tau_b,\omega,\a}$  being obtained from the corresponding 
component of $\partial \Sigma_{\tau_u,\pi,\tau_b,\varphi}$ by contracting its $\iota$-intervals to 
points.  We then label each c.c. of $\partial\Sigma_{\tau_u,\tau_b,\omega,\a} $ with the same $\mfb$-label and same multiplicity so that, counting the latter multiplicity, there are exactly  $n$ $R$-boundaries,  $m$ $R^{-1}$-ones,   one  $W$ and  one  $W^{-1}.$  For $B\in\mathcal{S},$ we  denote by $\partial \Sigma_{B}$ the union of $B$-boundaries of  $\Sigma_{\tau_u,\tau_b,\omega,\a}.$
%\color{orange}
{Following an oriented boundary component of $\Sigma_{\tau_u,\tau_b,\omega,\a}$  reading the label  
of the transverse vector whenever an $\a$-arc is crossed positively, or its inverse when it is crossed negatively, spells out a word   cyclically equivalent either to $\omega,\omega^{-1}$  or to a power of $r.$  Since $\omega,\omega^{-1}$ are cyclically reduced, the two endpoints of  an $\a$-arc cannot appear \emph{consecutively} on the same $W$ or $W^{-1}$ boundary.
}
 
\begin{rmk}
For each open disc, its boundary in $\Sigma_{\tau_u,\tau_b,\omega,\a}$ is given by the range of 
the closed curve  obtained by concatenation of $\a$ arcs and $\tau,\pi$-intervals according to the associated $\mathcal{C}_I$-circle.
Mind that this curve might go through the same $\a$-arc of $\Sigma_{\tau_u,\tau_b,\omega,\a}$  twice with opposite orientation.
\end{rmk}

 For $B,B'\in\mathcal{S},$ we say an $\a$-arc is a $BB'$-arc when one endpoint belongs  to an  $B$-boundary and the other to a $B'$-boundary.  Recall that since $\a\in S_{\mfa,\mfb},$ \textbf{there are no $RR^{-1}$-arcs}.\footnote{In contrast, there might be 
$WW^{-1}$ arcs.}  We say an arc is  a  $WR^*$-arcs if it is  of type $WR,WR^{-1},W^{-1}R$ or $W^{-1}R^{-1}$.

\begin{rmk}
As the boundary components of $\Sigma_{\tau_u,\tau_b,\omega,\a}$ can include both $W$ and $W^{-1}$-intervals, the argument of \cite{MageeII} does not apply a priori. We shall however 
argue that it can be adapted and that the bound \eqref{eq---EulerBoundCullerS} still holds. 

\end{rmk}

We will next compute and bound $\chi(\Sigma_{\tau_u,\tau_b,\omega,\a})$ cutting $\Sigma_{\tau_u,\tau_b,\omega,\a}$ along some $\a$-arcs, extracting discs or annuli  along  $W$ and $W^{-1}$-boundaries of $\Sigma_{\tau_u,\tau_b,\omega,\a}$, such 
that 
each component has in its boundary only $W$-intervals or  only $W^{-1}$ intervals.\footnote{Boundary components of pieces shall be concatenation of  $\a$-arcs and of boundary 
intervals. Also, since there are no $RR^{-1}$-arcs, boundary component of pieces shall include only $R$-intervals, or only $R^{-1}$-intervals.}  As  in \cite{MageeII}, for these { "good" components}, called   below pieces,  the shortest length condition on  $\om$ allows to  bound from above the number of $W^{\pm}R^{\pm},W^{\pm}R^{\mp} $ arcs by the ones of $RR$ or $R^{-1}R^{-1}$  arcs.\footnote{Allowing components with  $WW^{-1}$-arcs would break this bound.}   The main points here is then to prove that a variant of \cite{MageeII}  yields the correct Euler characteristic bound for the remaining components despite having possibly more types of components.  The argument can be split in two, with a first part bounding  the Euler characteristics, denoted $\chi_-$ below, of a "neighbourhood" of the good components. 

\begin{defn}  Each $2$-cell of $X_{\tau_u,\tau_b,\omega,\a}$ has a boundary  in  $\Sigma_{\tau_u,\tau_b,\omega,\a}$ which is obtained by alternatively concatenating $\a$-arcs and    intervals  of $\partial\Sigma_{\tau_u,\tau_b,\omega,\a}.$  We say such a $2$-cell is a \emph{pre-piece} 
when exactly one of these intervals    belongs to  a $W$ or $W^{-1}$-boundary.  We then call \emph{piece}  any
  connected component of the union of pre-piece  open $2$-cells and of $WR^*$-arcs.\footnote{Mind that the closure   in $\Sigma_{\tau_u,\tau_b,\omega,\a}$ of a pre-piece might not be disc, whenever its boundary  uses twice the same $\a$-arcs in  $\Sigma_{\tau_u,\tau_b,\omega,\a}.$  } \end{defn}

Let us first label pre-pieces and pieces according to the arcs their boundary meet. 
 
\begin{lem} 
\label{lem---PrepieceType}
   Assume $X$ is a pre-piece.  There are $\varepsilon,\varepsilon'\in\{-1,1\}$ such that the closure $\overline{X}$ 
      only intersects  $W^{\varepsilon}$ and $R^{\varepsilon'}$-boundaries.   Then $\overline{X}$ only 
 contains  intervals 
 of $W^{\varepsilon}$ and $R^{\varepsilon'}$ boundaries and exactly two $W^{\varepsilon}R^{\varepsilon'}$-arcs.   Call  such a pre-piece a $W^{\varepsilon}R^{\varepsilon'}$-pre-piece.
 
\end{lem}

\begin{proof} By assumption,  $\partial X$   intersects an  interval $I$ of a $W^{\varepsilon}$-boundary for some $\varepsilon\in\{-1,1\}$.  {Since the word $\omega$ is reduced}, the two
  $\a$-arcs attached to the endpoints $I$ are distinct.  Since a pre-piece boundary  only has exactly one interval that belongs to either a $W$ or $W^{-1}$ boundary, the other endpoints  of these  $\a$-arcs which are not in $I$ cannot 
 belong to a $W$ or $W^{-1}$ boundary.   Since there are no $RR^{-1}$-arcs, $\partial X$ cannot meet both $R$ and $R^{-1}$ boundaries and the claim follows.
 
 \end{proof}

\begin{lem}\label{lem---PieceType} Each piece $P$ is either  an 
 open annulus, or the union of a disc $D$ with two distinct $W^{\varepsilon}R^{\varepsilon'}$-arcs belonging to $\partial D$.   We set  
 $\chi(P)=0$ when $P$ is an annulus and $1$ otherwise. In both cases, $\overline{P}$ only intersects pre-pieces and  $WR^*$-arcs of same type, say $W^{\varepsilon}R^{\varepsilon'},$ that we then call type of $P.$
\end{lem}

\begin{proof} Assume $X$ is  a  $W^{\varepsilon}R^{\varepsilon'}$-pre-piece or a  $W^{\varepsilon}R^{\varepsilon'}$-arc for some $\varepsilon,\varepsilon'\in\{-1,1\}.$ Thanks to Lemma \ref{lem---PrepieceType}, when $X$ is a pre-piece, among pre-pieces and $WR^*$-arcs,    $\overline{X}$  only  intersects closures of $W^{\varepsilon}R^{\varepsilon'}$-pre-pieces and $W^{\varepsilon}R^{\varepsilon'}$-arcs   along the two $W^{\varepsilon}R^{\varepsilon'}$-arcs  in its boundary.   Consider a connected set  $ P',$ union of  pre-pieces and $WR^*$-arcs  with  $ P'\supset X.$ By induction on the number of  pre-pieces $\overline{P}'$ intersects,  $P'$ is either  an open annulus or a disc $D$  with  two distinct $W^{\varepsilon}R^{\varepsilon'}$-arcs in its boundary, and $P'$ only intersects arcs and pre-pieces of type $W^{\varepsilon}R^{\varepsilon'}.$ \end{proof}

Let us now consider words spelled out when going through  boundaries of a pre-piece. Let $P$ be a $W^{\varepsilon}R^{\varepsilon'}$-piece.    Consider first  intersections with $\partial\Sigma_{W^\varepsilon}$ and then the one with $\partial \Sigma_{R^{\varepsilon'}}.$ 

\begin{defn} The side length $e(P)$ is the number of $W^{\varepsilon}R^{\varepsilon'}$-arcs endpoints the boundary $\partial \Sigma_{W^{\varepsilon}}\cap \overline{P}$ goes through.  Going throughout $\partial \Sigma_{W^{\varepsilon}}\cap \overline{P}$ according to its orientation and  reading letters  for each   $\a$-arc crossed, spells out a  word or a cyclic equivalence of a word $\omega_P$ when $\chi(P) $  is respectively $1$ or $0.$
\end{defn}
Either  $\chi(P)=0$ and $\partial \Sigma_{W^{\varepsilon}}\cap \overline P=\partial \Sigma_{W^{\varepsilon}},$ or $\chi(P)=1$ and  $\partial \Sigma_{W^{\varepsilon}}\cap \overline P$  is a single interval of $\partial\Sigma_{W^{\varepsilon}}$. 
In both cases, 
 \begin{equation}
 e(P)=|\omega_P|\label{eq---SidesLengthWord}
 \end{equation}
 and\footnote{Besides remark that this intersection has respectively $e(P)$ or $e(P)-1$ $\pi$-intervals.}    $e(P)-\chi(P)$  is the number of pre-pieces included in $P.$ The word $\omega_P$ is a subword of a cyclic permutation of $\omega^{\varepsilon}$ when $\chi(P)=1,$ and   the cyclic-equivalence class of  $\omega^{\varepsilon}$ when $\chi(P)=0.$

\begin{defn} Assume $ P'$  is a $W^{\varepsilon}R^{\varepsilon'}$-pre-piece,  $\mathfrak{he}(P')+1 $ denotes the number of connected components of $  \partial \Sigma_{R^{\varepsilon'}}\cap\overline{P}'. $  
\end{defn}
By definition, any  $W^{\varepsilon}R^{\varepsilon'}$-pre-piece $P'$ meets  $\partial\Sigma_{R^{\varepsilon'}}$ in its boundary and $\he(P')\ge0.$   Since there are no $RR^{-1}$-arcs,  possibly counting  $\a$-arc with both orientations, a loop  parametrizing $\partial P$  goes through $\he(P')$ oriented arcs of type $R^{\varepsilon'}R^{\varepsilon'}$ and two $W^{\varepsilon}R^{\varepsilon'}$-ones. Now for a piece $P,$ set 
\begin{equation}
\he(P)=\sum_{P'} \he( P'),\label{eq---DefHE}
\end{equation}
summing over pre-pieces $P'$ included in $P.$   Then $\he(P)$ counts\footnote{Summing over the same index as in \eqref{eq---DefHE} leads also to
\begin{equation}
\sum_{P'}(\he(P')+1)=e(P)-\chi(P)+\he(P).
\end{equation}} the number of oriented $R^{\varepsilon'}R^{\varepsilon'}$-arcs  visited when going through $\partial P.$  

As in \cite[Lemma 4.10.]{MageeII}, the two quantities $\he(P)$ and $e(P)$ are constraint by the Birman-Series bound  refining Dehn's algorithm as follows.

\begin{lem}\label{lem---BoundBSPieces} Assume $g\ge 1.$ For any piece $P,$  
\begin{equation}
 e(P)\le (2g-1)\he(P)+2g \chi(P). 
\end{equation}
\end{lem}

The proof is a variation of \cite[Lemma 5.18]{MageePuder}.   Since it is the main place where the shortest length condition plays a role,  let us detail a proof for completeness.

When $\om_1,\om_2$  are two words with letters  in $\mathcal{A},$ let us write $\om_1\prec_c\om_2$  when $\om_1$ is a subword of a cyclic permutation of $\om_2,$ and $\om_1\sim_c\om_2$ when $\om_1$ is  cyclic permutation of $\om_2.$ Consider the word 
\begin{equation}
\hat{r}= [a_1 \,b^{-1}_1]\ldots [a_g\,b^{-1}_g].
\end{equation}
 For any letters $x,y\in \mathcal{A}$ set 
 \begin{equation}
\he(x,y)=\inf\{|w|\ge 0: w\, \text{word}\,\text{in}\,\mathcal{A} \,\text{with}\, x^{-1}wy\prec_c \hat{r} \}.
 \end{equation}
 The following lemma is straightforward to check.
 
 \begin{lem} \label{Lem---OrderVertexFaceCayleyGraph}  Consider two letters  $x,y\in\mathcal{A}.$ The following conditions are equivalent. 
 \begin{enumerate}
 \item $xy\prec_c r,$ 
 \item  $x^{-1}y\prec_c \hat{r},$
 \item $\he(x,y)=0.$
 \end{enumerate}
 \end{lem}

  For any word $\omega=x_0x_1\ldots x_{l}$ in $\mathcal{A}$ with $l\ge 1$, define the (clockwise) winding of $\om$ as 
\begin{equation}
\he^{(1)}(\om)=\sum_{k=0}^{l-1} \he(x_k,x_{k+1})
\end{equation}
or setting $x_{l+1}=x_0,$
\begin{equation}
\he^{(0)}(\om)=\sum_{k=0}^{l} \he(x_k,x_{k+1}).
\end{equation}

 \begin{lem}{(\cite{BirmanSeries} and \cite[Lem 5.18]{MageePuder})} \label{Lem-BSWORDS} Assume $\omega$ is a subword of $\g,$ where  $\g$ is a shortest length representative of a non-trivial element of $\Gamma_g$ with $g\ge1.$ Then 
\begin{equation}
|\om|\le (2g-1)\he^{(1)}(\om)+2g. \label{eq---BSInterval}
\end{equation}
When $\om=\g,$ 
\begin{equation}
|\om|\le (2g-1)\he^{(0)}(\om).
\end{equation}
\end{lem}

\begin{proof} We shall prove here only the first claim as the second one follows by the same argument and write $\he$ in place of $\he^{(1)}.$  

Let us fix  a word $\om$ in $\mathcal{A}.$  For any pair of words $\om_1,\om_2,$ write $\om_1\prec\om_2$ when $\om_1$ is a subword of $\om_2.$  

  Call a word $b$ in $\mathcal{A}$ a bloc if $b\prec_c r$ and $b\prec \om.$ Say it is long when $|b|>2g.$
  
  Call a word $\mathscr{C}$ a chain\footnote{By convention a chain with $l=1$ is a bloc, while a chain with $l=2$ is a word $b_1b_2$ with $b_1,b_2$ blocs such that $b_1b_2\prec \om$ and $\he(b_1b_2)=1.$ }  when $\mathscr{C}\prec \om$ and $\mathscr{C}=b_1\ldots b_l$ with $b_i$ blocs such that $|b_{i}|=2g-1$ for $2\le i\le l-1$ and $\he(b_{i}b_{i+1})=1$  for $1\le i<l.$ Say it is a long chain  when $|b_{1}|,|b_l|\ge 2g. $

For any bloc $b,$ denote $\overline{b}$ the shortest  word in $\mathcal{A}$ with $b\overline{b}\sim_c r$. For any chain $\mathscr{C}=b_1\ldots b_l,$ set $\overline{\mathscr{C}}=\overline{b}_1\ldots \overline{b}_l$.   For any long bloc $b$,  resp. long chain $\mathscr{C}$, replacing $b$ by $\overline{b}$ in $\g,$  resp.  $\mathscr{C}$ by $\overline{\mathscr{C}},$  yields a word $\g'$ with $|\g'|<|\g|$ representing the same element in $\Gamma$ as $\g.$  Assume now $\om$ as in the statement.  Since $\g$ is a shortest length representative, from the above discussion  there are neither long blocs nor long chains.  Let us write $\om=b_1\ldots b_l$ with blocs $(b_i)_{1\le i\le l}$   such that $\he(b_ib_{i+1})\ge 1 $ for $i<l.$ Set $d_i=|b_i|-(2g-1)\he(b_{i}b_{i+1}) $ for $i<l$ and $d_l=|b_l|-2g+1$ so that 
\[|\om|-(2g-1)\he(\om)=2g-1+\sum_{i=1}^md_i.\]
Since there is  no long bloc and no long chain, $d_m\le 1$ for all $m$ and whenever   $d_i=1=d_j$ for some $1\le i<j\le l$, there is $i< k<j $  with $d_{k}=-1.$ Consequently $\sum_{i=1}^md_i\le 1$ and \eqref{eq---BSInterval} follows from the last display. \end{proof}

\begin{rmk} \label{rmk---CClockWinding} Since  $|\om|=|\om^{-1}|,$ the inequality also holds with $\he^{\chi}_-(\om)$ defined with $\hat{r}^{-1}$ in place of $\hat{r}.$  Also note that  $r^{-1}= [b_g\,a_g]\ldots [b_1\,a_1]$ and $\hat{r}^{-1}=  [b_g\,a^{-1}_g]\ldots [b_1\,a^{-1}_1],$ so by Lemma  \ref{Lem---OrderVertexFaceCayleyGraph},  for any letters  $x,y\in\mathcal{A},$  
\begin{enumerate}
\item $xy\prec_c\hat{r}^{-1},$
\item $x^{-1}y\prec_cr^{-1},$
\item $\he_-(x,y)=0$
\end{enumerate}
are equivalent.  
\end{rmk}

\begin{proof}[Proof of Lemma \ref{lem---BoundBSPieces}]  Assume $P$ is an $W^{\varepsilon}R^{\varepsilon'}$-piece. We shall consider words spelled out  going around  $\partial P$ and the constraint it imposes on $\he(P).$

     Consider a $W^{\varepsilon}R^{\varepsilon'}$-pre-piece $P'$ of $P$. On the one hand,  each connected component of $\partial  P'\cap \Sigma_{R^{\varepsilon'}}$ is the concatenation of exactly one $\pi$-interval and two $\tau$-intervals and 
  meets exactly two $R^{\varepsilon'}R^{\varepsilon'}$-arcs endpoints. Since these $\a$-arcs endpoints are consecutive in
 
   {$\partial\Sigma_{R^{\varepsilon'}}$},  following the orientation of 
  $\partial\Sigma_{R^{\varepsilon'}},$ their  $\mfa$-value  must be given by two consecutive\footnote{\label{LRNote}reading from left to right} letters of a cyclic permutation of $r^{\varepsilon'}.$ On the other hand, for a 
  single $\a$-arc, the  $\mfa$-values at both ends  are inverse of one another.

  All in all,  starting from the interior of the interval $\partial \Sigma_{W^{\varepsilon}}\cap \overline P',$ the $\mfa$-value read along a loop parametrising $\partial P'$ following its orientation  spells out the word
  
    \begin{equation}
 x_{\he(P')+1}x_{\he(P')+1}^{-1}x_{\he( P')}x_{\he( P')}^{-1}\ldots   x_{1} x_{1}^{-1} x_{0}x_{0}^{-1}
 \end{equation}
 
where  $x_0^{-1}x_{\he(P')+1}$ is  associated to the interval   $\partial\Sigma_{W^{\varepsilon}}\cap\overline{P}'$ with same orientation as $\partial \Sigma_{W^{\varepsilon}}$ so that $x_0^{-1}x_{\he(P')+1} \prec_c\om^{\varepsilon},$ 
  while  for any $k\in\{0,\ldots, \he( P')\},$   $x_{k}^{-1}x_{k+1}$  is associated to an interval of $\partial\Sigma_{R^{-\varepsilon'}}\cap\overline{P}'$ with orientation opposite to $\partial\Sigma_{R^{-\varepsilon'}}$ and $x_{k}^{-1}x_{k+1}\prec_c r^{-\varepsilon'}.$
 Thanks to Lemma  \ref{Lem---OrderVertexFaceCayleyGraph} and Remark \ref{rmk---CClockWinding},   
 \[x_0x_1\ldots x_{\he(P')}x_{\he(P')+1}\prec_c \hat{r}^{-\varepsilon'}\] 
 so that 
 \begin{equation}
 \he(P')=\he_{-\varepsilon'}(x_0^{-1},x_{\he(P')+1})+4 g r_{P'} \label{eq---RamifIntMap}
 \end{equation}
 for some $r_{P'}\in \N,$ where $\he_1,\he_{-1}$ stand for $\he,\he_-.$    Now writing  $\om_P=z_0\ldots z_l, $ recall  $P$ has $e(P)-\chi(P)=|\om_P|-\chi(P)=l+1-\chi(P)$ pre-pieces  $(P'_k)_{0\le k\le l-\chi(P)}. $ Setting   $z_{l+1}=z_0$ and using 
 \eqref{eq---RamifIntMap},
\begin{equation}
\he(P)=\sum_{k=0}^{l-\chi(P)} \he(P'_{k}) \ge \sum_{k=0}^{l-\chi(P)} \he_{-\varepsilon'}(z_{k},z_{k+1})=\he^{(\chi(P))}_{-\varepsilon'}(\om_P).
\end{equation}
Since  $\om_P\prec_c \om^{\varepsilon},$ with  $\om_P\sim_c \om^{\varepsilon}$ when $\chi(P)=0,$   Lemma \ref{Lem-BSWORDS} implies
\[(2g-1)\he(P)\ge |\om_P|- 2g\chi(P)=  e(P)-2 g\chi(P). \]\end{proof}

 Let us now prove Lemma \ref{lem---BoundBSPieces} implies \eqref{eq---EulerBoundCullerS}, which will conclude the proof of Proposition \ref{Prop---GeoBoundEuler}.

  \begin{defn} Say a $2$-cell of $X_{\tau_u,\tau_b,\omega,\a}$ is a \emph{joint disc} if it is not a pre-piece.
  \end{defn}
  Let us first reduce the Euler characteristic computation in \eqref{eq---EulerBoundCullerS} to the one of a graph. Consider the surface $\Sigma^*$ obtained from $\Sigma_{\tau_u,\tau_b,\omega,\a}$ by cutting it along all its $RR$ or $R^{-1}R^{-1}$ arcs.   Then denoting by $N_{RR^*}$ the number of such arcs
\begin{equation}
\chi(\Sigma^*)=\chi(\Sigma_{\tau_u,\tau_b,\omega,\a})+N_{RR^*}. \label{eq---EulerCutRR}
\end{equation}
Since all connected components of $\Sigma^* $ and $\Sigma_{\tau_u,\tau_b,\omega,\a}$ have non-empty boundary, they deform retract to the dual graph  associated to any cell decomposition.
Denote $\hat{\mathcal{G}}$  the graph dual to the $1$-skeleton of $X_{\tau_u,\tau_b,\omega,\a}$, so that vertices are labeled by $2$-cells of $X_{\tau_u,\tau_b,\omega,\a}$ and edges by $\a$-arcs.  
  Let  $\hat{\mathcal{G}}_{R}$ be the graph obtained from $\hat{\mathcal{G}}$ by removing all edges given by  $RR$  and $R^{-1}R^{-1}$-arcs.   Then $\Sigma^*$ deforms retract to $\hat{\mathcal{G}}_{R}$ so that 
\begin{equation}
\chi(\Sigma^*)=\chi(\hat{\mathcal{G}}_{R})=\sum_{D} \left(1-\frac{d^*(D)}{2}\right)\label{eq---EulerDualGraph}
\end{equation}
  where the sum is over $2$-cells of  $X_{\tau_u,\tau_b,\omega,\a}$ and $d^*(D)$ is the degree of $D$ in $\hat{\mathcal{G}}_R,$ that is, the number of oriented $WR^*$ or $WW^*$-arcs the boundary $\partial D$ goes through.  Denoting  by $d_{WR^*}(D)$ the number of $WR^*$-arcs $\partial D$ goes through,  
  \[d_{WR^*}(D)\le d^*(D)\]
  and  as there is no $RR^{-1}$ arcs, $D$ is a pre-piece  if and only if
\begin{equation}
d_{WR^*}(D)= d^*(D)=2.\label{eq---EqPiece}
\end{equation}
The only contribution to $\chi(\Sigma^*)$ therefore comes from joint discs.   Let us  split the sum \eqref{eq---EulerDualGraph} as 
\[\chi(\Sigma^*)=\chi_-+\chi_+\]
with 
\[\chi_-= \sum_{D \,\text{joint disc}: \,d_{WR^*}(D)>0} \left(1-\frac{d^*(D)}{2}\right)\quad\text{and}\quad \chi_+= \sum_{D \,\text{joint disc}: \,d_{WR^*}(D)=0}
 \left(1-\frac{d^*(D)}{2}\right).\]

Let us start with the left sum.

\begin{lem} If $D$ is a joint disc with $d_{WR^*}(D)>0,$ then \label{Lemma---JointDNeg}
\begin{equation}
1-\frac{d^*(D)}{2}\le -\frac{d_{WR^*}(D)}{4}.
\end{equation}
\end{lem}
\begin{proof} Since type of boundaries alternate between %\color{orange}
{$R^*$ and $W^*$} boundaries when going through $\partial D$ and crossing a $WR^*$-arc, $d_{WR^*}(D)$ is even. If $d_{WR^*}(D)\ge 4, $ then 
\[d^{-1}_{WR^*}(D)\left(1-\frac{d^*(D)}{2}\right)\le d^{-1}_{WR^*}(D)\left(1-\frac{d_{WR^*}(D)}{2}\right)\le \frac 1 4-\frac 1 2=-\frac 14.   \]
If $d_{WR^*}(D)=2,$ since $D$ is not a pre-piece, $d^*(D)\ge 3$ and   
\[1-\frac{d^*(D)}{2}\le   -\frac 1 2=-\frac{d_{WR^*}(D)}{4}. \]
\end{proof}

 Since any $WR^*$-arc bounding a joint disc is also bounding a piece, while any piece boundary goes through exactly two distinct oriented $WR^*$-arcs,\footnote{with the convention that for $WR^*$-arc that do not belong to any pre-piece, the boundary of the piece is exactly the arc with some orientation followed by its reverse. } 
 
  \begin{equation}
  \sum_{P \,\text{piece}} 2\chi(P) = \sum_{D \,\text{joint disc}: \,d_{WR^*}(D)>0} d_{WR^*}(D).
 \end{equation}
Together with Lemma \ref{Lemma---JointDNeg},
 \begin{equation}
 \chi_-= \sum_{D \,\text{joint disc}: \,d_{WR^*}(D)>0} \left(1-\frac{d^*(D)}{2}\right)\le -\frac 1 2 \sum_{P \,\text{piece}} \chi(P). \label{eq---PreBoundNegPart}
 \end{equation}
Now using the Birman-Series's inequality of  Lemma \ref{lem---BoundBSPieces},
\begin{equation}
-\frac 1 2 \sum_{P \,\text{piece}} \chi(P)\le  \frac 1 {4g} \sum_{P \,\text{piece}} \left( (2g-1) {\he(P)}-e(P)\right).\label{eq---FirstBoundNeg}
\end{equation}
Let us rewrite the right-hand side in terms of number of $\a$-arcs. Since each  $WR^*$-arc  is included  in exactly one piece, their number $N_{WR^*}$ is 
\begin{equation}
N_{WR^*}= \sum_{P \,\text{piece}}\# \{WR^*\text{-arc included in }P\}=\sum_{P \,\text{piece}}  e(P).
\end{equation}
 Denote by $\he'_{RR^*}$ the number of oriented $RR^*$-arcs  through which no piece boundary goes positively. Since each oriented $RR^*$-arc is positively bounding at most one piece,
\begin{equation}
2 N_{RR^*}=\he'_{RR^*}+\sum_{P \,\text{piece}}  \he(P).
\end{equation}
Using the last two identities with \eqref{eq---PreBoundNegPart} and \eqref{eq---FirstBoundNeg}, 
\begin{equation}
\chi_-\le N_{RR^*} - \frac{2 N_{RR^*} +N_{WR^*}}{4g} - \frac{(2g-1)\he_{RR^*}'}{4g} .\label{eq---EulerNegBound}
\end{equation}
By definition, the vertices of $X_{\tau_u,\tau_b,\omega,\a}$  belonging to a $R^*$-boundary and an  $\a$-arc can be identified with $[4g(n+m)]$. Since each $\a$-arc has two distinct endpoints, 
 \begin{equation}
 4g(n+m)=2N_{RR^*}+N_{WR^*}.\label{eq---ExtensionSizeArcsCount}
 \end{equation}
Gathering the above bounds yields
 \begin{equation}
\chi_-\le N_{RR^*} -(n+m)- \frac{(2g-1)\he'_{RR^*}}{4g}.\label{eq---LoopPart}
 \end{equation}
Recalling \eqref{eq---EulerCutRR},  to show \eqref{eq---EulerBoundQBrauerMap}  it is enough to prove 
 \[\chi_+\le   \frac{(2g-1)\he'_{RR*}}{4g}.\]

 As $g\ge 1,$ the following Lemma therefore concludes the proof of Proposition \ref{Prop---GeoBoundEuler}.
 \begin{lem} For any $g\ge1,$  \label{Lemma---JointDPos}
\begin{equation}
\chi_+\le \frac{\he'_{RR^*}}{4g}.
\end{equation}
\end{lem}
\begin{proof} Assume $D$ is a joint disc with $d_{WR^*}(D)=0.$ Then, as there are no $RR^{-1}$ arcs,  its boundary meets only $WW^*$,  only $RR$ or only $R^{-1}R^{-1}$ arcs. Call  $D$  a $W^*$-disc  in the first case, and  $R$-disc or $R^{-1}$-disc in the last two cases. If $D$ is a $W^*$-disc, since $\omega$ is reduced, $\partial D$ goes through at least  two $WW^*$-arcs so that $d^*(D)\ge2$ and 
\[{1-\frac{d^*(D)}2\le0}.\]
If $D$ is a $R^\varepsilon$-disc, $d^*(D)=0$ and for each oriented interval of $\partial \Sigma_{R^\varepsilon}$ meeting its boundary,  the $\mfa$-values $x,y$  at its endpoints satisfy $xy\prec r^{\varepsilon}.$  Since $\partial D$ alternates between $R^{\varepsilon}R^{\varepsilon}$-arcs and such intervals,  it follows that  $\partial D$ goes through  $4gn$ oriented $RR^*$-arcs for some $n\ge 1$. Since no piece boundary goes positively through such  oriented $RR^*$-arcs, 
\[ 4g\#\{R^{*}\text{-discs}\}\le \he'_{RR^*}.\]  
Together with the previous display,
\[\chi_+\le \#\{R^{*}\text{-discs}\}\le \frac{\he'_{RR^*}}{4g}.\] 
\end{proof}

Gathering the last Lemma with  \eqref{eq---LoopPart},
\begin{equation}
\chi(\Sigma_{\tau_u,\tau_b,\om,\a})= \chi(\Sigma^*)-N_{RR^*}=\chi_-+\chi_+-N_{RR^*}\le   - (n+m)- \frac{2g-2}{4g}\he'_{RR^*},
\end{equation}
which concludes the proof of Proposition \ref{Prop---GeoBoundEuler} and thereby completes the one of Theorem \ref{thm---BoundTraceWordCharactRel Size Irrep}.

 \begin{rmk} Note that the above argument to bound $\chi(\Sigma^*)$ follows from  \cite{MageeII}, where it was applied with only one $W$-boundary.  Essentially only the bound of $\chi_+$ needed a change. Note also the same bound applies to any number of $W,W^{-1}$ boundaries. \end{rmk}

\subsection{Second moment bound for Wilson loops}

Assume $g\ge2.$ From here onwards, $\Gbb_g$ is a fixed   map  of genus $g$ with one vertex $v_*$, $4g$ oriented edges $a_1^{\pm},\ldots,a_{2g}^\pm $ and one face, with the standard cyclic ordering of its edges %\color{orange}
{$a_1a_2^{-1}a_1^{-1}a_2a_3\ldots$} around its vertex.  We call a map \emph{regular} if it is a refinement of  $\Gbb_g. $   Let us say  a combinatorial loop $\g'$ is  an \emph{almost-shortest loop representative}  if  $\g'\in\Ld(\Gbb)$ is a  loop   of a regular map $\Gbb$ with
\begin{equation}
\g'=e\ell e^{-1}\g\label{eq---AlmostSLRep}
\end{equation}
where %\color{orange}
{$\g$ is  a path of $\Gbb_g$ that is a cyclically shortest representative of an element of \[\G_g=\<a_1,\ldots ,a_{2g}|[a_1,a_2]\ldots[a_{2g-1},a_{2g}]=1\>,\]}while $\ell$ is a loop of $\Gbb$  included in a %\color{orange}
{closed disc D} and $e\in E^o$  with  $\underline{e}=v_*,$  $\overline{e}\in D$ and  interior disjoint from $D$ and from edges of $\Gbb_g.$   For any map $\Gbb=(V,E,F)$ and $T>0$ write 
$\Delta_\Gbb(T)=\{a\in {{\R_+^*}^F}:|a|=T\}.$

\begin{thm}\label{THM---CValmostshortestRep} Assume  $\g'$ is a non-contractible, almost shortest loop representative of $\G_g$ in a  regular map $\Gbb.$  There is $K>0$ such that for all  { $T>0, a\in \Delta_\Gbb(T)$} and $N\ge 1,$
\begin{equation}
\EE_{\YM_{\Gbb,a,U(N)}}[|W_{\g'}|^2]\le \frac{K}{N^2}\frac{\theta(q_{\tilde T})}{\theta(q_T)}\quad\text{and}\quad \EE_{\YM_{\Gbb,a,SU(N)}}[|W_{\g'}|^2]\le \frac{K}{N^2}, 
\end{equation}
where $\tilde T=\min(\frac T2, a(\Sigma\setminus D)).$  In particular for any shortest word representative $\g,$ there is $K>0$ such that for all $N\ge 1,$
\begin{equation}
\EE_{ABG_g}[|W_{\g}|^2]\le \frac{K}{N^2}.
\end{equation}
\end{thm}

Recall the definition \eqref{def---TruncYM} for the truncation of the Yang-Mills measure.
 
\begin{lem}\label{lem---YMShortestRep}  Assume $\g'  $ is a non-contractible, almost shortest loop representative in a regular map $\Gbb,$  $D$  the  associated closed disc.   For any 
$k>0, $
 there is $K>0$ such that   
for any  $T>0$ and any area vector $a\in\Delta_\Gbb(T),$
\[|\YM^{(k)}_{\Gbb,a,U(N)}[|W_{\g'}|^2]|\le \frac{K}{N^2}\theta(q_{T'})\quad\text{and}\quad |\YM^{(k)}_{\Gbb,a,SU(N)}[|W_{\g'}|^2]|\le \frac{K}{N^2}.\]
where $T'=a(\Sigma\setminus D)$.   
\end{lem}

\begin{proof}  By restriction and compatibility properties \ref{Prop--Comp} and \ref{Prop---Contract}, we can assume $\Gbb$ has exactly one face $f$ that is not included in $D$ so  that  $a(f)=T'>0.$ Then 
\begin{align*}
N^2\YM^{(k,\{f\})}_{\Gbb,a,G_N}[|W_{\g'}|^2]&= \YM_{\Gbb_D,a_D,G_N} ( \EE_{U}[|\Tr(H_\ell\omega(U_1,\ldots,U_{2g}) )|^2p_{a(f)}^{(k)}( H_{\partial D}^{-1} r(U_1,\ldots,U_{2g}) ) ] )
\end{align*}
where
\[p_{a(f)}^{(k)}=\sum_{\a\in\ZD:|\a|\le k} d_\a \chi_\a e^{-\frac{a(f)\mfc_\a }2} \]
and $U_1,\ldots,U_{2g}$ are independent and Haar distributed on $G_N$.
Theorem \ref{thm---BoundTraceWordCharactRel Size Irrep} now yields similarly to \eqref{eq---BoundSupTheta}, %\color{orange}
{\[N^2|\YM^{(k,\{f\})}_{\Gbb,a,U(N)}[|W_{\g'}|^2]|\le \sum_{\a \in \ZD:|\a|\le k}K_k d_\a N^{-|\a|} e^{- \frac{T'}{2}\mfc_\a}\le \theta(q_{T'})K_k \sum_{\a \in \ZD/\Z:|\a|\le k}d_\a N^{-|\a|}\]}
and 
 
{\[N^2|\YM^{(k,\{f\})}_{\Gbb,a,SU(N)}[|W_{\g'}|^2]|\le K_k \sum_{\a \in \ZD/\Z:|\a|\le k}d_\a N^{-|\a|}.\]}
By Lemma \ref{Lem---RatioPieri} for any $\a\in\ZD/\Z,$ $d_\a\le N^{|\a|}$ and the right-hand side's sum is bounded by  $\#\{\a \in \ZD/\Z:|\a|\le k\} \le \#\{\la\in\Yb:|\la|\le k \}^2.$

\end{proof}

\begin{proof}[Proof of Theorem \ref{THM---CValmostshortestRep}]  Assume $\g'$ as in the statement. Thanks to restriction and compatibility propositions \ref{Prop--Comp} and \ref{Prop---Contract}, we can apply propositions \ref{Prop---ConvWittenZeta} and \ref{Prop---TruncationWLE} with $d=2$, there is $k'>0$ and $K>0,$ such that for all $k\ge k',  N\ge 1,T>0$ and area vector $a\in \Delta_\Gbb(T),$
\begin{equation}
|\YM^{(> k)}_{\Gbb,a,U(N)}[|W_{\g'}|^2]\le \frac{K}{N^2}\theta(q_{\frac T2})\quad\text{and}\quad |\YM^{(\ge k)}_{\Gbb,a,SU(N)}[|W_{\g'}|^2]|\le \frac{K}{N^2}.\label{eq---TruncWLE}
\end{equation}
Since $\g'$ is an almost shortest representative using Lemma \ref{lem---YMShortestRep} and \eqref{eq---TruncWLE},  there is a constant  $K'>0$ such that for all $N\ge1$ and all $T>0$ and  area vector $a\in\Delta_\Gbb(T),$ 
\begin{equation}
\YM_{\Gbb,a,U(N)}[|W_{\g'}|^2]\le \frac{K'}{N}\theta(q_{\tilde T})\quad\text{and}\quad \YM_{\Gbb,a,SU(N)}[|W_{\g'}|^2]\le \frac{K'}{N}.\label{eq---VarBoundSumUpProof}
\end{equation}
The last bound is also valid with $\g$ in place of $\g'$ and by definition,
%Thanks to {\color{orange}Theorem semiclassical limit}, when $\g' =e\ell e^{-1}\g$ is as in \eqref{eq---AlmostSLRep}, 
\begin{equation}
 \mu_{ABG,SU(N)}[|W_{\g}|^2]=\lim_{a\to 0} \YM_{\Gbb,a,SU(N)}[|W_{\g'}|^2] \le \frac{K'}{N^2}. \label{eq---VarBoundSumUpProofABG}
\end{equation}
Since   $\YM_{\Gbb,a,G_N}(1),\mu_{ABG,SU(N)}(1)$ converge as $N\to\infty$ to a non-zero limit, the claim follows from Proposition \ref{Prop---ConvWittenZeta}, \eqref{eq---VarBoundSumUpProof} and  
\eqref{eq---VarBoundSumUpProofABG}.
\end{proof}

\section{Convergence towards the planar master field}

\label{sec---MasterFieldAllLoops}

We shall explain here how to adapt the arguments of \cite{DLII}  to reduce  Wilson loops convergence for all  loops to the one for loops within a topological disc of   
$\Sigma$ or loops appearing in Theorem \ref{THM---CValmostshortestRep}. This reduction is allowed by a system  of differential equations, named after Makeenko and Migdal,   expressing in particular  the derivative of  Wilson loop expectation and variance  
along some directions in the area 
simplex.  

Let us first recall and slightly amend the main statement of \cite{DLII}.   First recall the class of loops we shall work with.  Denote by $\mfL_g$ the set of loops  with simple and transverse intersections, within some regular map of genus $g\ge 0.$

Call elements of $\mfL_g$ for some $g\ge 1$, \emph{regular} loops.  For each $N\ge 1$ and
$\mfl\in \mfL_g, T>0$ define  the following function on its area simplex   
\[\Psi^N_{\mfl}(a)= { \EE^N_{\Gbb_\mfl,a}[|W_{\mfl} -\Phi_{\mfl}(a)|^2]}\quad \forall a\in {  \Delta_\mfl(T)}\setminus\{0\} .\]
where  $\Phi_\mfl(a)=\tau_{\tilde \Gbb_U,\tilde a_U}(\tilde \mfl) $  is  defined as in Theorem \ref{THM---MFConv} when $\mfl$ is contractible and $\Phi_\ell=0$ otherwise.

The following was proved  { in \cite[Sec. 4]{CDG}    and  \cite[Sec. 5]{DN}} when  the base surface $\Sigma$ is the plane or the sphere. The argument for general surfaces being very closed to these cases is  omitted here.

\begin{thm} \label{THM---ReducFiniteLengthTOReg} Assume $T>0$ and  \begin{equation}
\lim_{N\to\infty}\|\Psi_{\mfl}\|_{\infty,\Delta_\mfl(T)}=0\label{eq---ConvL2}
\end{equation}
for all regular loops $\mfl\in \mfL_g.$ Then if $\Sigma_T$ is a closed orientable surface of genus $g\ge 2$ with a Riemannian metric of total volume $T>0,$  for any finite length loop $\ell\in \Ld(\Sigma_g)$ 
\[\lim_{N\to\infty}\EE_{\Sigma_g}[|W_\ell-{ \tau_{\Sigma_g}(\ell)}|]=0. \] 
\end{thm}

We can now  prove the assumption of  Theorem \ref{THM---ReducFiniteLengthTOReg} holds true { which will conclude the proof of Theorem \ref{THM---ContMFConv}.}  First let us recall the main reduction result of \cite{DLII}.    From here onwards let us fix  a map $\Gbb^*_g$ { dual} to $\Gbb_g.$ A {regular} loop    $\mfl$  within a regular map $\Gbb$ determines a path on $\Gbb^*_g$ and we say $\mfl$ has  \emph{minimal tiling length} when  the latter  is a {cyclically shortest representative of } $\Gamma_g.$ 

When $K$ is a proper subset of faces in a  map $\Gbb=(V,E,F)$, denote $\Delta_{\Gbb,K}$ the set of area vectors $a\in \R^F$ with $a(f)=0$ for any $f\in K$  and  by $\Delta_{\Gbb,K}(T)=\{ a\in \Delta_{\Gbb,K}:{|a|}=T\}.$

\begin{thm} \label{THM---ReducMinimalTiling} Fix $T>0.$ Assume  that for any regular loop $\mfl$ with minimal tiling length,  there is a refining regular map $\Gbb$ and a proper subset  $K$  of its faces, such that 
 \begin{equation}
\lim_{N\to\infty}\|\Psi_{\mfl}\|_{\infty,\Delta_{\mfl,K}(T)}=0.\label{eq---SLConvL2}
\end{equation}
Then \begin{equation*}
\lim_{N\to\infty}\|\Psi_{\mfl}\|_{\infty,\Delta_\mfl(T)}=0
\end{equation*}
 for all regular loops $\mfl\in \mfL_g.$ 
\end{thm}

The above statement is a minor amendment of Proposition 3.22 and Corollary 3.23 of \cite{DLII}; we explain below  the necessary changes.

\begin{proof}    The functions  $(\Psi^N_\mfl,\mfl\in\mfL_g)$ satisfies a system of differential inequalities given in (49) and (50) of \cite{DLII}, where $\phi^\infty$ can be assumed to be {the Master field\footnote{following the terminology of \cite{DLII},  the latter is indeed an exact solution of MM equations.} on $\Sigma_g$ defined in \eqref{def---MasterField}.   }
The statement follows from Proposition 3.22 and Corollary 3.23 of \cite{DLII} considering cyclically shortest length loops instead of geodesic loops and with the notion defined therein of good boundary conditions of Makeenko-Migdal equations  amended 
as follows. 
Let us set say that a subset $\mathcal{F}$ of $\mathfrak{L}_g$ is a \emph{good boundary condition for the Makeenko--Migdal equations}  if assuming for all $\mfl\in \mathcal{F}$
 \begin{equation}
\lim_{N\to\infty}\|\Psi_{\mfl}\|_{\infty,\Delta_{\mfl,K}(T)}=0 \label{eq---GoodBC}
\end{equation}
for some  proper  subset $K$ of the faces of a {refining map} of $\mfl$ implies 
\[\lim_{N\to\infty}\|\Psi_{\mfl}\|_{\infty,\Delta_\mfl(T)}=0\quad\forall \mfl\in\mfL_g.\]
This new definition of  the good  boundary condition does not change the use of the two induction lemmas 4.2 and 4.5 of \cite{DLII}. Lastly,  replacing  minimal tiling length loops in place of geodesic loops   follows from\footnote{The latter theorem shows geodesic loops induce minimal tiling length loops.}  \cite[Theorem 2.12]{BirmanSeries} or a  simpler version of  Proposition 2.17 of 
\cite{DLII}.  
\end{proof}

We now conclude  the proof of Theorem  \ref{THM---ContMFConv} with this last lemma. 
 
\begin{lem} The assumptions of Theorem \ref{THM---ReducMinimalTiling} hold true for $G_N=U(N)$ or  $SU(N).$
\end{lem}

\begin{proof}
{For any minimal tiling length loop $\mfl\in\mfL_g$ } within a regular map $(\Gbb,\Gbb_g),$  consider a finer  regular map $\Gbb'$  containing  the dual $\Gbb_g^*$ of $\Gbb_g$ and a face $D$ of $\Gbb'$ neighbouring a vertex of $\Gbb_g.$  Then   $\mfl$ is homotopic  relatively to $D$ to an almost shortest length loop representative $\g'$ in the regular map $(\Gbb,\Gbb^*_g)$. Therefore for any area vector $a\in \Delta_{\Gbb',K},$ under ${\YM_{\Gbb',a,U(N)}}$, $W_{\mfl}$ has same distribution as $W_{\g'}.$ If $\g'$ is contractible it is included in a topological disc and by Theorem 2.14 of \cite{DLI},  
  \[ \sup_{a\in\Delta_{\Gbb',K}}\EE_{\YM_{\Gbb',a,G_N}}[|W_{\mfl}-\Phi_\mfl(a)|^2 ]= \sup_{a\in\Delta_{\Gbb',K}}\EE_{\YM_{\Gbb',a,G_N}}[|W_{\g'}-\Phi_\mfl(a)|^2 ]\underset{N\to\infty}{\to}0. \]
 Otherwise by Theorem \ref{THM---CValmostshortestRep},  
 
 \[\sup_{a\in\Delta_{\Gbb',K}} \EE_{\YM_{\Gbb',a,G_N}}[|W_{\mfl}|^2 ]=\sup_{a\in\Delta_{\Gbb',K}}\EE_{\YM_{\Gbb',a,G_N}}[|W_{\gamma'}|^2 ]\underset{N\to\infty}{\to}0. \]
\end{proof}

\appendix

\setcounter{secnumdepth}{5} 
\renewcommand\thefigure{\thesection.\arabic{figure}}  
\setcounter{figure}{0} 
 \section{Witten-Zeta function truncation}
 \label{sec---Witten-Zeta function truncation}
Let us recall  the argument of \cite{LarsenWZeta} to prove \eqref{eq---BoundWZT}  of proposition \ref{Prop---ConvWittenZeta}. For any $\a\in\ZD,$ set 
\begin{equation}
\omega_l(\a)=\a_l-\a_{l+1},\forall l\in \{1,\ldots, N-1\}. 
\end{equation}
For $s>0,$ consider the probability on $\ZD$ given by 
\[\]
 \begin{lem}[\cite{LarsenWZeta}Lemma 8] There are $v_1,\ldots,v_{N-1}>0$ with \label{Lem---BoundGLM}
\begin{equation}
v_l=v_{N-l}\ge l\max(1,\log(N)-\log(l)),
\end{equation}
such that  for all $\a\in\ZD,$
 \[d_\a\ge \prod_{l=1}^{N-1}(1+\omega_l(\a))^{v_l}.\]
 
 \end{lem}
 
 \begin{proof} Recalling Weyl dimension formula and using the AM-GM inequality,
 
 \begin{align*}
    d_\a&= \prod_{1\le  i<j\le N}(j-i+\omega_i(\a)+\omega_{i+1}(\a)+\ldots+\omega_{j-1}(\a) )\\
    &\ge   \prod_{1\le  i<j\le N, i\le l< j} (1+\omega_l(\a))^{\frac 1 {j-i}}\ge \prod_{1\le l\le N-1}(1+\omega_l(\a))^{v_l}
 \end{align*}
 where
 \begin{equation}
 v_l=v_{N-l}= \sum_{i,j:1\le i \le  l<j\le N}\frac{1}{j-i}, \label{eq---coeffFundWDimRep}
 \end{equation}
 where the first identity follows setting $i'=N-j+1,j'=N-i+1.$
 For all $ l\le\frac N2,$
 \[ v_l=\sum_{k= 1}^{N-1} \frac{1}{k} \#\{ i,j: 1\le i\le l<j\le N,j-i=k\}\ge \sum_{k= 1}^{l} \frac{1}{k} k=l. \]

and 
\[v_l\ge  \sum_{i,j:1\le i \le  l<j\le N} \int_{j}^{j+1}\frac{dx}{x-1}= l(\log(N)-\log(j)).  \]
 \end{proof}
 The above  bound   yields the following two estimates.

%{\color{orange}Convention Empty partition, N=1}
 \begin{coro} 
%  For any $\a\in \ZD,$ set 
% \[|\a|_L=\sum_{i=1}^{N-1} \log(1+\omega_i(\a)).\]
 For any $0\le s'<s,$
  \begin{equation}
 K_{s',s}=\sup_{N\ge 1}\sum_{\a\in \ZD/\Z} \frac{ \prod_{i=1}^{N-1}(1+\omega_i(\a))^{s'v_i}}{d_\a^s} <+\infty.
 \end{equation}
% \begin{equation}
% K_{s',s}=\sup_N\sum_{\a\in \ZD} \frac{N^{s' |\a|_L}}{d_\a^s} <+\infty.
% \end{equation}
 \end{coro}

 \begin{proof} Set  $u=s-s'$ and fix an integer $k>\frac 1u.$  Since for any $i, v_i\ge \max(\log(N),i),$ for any  $N\ge e^k,$
 \[ \sum_{\a\in \ZD/\Z} \frac{ \prod_{i=1}^{N-1}(1+\omega_i(\a))^{s'v_i}}{d_\a^s}\le \sum_{\omega\in \N^{N-1}}\prod_{i=1}^{N-1}(1+\omega_i(\a))^{-u v_i}=\prod_{i=1}^{N-1}\zeta( u v_i)\le  K_{s',s} \]
% with  $K_{s',s}=\zeta(k u)^{{\color{orange}2}k} \prod_{j\ge k}^{+\infty}\zeta( j u )^2.$ Since $\zeta(ju)-1\sim_{j\to\infty}2^{- ju},$  $K_{s',s}<\infty.$
 with  $K_{s',s}=\zeta(k u)^{{2}k} \prod_{j\ge k}^{+\infty}\zeta( j u )^2.$ Since $\zeta(ju)-1\sim_{j\to\infty}2^{- ju},$  $K_{s',s}<\infty.$
 \end{proof}
Since $v_i\ge \log(N)$ for all $i,$ $d_{\a}\ge N$ for any non constant $\a\in \ZD/\Z ,$ and for any $s>0$
\[0\le \zeta_{SU(N)}(s)-1\le \frac{K_{0,\frac s 2}}{N^{\frac s 2}}.\]
In particular  
\begin{equation}
\lim_{N\to\infty}\zeta_{SU(N)}(s)=1.
\end{equation}

 \begin{coro} \label{Coro----MassLargeSize} For $s>0,N\ge 1$ define a measure on $\ZD/\Z$ setting 
  \begin{equation}
 \mu_{N,s}(\a)= d_\a^{-s}\quad\forall \a\in\ZD/\Z.
 \end{equation}
 Set for  any $l> 0 $ and $1\le k<\frac N2,$
  \[A_N(l)= \bigcup_{1\le  i\le N} \{\a\in\ZD/\Z:  1+\omega_i(\a) \ge e^{ l}\} \quad\text{and}\quad B_N(k)=\bigcup_{k\le i\le N-k}\{\a\in\ZD/\Z: \omega_i(\a)\ge 1\}.\]
% \[A_N(k)=\{\a\in\ZD: \sum_{i= 1}^{N-1}\log(1+\omega_i(\a))\ge  k\}\]
 For any $s>0$ and $  k\ge 2,$ there is a constant $K>0$ such that for all $N> k^4,$
\begin{equation}
\mu_{N,s}( A_N(k)\cup B_N(k) )\le K N^{- \frac{k s\log(2)}{8} }\label{eq---MassBoundZeta}
\end{equation}
and in particular  
\begin{equation}
\mu_{N,s}(\{\a: |\a| >2 k e^k\} )\le K N^{- \frac{k s\log(2)}{8}}. \label{eq---MassLargeSize}
\end{equation}
%and for all $ 1\le k< \frac N 2,$
%\begin{equation}
%\mu_{N,s}(\exists i\in[N]: k\le i\le N-k, \omega_i(\a)\ge 1)\le K N^{- \frac{k \log(2)}{4} }
%\end{equation}
 \end{coro}
  \begin{proof} For $k\ge 1,$ if $\a\in\ZD/\Z$ with $|\a|>2ke^k$, then either there is some $i$ with $\omega_i(\a)\ge e^k$, or  $\omega_i(\a)< e^k$ for all $i$ and  there are at  least  
  $2k$ indices  $j$ with $\omega_j(\a)\ge1.$ Therefore $\{\a:|\a|>2ke^k\}$ is included in $A_N(k)\cup B_N(k)$ and \eqref{eq---MassLargeSize} follows from 
  \eqref{eq---MassBoundZeta}.
  
Since $v_i\ge \log(N)$ for all $i,$ for all  $s>0, l\ge 1,$
\begin{equation}
\mu_{N}(A_N(l))\le  \sum_{\a\in A_N(l)} d_\a^{-\frac s 2} \prod_{i}(1+\omega_i(\a))^{-\frac {\log(N) s}{2}}\le e^{- \frac{\log(N)l s }{2}} K_{0,\frac s 2}=N^{-\frac{s }{2}l}.  \label{eq---Chernoff}
\end{equation}
To conclude  for $N\ge k^4,$ for all $i $  with $k\le i\le N-k,$ $v_i \ge  k \frac{\log(N)}{4}. $ Indeed for $k\le i\le \sqrt{N}, $ $v_i\ge \frac{i \log(N)}{2}\ge \frac{k\log(N)}{2},$ while 
for $\sqrt{N}<i\le \frac N 2,$ $v_i\ge i\ge k N^{\frac 1 4}\ge k\frac{\log(N)}4. $  As in the previous bound, 
\begin{equation}
\mu_{N}(B_N(k))\le 2^{- \frac{\log(N)k s }{8}} K_{0,\frac s 2}=N^{-\frac{s \log(2) }{8}k} K_{0,\frac s 2}. \label{eq---ChernoffLengthPartitions}
\end{equation}
    \end{proof}
With a little more care, it is also possible to deduce the following from Lemma  \ref{Lem---BoundGLM}.

 \begin{lem} For all  $s>\log(N),$
 \[\zeta_{SU(N)} (s)<\infty.\]
 
 \end{lem}

\begin{rmk} { The bound  \eqref{eq---MassLargeSize} can also be improved using another lower bound on $d_\a.$  It is shown in \cite{MageeDelaSalle} that there is $c>0$ such that for all $N\ge 1,$
 \begin{equation}
 d_{\a}\ge e^{c \min(N,|\a|)}.
 \end{equation}
 Hence  for all $0<s'<s,$
 \[\mu_{N,s}(e^{ \frac{c}{s-s'}\min(N,|\a|)})\le \zeta_{SU(N)}(s').\]}

\end{rmk}

 \section{Unitary IRF Formulas}
 
 \label{sec---IRF}
 
The aim of this section is to prove  Theorem \ref{thm---IRFformula}, giving an exact value for $\mathcal{I}_{\a,\mathcal{S}}$ using a variant of the approach of \cite{IRFThierry}.    While \cite{IRFThierry} relies  on Schur-Weyl duality and Vershik-Okounkov  approach 
\cite{OkounkovVershik} to representations of the symmetric group to identify $\mathcal{I}_{\a,\mathcal{S}}$, our argument uses instead the eigenvalue decomposition of the Laplacian acting on tensors similarly to \cite{Naz}.     
 
\subsection{Spin networks integration}

Let us first note a general formula for a large class of observables of gauge configurations that only relies on Peter-Weyl theorem  and decomposition of tensor products into irreps.  

\subsubsection{Spin Networks and Peter-Weyl orthogonality theorem} The original notion of spin network is due to R. Penrose \cite{Penrose}. The definition we use below is close to  the one of \cite{LevSpin}; we also point the Reader to  \cite{LabourieLecture} for an introduction. For sake of simplicity,  the presentation below is rather pedestrian; a more elegant  but more abstract one could be to consider {semi-simple, monoidal categories \cite{TuraevQInv}. }

Consider a compact group $G$ and its set  $\mathcal{R}(G)$  of finite dimensional  representations up to isomorphism. For each $\mu\in\mathcal{R}(G),$ fix  $(V_\mu,\rho_\mu)$ a representation in  its  class,  denote by $\chi_\mu$ its character, $d_\mu$ its dimension and  by $\mu^*$ the class of its dual. 
\begin{defn}
A spin network is the data of a finite graph  $\mathcal{G}=(V,E)$ together with
\begin{itemize}
\item a map $\mu:E_o\to \mathcal{R}(G)$ with  
\[\mu(e^{-1})=\mu(e)^*\quad \forall e \in E_o,\]
\item a collection of tensors $(I_v)_{v\in V}$ with 
\begin{equation}
I_v\in \bigotimes_{e\in E_o: \underline e= v} V_{\mu(e)}\quad \forall v\in V.\label{contractionT}
\end{equation}
\end{itemize}
\end{defn}
We shall say that a map satisfying the first condition is skew symmetric. Denote by $\hat{G}$ the set of {irreducible} representations up to isomorphism. We shall call  the spin network $(\mu,I)$ is \emph{irreducible}, when 
\begin{equation*}
\mu(e)\in\hat{G}\quad\forall e\in E_o.
\end{equation*} 
From the first map, we can build   two  tensor spaces dual to one another: 
\[T_{V,\mu}=\bigotimes_{v\in V} \left(\bigotimes_{e\in E_o: \underline e=v}V_{\mu( e)}  \right) \]
and 
%\[T_E=\bigotimes_{e\in E} (V_{\mu( e)}^*\otimes V_{{\mu( e)}})\]
\[T_{E,\mu}=\bigotimes_{e\in E_o} V_{\mu( e)}^*= \bigotimes_{e\in E_o} V_{\mu( e^{-1})}\]
paired by the bilinear map $\<\cdot,\cdot\>_{E,V} $ satisfying
%\[ \<\varphi,t\>_{E,V}=\prod_{v\in V} \prod_{e\in E_o:\underline e=v}\<S_{e,},t_{v,e}\>,\]
%for any $(w,t)=(\otimes_{e\in E} (S_{e,\underline e}\otimes S_{e,\overline e}),\otimes_{v\in V}(\otimes_{e\in E_o: \underline e=v} t_{v,e}))\in T_E\times T_V.$
\[ \<\varphi,t\>_{E,V}=\prod_{v\in V} \prod_{e\in E_o:\underline e=v}\<\varphi_e,t_{v,e}\>,\]
for any $(\varphi,t)=(\otimes_{e\in E_o}\varphi_{e},\otimes_{v\in V}(\otimes_{e\in E_o: \underline e=v} t_{v,e}))\in T_{E,\mu}\times T_{V,\mu}$.  When there is no ambiguity, we shall drop the second index $\mu$ and write simply $T_V$ and $T_E.$
Moreover, for any choice of edges orientation,  there is  a $G$-isomorphism between  $T_E$ and $\bigotimes_{e\in E}\Hom(V_{\mu(e)},V_{\mu(e)})\simeq \bigotimes_{e\in E} V^*_{\mu(e)}\otimes V_{\mu(e)}. $

From the  tensor  data \eqref{contractionT}, identifying $T_{E,\mu}$ with $\bigotimes_{e\in E}\Hom(V_{\mu(e)},V_{\mu(e)})$,  we define the spin network map
\begin{align*}
S_{\mu,I}:G_{0}^{E_o}&\to \C\\
h&\mapsto \<\otimes_{e\in E} \rho_{\mu(e)}(h_e),\otimes_{v\in V}I_v\>_{E,V}.
\end{align*}
Note that $S_{\mu,I}$ does not depend on the choice of orientation made to identify $T_{E,\mu}$ with  $\bigotimes_{e\in E}\Hom(V_{\mu(e)},V_{\mu(e)})$.

\begin{ex} Consider  $\mu\in\hat{G}$ and   a  {non-rooted loop} $\ell$  in the graph $\mathcal{G}$ using each edge at most once.    Then  $W_{\ell, \mu}:G_{0}^{E_o}\to\C, h\mapsto \chi_\mu(h_\ell)$ is an irreducible spin network.
\end{ex}

Consider a random variable $h$ in $G_o^{E_o}$ distributed according to the Haar measure and   $\mu\in \mathcal{R}(G).$     By invariance of the Haar measure, when $g$ is Haar distributed,  $\EE[\rho_\mu(g)]=P_{\mu}$ is the  equivariant projection on $G$-invariants 
$V_\mu^G$ of $V_{\mu},$   yielding by linearity the following first moment of spin networks.

\begin{lem}   \label{ESN} Assume $h$ is a Haar distributed random variables on $\mathcal{M}(\mathcal{G},G).$  For any spin network $(\mu,I),$ 
\[\EE_{\mathcal{G},G}(S_{\mu,I})=\<\otimes_{e\in E} P_{\mu(e)},\otimes_{v\in V} I_v\>_{E,V}=S_{\mu_0,I_0}(\id_{\mathcal{G},G}),\]
where for any $e\in E_o, $ $\mu_0(e)$ is the trivial representation with multiplicity $\dim(V_{\mu(e)}^G)$ and for any 
$v\in V,$  $I_{o,v}$ is the equivariant\footnote{for the  action of $G^{\{e\in E_o:\underline{e}=v\}}$ on $\otimes_{e\in E_o:\underline{e}=v} V_{\mu(e)}.  $} projection of $I_v$ onto $\bigotimes_{e\in E_o: \underline e=v}V_{\mu( e)}  ^G.$ 
\end{lem}
 
The orthogonality of coefficients of irreps yields in turn a second moment identity.  {For any $\mu,\nu \in\hat{G},$    when $g$ is Haar distributed, Peter-Weyl orthogonality yields
 
\begin{equation}
\EE[\<\varphi, \rho_{\mu}(g)x\>  \<\psi, \rho_{\nu}(g)y\>]=\EE[\<\varphi, \rho_{\mu}(g)x\>  \<\rho_{\nu^*}(g)\psi, y\>]=\delta_{\mu,\nu} \frac{\<\varphi, y\>\<\psi,x\> }{d_{\mu}} \label{eq---PWHaar}
\end{equation}}
for all $\varphi\in V_{\mu^*},\psi \in V_\nu^*, x\in V_\mu,y\in V_\nu.$ Consequently, the equivariant projection $P_{\mu,\nu^*}$ of $V_{\mu}\ts V_{\nu}^*$ onto $ (V_{\mu}\ts V_{\nu}^*)^{G}$ is 

\begin{equation}
P_{\mu,\nu^*}=\EE[\rho_{\mu}(g)\ts \rho_{\nu^*}(g)]= \frac{\delta_{\mu,\nu}}{d_{\mu}}  \<1\,2 \>_{\mu}\label{eq---PeterWeylProj}
\end{equation}
where  for any $x\in V_\mu,\varphi\in V_\mu^*, $ identifying  $\Id_{V\mu}$ with an element of $V_\mu\ts V_\mu^*,$ 
\begin{equation}
\<1\, 2\>_\mu (x\ts \varphi)= \<\varphi,x\> \Id_{V_\mu}.
\end{equation}
We shall use next the Peter-Weyl orthogonality in this last form together with Lemma \ref{ESN}. Before doing so, let us recall that is also implies the following orthogonality relation of irreducible spin networks.

Consider the natural bilinear pairing  $\<\cdot,\cdot\>_V $ of   $T_{V,\mu^*}=\bigotimes_{v\in V} \left(\bigotimes_{e\in E_o: \underline e=v}V_{\mu( e)}^*  \right) $  with $T_{V,\mu}$ setting 
\begin{equation}
 \<\varphi, t\>_{V}= \prod_{v\in V, e\in E_o: \underline e=v}\<\varphi_{v,e},t_{v,e} \>
\end{equation}
for any $(\varphi,t)=(\otimes_{v\in V}(\otimes_{e\in E_o: \underline e=v} \varphi_{v,e}),\otimes_{v\in V}(\otimes_{e\in E_o: \underline e=v} t_{v,e}))\in T_V^*\times T_V.$ 
The orthogonality \eqref{eq---PWHaar} then yields the following.

\begin{thm}[Peter-Weyl] \label{thm: OrthoSN}When $(\mu,I),(\nu,J)$ are irreducible spin networks,
\[\EE_{\mathcal{G},G}[S_{\mu,I}S_{\nu,J}]=\delta_{\mu,\nu^*} \frac{\<I,J\>_V}{\prod_{e\in E} d_{\mu(e)}}.\]
\end{thm} 
In the above lemma, note that $\mu=\nu^*$ implies  $I\in  \bigotimes_{v\in V} \left(\bigotimes_{e\in E_o: \underline e=v}V_{\nu( e)}^*  \right)$ so that $\<I,J\>_V$ is well defined.

\subsubsection{Tensor product decompositions}  Let us fix  a few notations relative to the decomposition of the tensor product of two irreps.  For any $\a,\b \in \hat{G},$   consider the isotypic decomposition
\begin{equation}
V_\a\otimes V_\b=  \bigoplus_{\g\in \hat{G}}  (V_\a\otimes V_\b)^{\g}.\label{Isotop Tens}
\end{equation}
For $\g\in \hat{G},$   the associated projection on  $(V_\a\otimes V_\b)^{\g}$ is
\[P^{\a,\b}_\g=d_\g\EE[\chi_\g(g^{-1})\rho_\a(g)\otimes \rho_\b(g) ]\] 
and  we write
\begin{equation}
(t)^\g=P^{\a,\b}_\g(t)\quad \forall t\in V_\a\otimes V_\b.
\end{equation}
We shall denote $P_{\a,\b}^\g\in \Hom_G((V_\a\ts V_\b)^\g,V_\a\ts V_\b)$ its right inverse.

We also use the equivariant identification
\begin{equation}
(V_\a\otimes V_\b)^\g\simeq H^\g_{\a,\b}\otimes V_\g,\label{eq---TensProdIsoT}
\end{equation}
where $H^\g_{\a,\b}=\Hom_G(V_\g,V_\a\otimes V_\b )$ and  in the right-hand side,  $G$ acts  only\footnote{$H_{\a,\b}^\g,$ called multiplicity module,  is the space of multiplicities of the occurrence of $V_\g$ in the isotypic decomposition of $V_\a\otimes V_\b.$ } on  the right tensor, that is, for any $g\in G,$ $h\in H^{\g}_{\a,\b}$ and $v\in V_\g,$
\[g. (h\otimes v)= h\otimes (\rho_\g(g).v).\]
Using this identification, we define a bilinear map  $(V_\a\otimes V_\b)  \times V_\g^*\to H_{\a,\b}^\g$ setting for any  $\varphi\in V_\g^*,$ 
\begin{equation}
\<\varphi, t\>=\left\{\begin{array}{ll}\<\varphi,v\> h& \text{ if }t= h\otimes v\in (V_\a\otimes V_\beta)^\g,\\&\\0&\text{ if  }(t)^\g=0.\end{array}\right.\label{Pairing Isotop}
\end{equation}
Note that the vector space  $H^{\g^*}_{\b^*,\a^*}$ is naturally dual to $H^{\g}_{\a,\b}$ via the pairing defined by 
\begin{equation}
\< \varphi,\psi\>_{H^{\g^*}_{\b^*,\a^*},H^{\g}_{\a,\b}}= \sum_{i=1}^{d_\g}\<\varphi(\varepsilon^i),\psi(e_i)\>_{V_{\b}^* \ts  V_{\a}^*, V_\a\ts V_\b }\label{eq---DualMultTens}
\end{equation}
where  $(e_i)_i$ is a basis of $V_\g$ and $(\varepsilon^i)_i$  is  its dual,   the latter pairing being independent of the choice of basis.  
Using the identification \eqref{eq---TensProdIsoT}, we shall also identify $P^{\a,\b}_\g$and $P_{\a,\b}^\g$ with  elements of $\Hom(V_{\a}\ts V_{\b}, H^{\gamma}_{\a,\b}\ts V_\g)$ and $\Hom(V_{\a}\ts V_{\b}, H^{\gamma}_{\a,\b}\ts V_\g)$ that we shall represent as in Figure \ref{Fig---TensorIsotop}.

  \begin{figure}\label{Fig---TensorIsotop}
 
{\centering  \includegraphics[width=8cm,height=4cm]{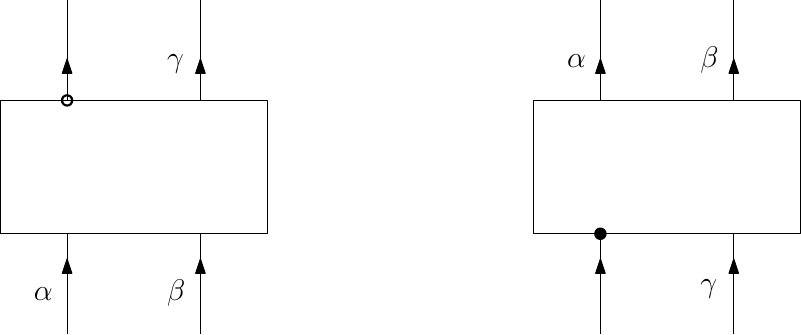}
}

 \caption{The  graphical representation of  the morphism $P^{\a,\b}_\g\in \Hom(V_{\a}\ts V_\b, H^{\g}_{\a,\b} \ts V_{\g})$ and $P^{\g}_{\a,\b}\in \Hom(H^{\g}_{\a,\b} \ts V_{\g},V_{\a}\ts V_\b).$ }
\end{figure}

When $\dim(H_{\a,\b}^\g)=1$ and that a choice of vector in $H_{\a,\b}^\g$ has been fixed, we shall identify $P_{\a,\b}^\g,P_{\g}^{\a,\b}$ with elements of $\Hom(V_{\a}\ts V_\b, V_\g)$ and   $\Hom(V_\g,V_{\a}\ts V_\b)$ that we represent as in Figure  \ref{Fig---TensorIsotopMFree}

  \begin{figure}\label{Fig---TensorIsotopMFree}

{\centering  \includegraphics[width=8cm,height=4cm]{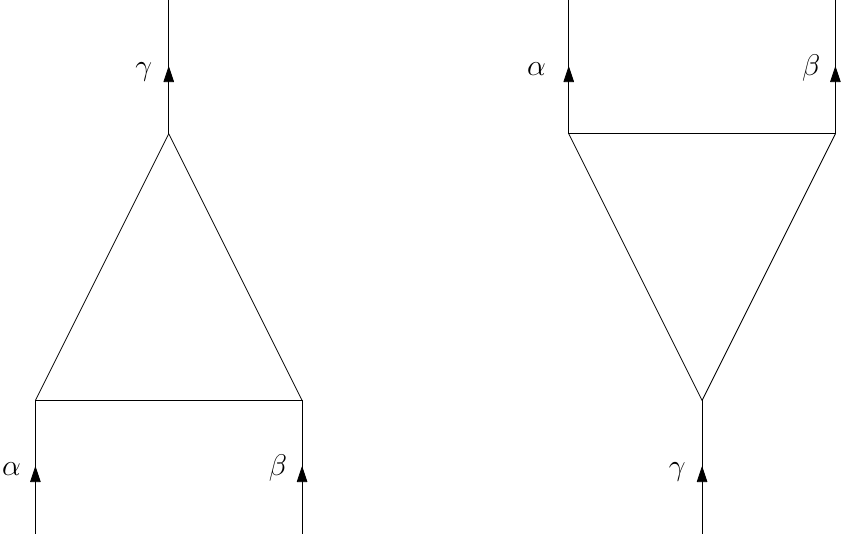}
}

 \caption{Graphical representation of  the morphisms $P_{\a,\b}^\g\in \Hom(V_{\a}\ts V_\b,  V_{\g})$ and $P_{\g}^{\a,\b}\in \Hom(  V_{\g},V_{\a}\ts V_\b)$ when $\dim(H^{\g}_{\a,\b})=1$ and an element of $H_{\a,\b}^{\g}\setminus\{0\}$ has been fixed. }
\end{figure}

\subsubsection{Interaction-round-the-face models and Wilson loop expectations} In this subsection we  fix a generalised map $\Gbb=(V,E,F),$  $\la: F\to\hat G.$ We consider the spin network
\[W_{\la}=\prod_{f\in F}W_{\partial f, \la(f)}:  G_{0}^{E_o}\to\C\]
and  an irreducible spin network   $(\mu, I)$ on the graph $\mathcal{G}=(V,E).$  We would like to give  an  expression for
\[\mathcal{I}_{I,\mu,\la}= \EE_{\mathcal{G},G}[S_{\mu,I} W_{\la} ] \]
using Lemma \ref{ESN} and the Peter-Weyl  formula \eqref{eq---PeterWeylProj}. 
% {\color{orange}using Theorem \ref{thm: OrthoSN} and the  decomposition of $W_{\la}$ into irreducible spin networks we shall explain below. }
 Before doing so, mind the following simpler situation.

\begin{lem} \label{eq---AdjCondSpinN} If $\exists e\in E_o$ with $\dim(\Hom_G(V_{\la(R(e))},V_{\la(L(e))}\otimes V_{\mu(e)}))=0$ then $\mathcal{I}_{I,\mu,\la}=0.$
 
\end{lem}
\begin{proof}The spin network $S_{\mu,I}W_{\la}$ is   a spin network of the form $S_{\nu,J},$ with $V_{\nu(e)}=\Hom(V_{\la(R(e))},V_{\la(L(e))}\otimes V_{\mu(e)})$ for all $e\in E_o,$  where  for any $e\in E_o,$ we make the identification
%\footnote{the corresponding operators are sometimes called {\color{orange}Wigner-Racah operators.}}
\begin{align*}
\Hom(V_{\la(R(e))},V_{\la(L(e))}\otimes V_{\mu(e)})^* \simeq\Hom(V_{\la(R(e^{-1}))},V_{\la(L(e^{-1}))}\otimes V_{\mu(e^{-1})}).
\end{align*}
The claim then follows from Lemma  \ref{ESN}.

\end{proof}
%\begin{proof} Since $S_{\mu,I}S_{\l}$ is a linear combination of spin networks $S_{\nu,J}$ where for all $e\in E_o,$ $\nu(e)=\Hom_G(V_{\l(L(e))}\otimes V_{\mu(e)}, V_{\l(R(e))}),$ the claim follows from Lemma  \ref{ESN}.
%\end{proof}
%decompose $W_{\la}$ and $S_{\mu,I}W_{\la}$ into irreducible components.  

Let us now fix further notations to rewrite $S_{\mu,I}W_{\la}$ more explicitly as a single spin network.

When $\la: F\to \hat{G} ,$  let  
\[\partial_l \la,\partial_r\la,\partial \la:E_0\to \mathcal{R}(G) \]
be the  maps with  
\[\partial_l \la(e)= \la(L(e)) ,\partial_r \la(e)= \la(R(e)) \text{ and } \partial \la(e)= \la(L(e))\otimes \la(R(e))^*\]
 for all $e\in E_o$ and define the tensor
\begin{equation}
J^\la=\ts_{v\in V} J^\la_v \in T_{V,\partial \la}\label{Face tensor}
\end{equation}
%\[J_l,J_r,J\in T_{V,\partial \la}\]
where   $J^\la_v\in \bigotimes_{e\in E_o: \underline e=v}V_{\partial \la( e)}  $  is  the image of the identify map through the identification of $\bigotimes_{e\in E_o: \underline e=v} \Hom (V_{\la (R(e))},V_{\la (R(e))}) $  with $\bigotimes_{e\in E_o: \underline e=v}V_{\partial \la( e)}.  $

When $\mu:E_o\to  \hat{G},\g:E_0\to \hat G$ are skew symmetric and $E_+$ is an orientation of $E$,  define    skew symmetric  maps 
\[\partial \la \# \mu,(\partial \la \# \mu)^{\g,l}:E_o\to  \mathcal{R}(G)\]
setting  for  all $e\in E_+,$
\begin{equation}
\partial \la \# \mu(e)=\partial_l\la(e)\otimes \mu(e)\otimes\partial_r\la(e)^* \
\end{equation}
and 
\begin{equation}
(\partial \la \# \mu)^{\g,l}(e)=(\partial_l\la(e)\otimes \mu(e))^{\g(e)}\otimes\partial_r\la(e)^*
\end{equation}
and  identifying  $(\partial_l\la(e)\otimes \mu(e)\otimes\partial_r\la(e)^*)^*$  with  $\partial_r\la(e)\otimes \mu(e^{-1})\otimes\partial_l\la(e)^*.$  Lastly, let us also define  two  bilinear maps from $T_{V,\partial \la}\times T_{V,\mu}$ to respectively  $T_{V,\partial \la\# \mu}$ and $T_{V,(\partial \la\# \mu)^{\g,l}}$,  mapping $(K,I)$  respectively to
\begin{equation}
K\#I =\bigotimes_{v\in V} \left(\otimes_{e:E_o:\underline e=v}\varphi_e\otimes t_{e}\otimes \psi_{e}\right)
\end{equation}
and 
\begin{equation}
(K\#I)^{\g,l} =\bigotimes_{v\in V} \left(\otimes_{e:E_o:\underline e=v} x_{e}\right)
\end{equation}
for any 
$K=\bigotimes_{v\in V} \left(\otimes_{e:E_o:\underline e=v}\varphi_e\otimes \psi_{e}\right)\in T_{V,\partial \la}$ and  $I=\otimes_{v\in V}(\otimes_{e\in E_o: \underline e=v} t_{e})\in T_{V,\mu},$ where  for each  $e\in E_+$ we identify $\left((\partial_l\la (e)\ts\mu(e))^{\g (e)}\ts \partial_r\la(e)^*\right)^* $ with $ \partial_r\la(e)\ts ( \mu(e^{-1})\ts \partial_l\la (e)^* )^{\g (e^{-1})} $ and set
\[x_e= \left\{\begin{array}{ll}(\varphi_e\otimes t_{e})^{\g(e) }\otimes \psi_{e}& \text{ if }e\in E_+,\\&\\\varphi_e\otimes (t_{e} \otimes \psi_{e})^{\g(e)}& \text{ if }e^{-1}\in E_+.\end{array}\right.\]
This is best understood graphically, for instance $J_\la\#I$ is drawn on Figure \ref{Fig---SNFacesAroundVertex}.

  \begin{figure}\label{Fig---SNFacesAroundVertex} 
%{\centering {\includegraphics[width=6cm,height=4 cm]{BrauerMAPAC.pdf}}
% \qquad \includegraphics[width=6cm,height=7 cm]{SurfaceAC.pdf}
%}

{\centering  \includegraphics[width=10cm,height=6cm]{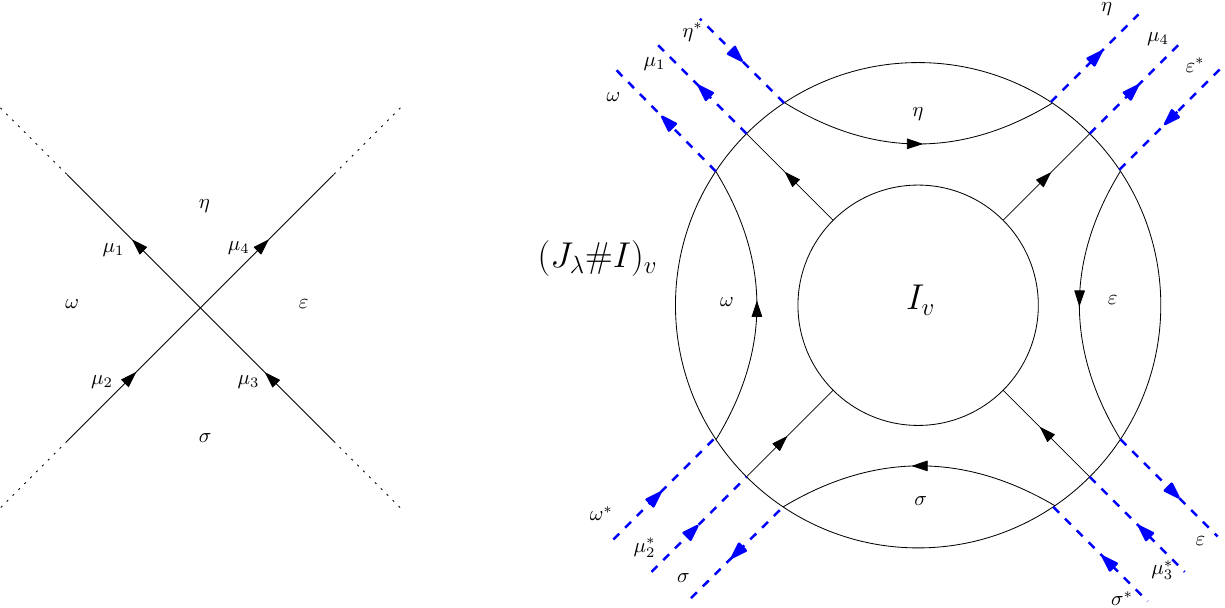}
}

 \caption{Left: the $\hat G$-labeled edges of a spin network near a vertex $v$ with  edges orientations fixed.  The graph is here the one of a map and values of $\la:F\to \hat G$ are drawn on each face near $v.$  Right: The outer circle represents the tensor $(J_\la\#I)_v$. }
\end{figure}
 
%( \partial_l\la (e)\ts\mu(e))^{\g (e)}

With these notations, when $\mu:E_0\to  \hat{G}$  is skew symmetric and $I\in T_{V,\mu},$  using \eqref{Isotop Tens} for each oriented edge of $E_+,$  by linearity
 \begin{equation}
 S_{\mu,I}W_\la=S_{ \partial \la\#\mu, J^\la\#I}=\sum_{\g:E_0\to \hat G:\g(e^{-1})=\g(e)^*}  S_{(\partial\la\#\mu)^{\g,l}, (J^\la\#I)^{\g,l}}.\label{eq---DecProductSpinN}
 \end{equation} 
We are now  almost ready to apply Lemma \ref{thm: OrthoSN} together with Peter-Weyl projection formula \eqref{eq---PeterWeylProj}; it remains to account for multiplicity spaces of $(\partial\la\#\mu)^{\g,l}.$   When $\mu:E_o\to  \hat{G}$  and $\la:F\to\hat G$ are fixed as above, define a skew symmetric map using the pairing \eqref{eq---DualMultTens}, setting for  any $e\in E_o,$ 
\[H_{e,\g}= \left\{\begin{array}{ll} H_{\la(L(e)) ,\mu(e)}^{\g(e)} & \text{ if } e\in E_+,\\&\\H_{\mu(e), \la (R(e))}^{\g(e)} & \text{ if } e^{-1}\in E_+.\end{array}\right.\] 

Let us define bilinear maps  on $T_{V,\partial \la} \times T_{V,\mu}$ setting for all $e\in E_o,$ first
 \[\<K\#I\>^{\g,l}_e = \left\{\begin{array}{ll}\<\psi_e, (\varphi_e\otimes t_{e})^{\g(e)} \> & \text{ if }e\in E_+,\\&\\\<(t_{e} \otimes \psi_e)^{\g(e)},\varphi_e\>& \text{ if }e^{-1}\in E_+\end{array}\right.\]
 and     
\begin{equation}
\<K\#I\>^{\g,l}_{v} =\otimes_{e\in E_o:\underline e=v}  \<K\#I\>^{\g,l}_e   \in \bigotimes_{e\in E_o:\underline e= v} H_{e,\g}
\end{equation}
 using the pairing \eqref{Pairing Isotop}, and  then using \eqref{eq---DualMultTens},
% \begin{equation}
%\<K\#I\>^{\g,l} =\prod_{e\in E_+} \big\< \<K\#I\>^{\g,l}_{\overline e},\<K\#I\>^{\g,l}_{\underline e}\big\>_{H_{e,\g}^*,{H_{e,\g}}},\end{equation}
 \begin{equation}
\<K\#I\>^{\g,l} =\prod_{e\in E_+} \big\< \<K\#I\>^{\g,l}_{e^{-1}},\<K\#I\>^{\g,l}_{ e}\big\>_{H_{e^{-1},\g},{H_{e,\g}}}\in\C\end{equation} 
for any 
$K=\bigotimes_{v\in V} \left(\otimes_{e:E_o:\underline e=v}\varphi_e\otimes \psi_{e}\right)\in T_{V,\partial \la}$ and  $I=\bigotimes_{v\in V}(\otimes_{e\in E_o: \underline e=v} t_{e})\in T_{V,\mu}.$
%  When $K=\bigotimes_{v\in V} \left(\otimes_{e:E_o:\underline e=v}(k_{v,l}\otimes k_{v,r})\right)\in T_{V,\partial \l}$ and  
% $I=\otimes_{v\in V}(\otimes_{e\in E_o: \underline e=v} t_{v,e})\in T_{V,\mu}$, we set 
% \begin{equation}
% \<K\#I\>^{\g,l}= \prod_{v\in V,e:E_o:\underline e=v}\<P^{\partial_l \la(e),\partial_r\la(e)}_\g (k_{v,l}\otimes t_{v,e}),k_{v,r}\>
% \end{equation}

\begin{lem} \label{Gen Int formula SN compact form}For any $\mu: E_o\to \mathcal{R}(G) $ skew symmetric, $\la:F\to \hat G$ and $I\in T_{V,\mu},$ 
\[\EE[S_{\mu,I}W_\la]= \frac{1}{\prod_{e\in E_+}d_{\partial_r\la(e)}}\<J^\la\#I\>^{\partial_r\la,l}. \]
\end{lem}

{\begin{proof} For any skew symmetric $\g:E_o\to\hat{G},$  for all $e\in E_+$, identifying  $(\partial \la \# \mu)^{\g,l}(e)$ with $H_{e,\g}\ts V_{\g(e)} \ts V_{\partial_r\la(e)^* }, $ the projection on $G$-invariants of $(\partial \la \# \mu)^{\g,l}(e)$ is
\[ P_{(\partial\la\#\mu)^{\g,l}(e)}= \Id_{H_{e,\g}}\ts P_{\g(e), \partial_r\la(e)^*}.\]
Together with  Lemma \ref{thm: OrthoSN}  and \eqref{eq---PeterWeylProj}, we find
\begin{align*}
\EE[S_{(\partial\la\#\mu)^{\g,l}, (J^\la\#I)^{\g,l}}] &= \frac{ \delta_{\partial_r\la,\g}}{\prod_{e\in E_+} d_{\partial_r \la(e)}} \<\otimes_{e\in E_+} \left(  \Id_{H_{e,\g}}\ts \<1\, 2\>_{\partial_r\la(e)}\right),  (J^\la\#I)^{\g,l} \>_{E,V} \\
&= \frac{ \delta_{\partial_r\la,\g}}{\prod_{e\in E_+} d_{\partial_r \la(e)}} \<J^\la\#I\>^{\partial_r\la,l}.
\end{align*}
Summing over $\g:E_0\to \hat G,$  we conclude thanks to \eqref{eq---DecProductSpinN}. This computation is illustrated in Figure \ref{Fig---PWAroundVertex}.

  \begin{figure}\label{Fig---PWAroundVertex}
%{\centering {\includegraphics[width=6cm,height=4 cm]{BrauerMAPAC.pdf}}
% \qquad \includegraphics[width=6cm,height=7 cm]{SurfaceAC.pdf}
%}

{\centering  \includegraphics[width=6cm,height=6cm]{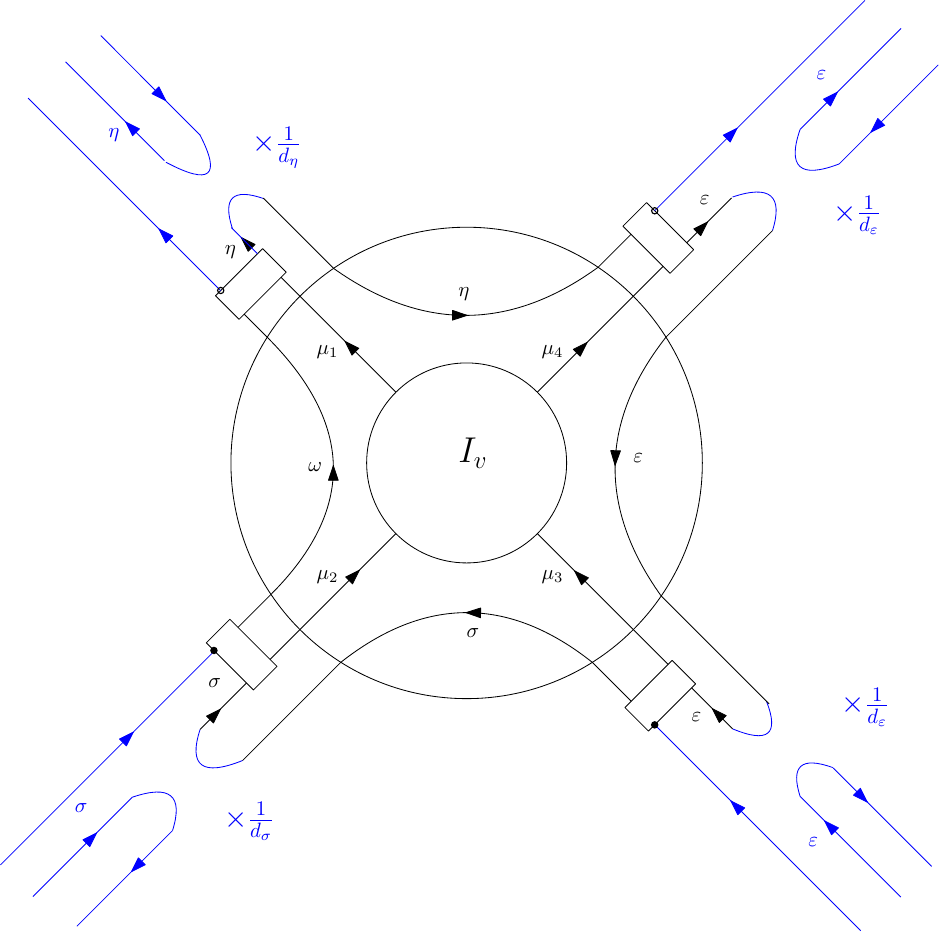}
}

 \caption{Graphical representation of the computation of Lemma \ref{Gen Int formula SN compact form}'s proof near a vertex $v$, with same notations as in Figure \ref{Fig---SNFacesAroundVertex}. Blue parts represent contributions of the integration of the  Haar measure along each edge of the graph.    }
\end{figure} 
\end{proof}}

Let us improve slightly this last formula trying to write the contraction in the right-hand side as a product over vertices.  For  $\mu: E_o\to \hat G $ skew symmetric and $\la:F\to \hat G$,  set   for all $e\in E_+,$ $m_{e} =\dim ( H_{\partial_l \la(e),\mu(e)}^{\partial_r\la(e)}),$  
fix  a basis $(\varepsilon_{i,e})_{1\le i\le m(e) }$ of $(H_{\partial_l \la (e),\mu(e)}^{\partial_r\la(e)})^*$ and denote by  $(\varepsilon_{i,e^{-1}})_{1\le i\le m(e) }$ the dual basis of  $(H_{\mu(e^{-1}), \partial_r \la (e^{-1})}^{\partial_l\la(e^{-1})})^*$. Expanding  $\<J^\la\#I\>^{\g,l}_e $ in the basis of $H_{e,\g}$ dual to  these for all  $e\in E_o$  yields the following expression.

\begin{lem} \label{lem: Spin Face General} For any $\mu: E_o\to \hat G $ skew symmetric, $\la:F\to \hat G$ and $I\in T_{V,\mu},$ 
\[\EE_{\mathcal{G},G}[S_{\mu,I}W_\la]= \frac{1}{\prod_{e\in E_+}d_{\partial_r\la(e)}}\sum_{\substack{i:E_o\to \N\\ 1\le i(e)\le m(e), i(e)=i(e^{-1}), \forall e\in E_o}}\prod_{v\in V}\big\<\otimes_{e\in E_o:\underline e=v} \varepsilon_{i(e),e}, \<J^\la\#I\>_v^{\partial_r\la,l}\big\>. \]
\end{lem}
Mind that the last choice of basis does not change the summand  in the right-hand side as $\varepsilon_{i(e),e}$ always comes therein with its dual $\varepsilon_{i(e),e^{-1}}$. 

When  $V_{\la(L(e))}\otimes V_{\mu(e)}$ is multiplicity free for each $e\in E_+$, this amounts to choose    for each $e\in E_+$  an element $h_e\in H_{\partial_l\la(e),\mu(e)}^{\partial_r\la(e)}\setminus\{0\}.$  Applying the last lemma under this assumption yields the next Lemma, where for each such choice we identify $\bigotimes_{e\in E_o:\underline e= v} H_{e,\g}$ with $\C$ for each vertex $v.$ In this case the computation can be illustrated by Figure \ref{Fig---PWAroundVertexMFree}.

  \begin{figure}\label{Fig---PWAroundVertexMFree}
%{\centering {\includegraphics[width=6cm,height=4 cm]{BrauerMAPAC.pdf}}
% \qquad \includegraphics[width=6cm,height=7 cm]{SurfaceAC.pdf}
%}

{\centering  \includegraphics[width=6cm,height=6cm]{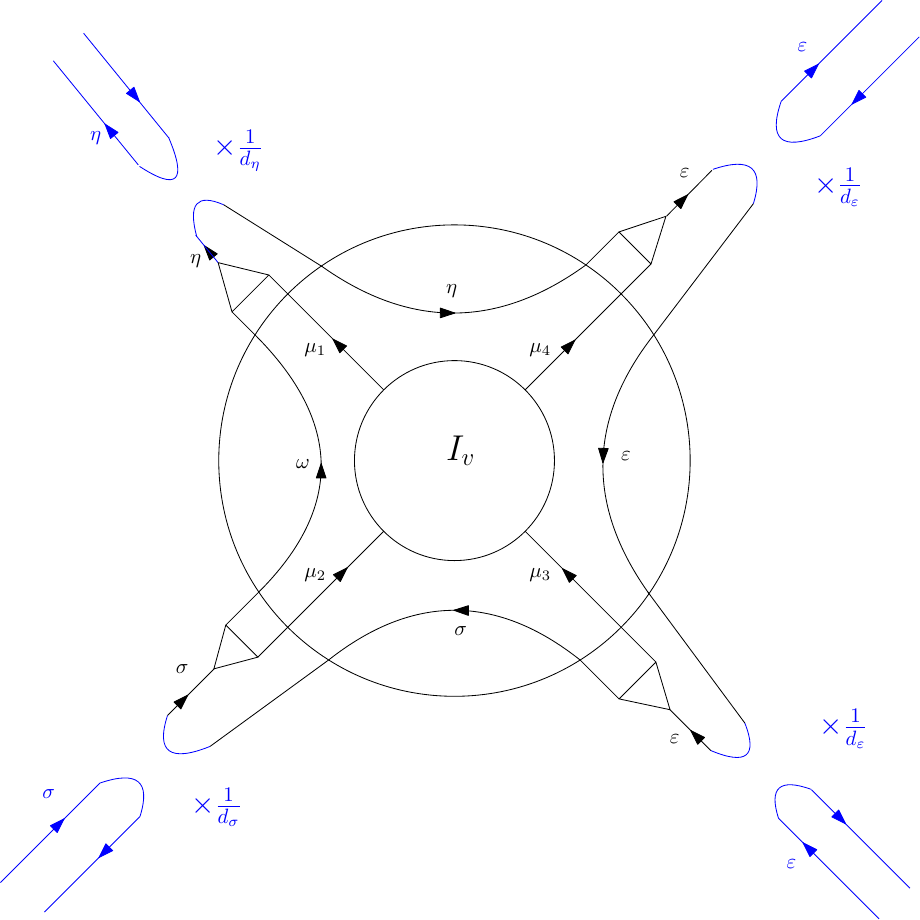}
}

 \caption{Graphical representation of the computation of Lemma \ref{Gen Int formula SN compact form}'s proof near a vertex $v$, with same notations as in Figure \ref{Fig---SNFacesAroundVertex}, in the multiplicity free case of Lemma \ref{LEM---SpinNCharacFaces}. Blue parts represent contributions of the integration of the  Haar measure along each edge of the graph.}
\end{figure}

%When $\a,\b \in  \hat{G}$ are such that the decomposition  $V_{\a}\otimes V_\beta$ is multiplicity free,  let us make an identification $(V_\a\otimes V_\b)^\g\simeq V_\g$ and $H_{\a,\b}^\g\simeq \C,$  for any $\g\in \hat{G}$ contributing to the decomposition, which 
%then amounts to a choice of a non-zero element in the 
%line $H_{\a,\b}^\g.$   These identifications are unique up to a scalar, but these choices do not  change the right-hand side of Lemma \ref{lem: Spin Face General}. With this latter convention, the last expression can be simplified as follows in the multiplicity-free case.
 
\begin{lem} \label{LEM---SpinNCharacFaces} Assume  that for any $e\in E_+,$ $V_{\la(L(e))}\otimes V_{\mu(e)}$ is multiplicity free. Then for any choice of $h_e\in H_{\partial_l\la(e),\mu(e)}^{\partial_r\la(e)}$ for all $e\in E_+,$ 
\[\EE[S_{\mu,I}W_\la]=  \frac{1}{\prod_{e\in E_+}d_{\partial_r\la(e)}}\prod_{v\in V}   \<J^\la\#I\>_v^{\partial_r\la,l} . \]
 \end{lem}

\begin{rmk} Note that when $G$ is $U(N)$ or $SU(N)$ and $V=\C^N$ is the standard representation, then by Pieri's rule,  $V_\a\otimes V$ is multiplicity free for any $\a\in \hat G$. 
\end{rmk}

\subsection{Wilson loops expectations for the standard representation}  The last formula in the unitary and special unitary cases can be made fully explicit in the case of a Wilson loop in the standard representation. It implies in particular the following.

Assume  $\Gbb=(V,E,F)$ is a generalised map with edge orientation $E_+$ associated to a tuple $\mathcal{S}$ of loops with simple, transverse intersections as defined in section \ref{sec---GeneralisedMap}.

When  $\mu: E_o\to \mathcal{R}(G)$ is skew symmetric, constant equal to the class of  $V_\sigma=\C^N$ on $E_+$,  since $\mathcal{G}=(V,E)$ is $4$-regular with two inwards and two outwards edges of $E_+$ at each vertex, we can and do make the identification
\[\bigotimes_{e\in E_o: \underline e= v} V_{\mu(e)}\simeq  \End(V_{\sigma}^{\otimes 2} ) \]
to  define the flip tensor $\tau_\mu\in T_{V,\mu}$ setting for any $u,w\in V,$ 
\begin{equation}
\tau_{\mu,v}:\begin{array}{ccc}V_{\sigma}^{\otimes 2}&\longrightarrow& V_{\sigma}^{\otimes 2}\\ u\otimes w&\longmapsto &w\otimes u.\end{array}
\end{equation}
so that 
 \begin{equation}
S_{\mu,\tau_\mu}=\prod_{\mfl\in \mathcal{S}} \chi_\sigma(h_{\mfl})=\prod_{\mfl\in \mathcal{S}} \Tr(h_{\mfl}).\label{eq---WilsonASSpinN}
\end{equation}
 
For all $v\in V,$ we  denote by $\mathfrak{s}_v,\mathfrak{w}_v,\mathfrak{n}_v,\mathfrak{e}_v$ the faces cyclically\footnote{mind that some of them might be identical} ordered around $v$ starting with the one bounded by two incoming edges of $E_+$ at $v.$ 

Let us assume from there onwards that $G=U(N)$ or $SU(N)$  and $V_\sigma=\C^N.$   In this case when $\la:F\to \hat G,$ the constraint \begin{equation}
\dim(\Hom_G(V_{\la(R(e))},V_{\la(L(e))}\otimes V_{\mu(e)}))>0,\forall e\in E_+\label{eq---AdjWeight}
\end{equation}
appearing in Lemma \ref{eq---AdjCondSpinN} takes a particularly nice form.  Indeed recall Pieri's rule: for any  $\a\in \ZD,$ 
\begin{equation}
V_\a\ts V= \bigoplus_{i} V_{\a+\omega_i}\quad\text{and}\quad V_\a\ts V^*=\bigoplus_{j} V_{\a-\omega_j}  
\end{equation}
where the sums  are respectively over $i$ and $j$ with  $1\le i,j\le N$ and\footnote{recall the definition for the canonical basis  $(\om_i)_i$ in section  \ref{sec---StatIRF}. } $\a+{\om_i},\a-\om_j\in\ZD. $  Therefore $\la:F\to \hat{G}$ satisfies \eqref{eq---AdjWeight} if and only if there is $i: E_+\to \{1,\ldots,N\}$ such that 
\begin{equation}
\partial\la(e)=\la(L(e))-\la(R(e))=-\om_{i(e)} \label{eq---diffWeight}
\end{equation}
or equivalently $\la\in \Omega_N$ or $\Omega_N^*$ when $G$ is resp. $U(N)$ or $SU(N)$ as defined in section \ref{sec---StatIRF}.  Let us give a third reformulation.  When the last condition is satisfied,  consider  for each    $k\in\{1,\ldots,N\}$ the co-exact  co-cyle 
\begin{equation}
\mathscr{C}_{k,\la}=-\partial \la _k\label{eq---DefCoExactCoCycle}
\end{equation} so that \begin{equation}
\mathscr{C}_{k,\la}= \sum_{e\in E_+:i(e)=k} e_+
\end{equation}
where  the right-hand side vanishes when the index set is empty. Let $\mathcal{S}_\la$ be the unique collection of simple labeled loops $(\mfl_1,i_1),\ldots,(\mfl_q,i_q)$ of $\Gbb$ with distinct edges obtained by  de-singularising some of the intersection points of $\mathcal{S}$ such that
\begin{enumerate}
\item for all $k$ with  $\mathscr{C}_{k,\la}\not=0,$
\[ \mathscr{C}_{k,\la}= \sum_{1\le l\le q: i_l =k} \om_{\mfl_l},\]
\item   two  loops of the same label cannot cross transversally,
\item   $i_1,\ldots,i_q\in\{1,\ldots,N\}.$ 
\end{enumerate}

A vertex $v\in V$  of $\Gbb_\mathcal{S}$ can then be of two types.  Call  $v$ a \emph{crossing point} if   loops of $\mathcal{S}_\la$ visiting  $v$ of $\Gbb$ cross transversally at $v,$ and  
\emph{contact point} otherwise.    When the two (oriented) loops ordered by their incoming edges at $v,$ have label $a$ and $ b$,  call $v$  an $(a,b)$-contact/crossing point.   By (1), for any $a,$ reading the  value of $\la_a$  along a $b$-loop,  say on its left, $\la_a$ changes by  respectively $-1$ or $ +1$ at each  $(a,b)$ or $(b,a)$-crossing point 
the loop visits and does not change otherwise.  In particular for $b\not=a,$ there are as many    $(a,b)$ as $(b,a)$-crossing points along each $b$-loop, and    all in all  as many $(a,b)$ as $(b,a)$ crossing points.   

%
%A vertex $v\in V$   can then be of two types.  Whether  two loops of $\mathcal{S}_\la$ visit the same vertex $v$ of $\Gbb$ without transversally crossing, we call $v$ a \emph{contact point}. Or the  loops transversally cross, we call $v$ a 
%\emph{crossing point}.  When the two loops ordered by their incoming edges at $v,$ have label $a,b$  call $v$  an $(a,b)$-contact/crossing point.   By (1), for any $a,$ reading the  value of $\la_a$  along a $b$-loop,  say on its left, $\la_a$ changes by  respectively $-1$ or $ +1$ at each  $(a,b)$ or $(b,a)$-crossing point 
%the loop visits and does not change otherwise.  In particular for $b\not=a,$ there are as many    $(a,b)$ as $(b,a)$-crossing points along each $b$-loop, and so there are all in all  as many $(a,b)$ as $(b,a)$ crossing points.   

%{\color{orange}\begin{rmk} With the same argument as above for loops of $\mathcal{S}$ in place of $\mathcal{S}_\la,$ there are many $(a,b)$ as $(b,a)$-contact  points.
%\end{rmk}}
{\begin{rmk} With the same argument as above for loops of $\mathcal{S}$ in place of $\mathcal{S}_\la,$ there are many $(a,b)$ as $(b,a)$-contact  points.
\end{rmk}}

%\begin{rmk} By assumption (2), loops of $\mathcal{S}_\la$ are simple.
%\end{rmk}
\begin{rmk}  The faces of $\Gbb_{\mathcal{S}_\la}=(V',E',F')$ have boundary  loops given by $\mathcal{S}_{\la}$ and in particular the orientation does not change along their boundaries, i.e.,  for any $f'\in F',$ $m_{f'}=0.$\label{Rk---CuttingLoopsNoAlternatingO}
\end{rmk}
\begin{rmk}\label{Rk---NSRep} A vertex is a contact point if and only if $\la(\mathfrak{n}_v)=\la(\mathfrak{s}_v). $  
\end{rmk}
Faces of $\Gbb_\mathcal{S}$ and $\Gbb_{\mathcal{S}_\la}$ can be  embedded in  the total surface $X_{\Gbb} $  of $\Gbb$ such that any face $f\in F$ is included in a face $f'$ of $\Gbb_{\mathcal{S}_\la}.$   Thanks to remark \ref{Rk---NSRep}, for a given $f'\in F',$ $\la(f)=\a$ is constant over all $f\in F$ included in $f'$  and we then set  $\la(f')=\a.$

%When $\mathcal{S}$ is a single loop The loops $\mathcal{S}_{\la}$ can be obtained from $\mathcal{\Sc}$ by swapping   $p-q$  intersection points of $\mathcal{S}$ 

%let $a$ and $b$ be the labels   of  respectively the  left and the right one\footnote{that is  bounding respectively positively and negatively the faces $\mathfrak{w}_v$ and $\mathfrak{e}_v.$ Equivalently, setting $i(e)=i(e^{-1})$ for all $e\in E_o,$ $(b,a,a,b)$ are the consecutive values of $i$ on the outgoing edges at $v.$    }
Let us introduce the main new ingredient of the next theorem. For each vertex, we shall see below that flip  tensors $\tau_{\mu_v}$ act like reflections\footnote{on a  multiplicity  plane or line.} defined by  angles $\theta_v$   we specify now.

 \begin{figure}
%{\centering {\includegraphics[width=6cm,height=4 cm]{BrauerMAPAC.pdf}}
% \qquad \includegraphics[width=6cm,height=7 cm]{SurfaceAC.pdf}
%}

{\centering  \includegraphics[width=9cm,height=9cm]{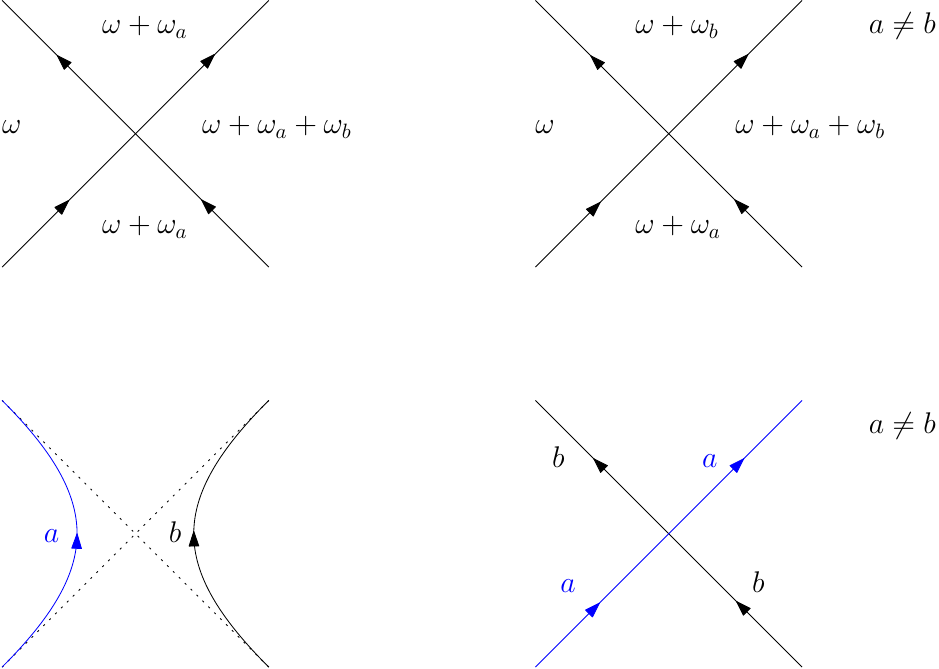}
}
 \caption{Top row: Loop $\mfl\in \mathcal{S}$ near an $(a,b)$-point together with the values of $\la:F_\mathcal{S}\to \hat{G}$ on its corners.  Bottom row:  the associated loops of $\mathcal{S}_\la$. Left column: contact point, $a,b $ are arbitrary. Right column: crossing point,  $a\not=b$.\label{Fig--CrossingContactPoints}}
\end{figure}

If $v$ is an $(a,b)$-point   as on { figure \ref{Fig--CrossingContactPoints}, writing $\la_v=\la(\mathfrak{s}_v), $ 
\begin{equation}
r_v=\la_{v,b}- \la_{v,a}+\rho_b-\rho_a+1\in\Z^*.
\end{equation}
Indeed  since $\la_v\in\ZD,$ $r_v  \ge 1$    if $b\le  a,$     and $r_v\le -2$ if $b>a+1,$ while when $b=a+1,$  since by    \eqref{eq---DefCoExactCoCycle},   $\la(\mathfrak{e}_v)=\la_v+\omega_{a+1}\in\ZD,$  
\[   r_v=  \la_{v,a+1}- \la_{v,a} =\la(\mathfrak{e}_v)_{a+1}-\la(\mathfrak{e}_v)_{a}-1\le -1. \]
Let   $\theta_v\in [0,\pi]$ be such that  
  \begin{equation}
  \cos(\theta_v)=\frac 1 {r_v}.
  \end{equation}
Note that $r_v$ and therefore $\theta_v$ only depend on $\la$ up to an additive constant, that is, only on the projection of  $\la$ onto $\ZD/\Z.$
%\begin{rmk}
%With the same argument as above,  for any integer $n\in\Z, $ along any loop  $\mfl\in\mathcal{S}_\la$  the  number of crossing points $v$  with $r_v=n$ visited by $\mfl$  is even.  In particular, for any $n$ the number of crossing points $v$ with $r_v=n$ is 
%even.
%\end{rmk}

\begin{rmk} An $(a,a)$-point $v$ is necessarily a contact point and $\theta_v=0.$ 
\end{rmk}

{\begin{rmk} For any  $b$-loop of $\mathcal{S}_\la$ and $a\not=b,$    the $\{a,b\}$-crossing points visited by the loop can be paired according to the value of $S_v=\partial_l\la(e)_a$  where $e$ is  the loop's  incoming edge at the crossing point.  Furthermore,  denoting by $e'$ the $b$-loop outgoing edge, the pairing can be chosen so that  for each paired crossing points $(v,w),$ $\partial_l\la(e')_a-\partial_l\la(e)_a$  takes opposite values at $v$ and $w$. For each such pair $(v,w)$ of crossing points, $r_v=-r_w$ and $\theta_v=\pi-\theta_w.$ In particular the number of crossing points is even and $S(\la)=\prod_{v\in V \text{crossing point}}\sin(\theta_v)$ is a rational function of $\la.$
 \end{rmk}}
 
% 
%If $v$ is an $(a,b)$-crossing point, 
%  
%\begin{equation}
%\cos(\theta_v)=\frac{1}{ (\la(\mathfrak{e}_v)+\rho)_b-(\la(\mathfrak{n}_v)+\rho)_a}=\frac{1}{ \la_{v,b}- \la_{v,a}+\rho_b-\rho_a+1}= \frac{2}{\|\la(\mathfrak{e}_v)+\rho\|^2-\|\la(\mathfrak{w}_v)+\rho\|^2+2}
%\end{equation}
%we set 
%%(\la(\mathfrak{e}_v)+\rho)_b-(\la(\mathfrak{n}_v)+\rho)_a
%
%\begin{equation}
%I
%\end{equation}
%by  \eqref{eq---DefCoExactCoCycle}, 
%\begin{equation}
%\la(\mathfrak{n}_v)=\la(\mathfrak{s}_v).
%\end{equation}
%
% either two loops denoting $e_{1,in}, e_{2,in}, e_{3,out},e_{4,out}in E_+$  the two incoming and outgoing edges at  $v$, since  
%
%% consider  the partition of   $E_+$  given by levels of $i:E_+\to \{1,\ldots,N\}$ and
%
%  associated to is the set of edges of $\mathcal{S},$  
% 
%
%
%

% let us identify $\hat{G}$ with $P_+=\{\l\in \mathbb{Z}^N:\l_1\ge \l_2\ge \ldots \l_N\}$ and set 
%{\color{orange}\[\mathcal{S}=\left\{\omega\in \mathbb{N}^N:\sum_{i}\omega_i=1\right\}.\]}
We can now state the main result of this section. The formula below is due to Thierry L\'evy in a slightly different form.\footnote{{The  expression of T. L\'evy is in terms of   Young diagrams contents instead of highest weights.
% and his argument uses the decomposition of Okounkov and Vershik \cite{OkVershik} of Coxeter transpositions  into a Gelfand-Zetlin basis,  whereas the argument below  uses the  action of  the Laplacian on $U(N)$ following similar ideas similar to  \cite{Naz} (the argument  of \cite{Naz} also allows to derive the same decomposition for Coxeter elements
 }} One of our motivation was to give a statement and a proof which is hopefully more easily adaptable  to other compact Lie groups.

\begin{thm}\label{THM---IRFOneSpinN} Assume $G=U(N)$ or $SU(N),$ $\mathcal{S}$ is a collection of loops with simple transverse intersections, with associated generalised map $\Gbb_{\mathcal{S}}=(V,E,F)$. Then for any  $\la:F\to\hat G,$    
\begin{equation}
\EE_{\mathcal{G}_\mathcal{S},G}[ W_\mathcal{S}W_\la] =\epsilon \frac{ \mathcal{D}_{\la,\mathcal{S}} \mathcal{I}_{\la,\mathcal{S}} 1_{\la\in \Omega_N}}{\prod_{f\in F} d_{\la(f)}^{\chi(f)}}\label{eq---WilsonLoopFaceCharac}
\end{equation}
 where 
\begin{equation}
   \mathcal{D}_{\la,\mathcal{S}}=\prod_{f\in F} d_{\la(f)}^{\chi(f)-\frac{m_f}2}= \prod_{f\in F'} d_{\la(f)}^{\chi(f)} \label{eq---DimEulerSwap}
\end{equation}
 with $F'$ the  set of faces of  $\Gbb_{\mathcal{S}_\la}$   and 
\begin{equation}
 \mathcal{I}_{\la,\mathcal{S}}=\prod_{v\in V\,\text{contact point}} \cos(\th_v) \prod_{v\in V\,\text{crossing point}} \sin(\th_v).
\end{equation}
\end{thm}

\begin{rmk} Equivalently to \eqref{eq---OSWAPF},  for any  face $f\in F,$   $\frac{m_f}2$ counts the  number of  south (or north) corners  within the face $f,$  that is, 
\[m_f= 2 \#\{v\in V: \mathfrak{s}_v=f\}=2 \#\{v\in V: \mathfrak{n}_v=f\}. \] 
\end{rmk}

This last expression is explicit enough to prove Theorem \ref{thm---IRFformula}.

\begin{proof}[Proof of Theorem \ref{thm---IRFformula}]  Under the assumption of the theorem, 
\[\zeta_{U(N)}(2g-2,q_T)\EE_{\YM_{U(N),\Gbb_\mathcal{S},a}}[W_{\mathcal{S}}]= \sum_{\la: F\to \ZD} \prod_{f\in F} d_{\la(f)}^{\chi(f)} \EE[ W_\mathcal{S} W_{\la} ]e^{-\<c_{\la},a\>}\]
and the formula follows from \eqref{eq---WilsonLoopFaceCharac}.  The special unitary and truncated cases follow verbatim from the same argument. 
\end{proof}
 
From Lemma \ref{LEM---SpinNCharacFaces},  the proof of Theorem \ref{THM---IRFOneSpinN}   can be completed using this last new algebraic input.

\begin{lem}\label{LEM---LemTraceVertex}
Assume $G=U(N)$ or $SU(N),$   $\Gbb=(V,E,F)$ is a generalised map as above,  $\la:F\to\hat G,$   $\mu :E_o\to\hat{G}$  is  skew symmetric with $V_{\mu(e)}=\C^N$ for all $e\in E_+$ and $I\in T_{V,\mu}$ is the flip tensor. Then  there is a choice of isomorphisms $H_{e,\partial_r\la(e)}\simeq\C$ for all $e\in E_+,$ such that for all $v\in V, $ 
\begin{equation}
\<J^\la\#\tau_\mu\>_v^{\partial_r\la,l}= \left\{\begin{array}{ll} d_{\la(\mathfrak{e}_v)}\cos(\theta_v) & \text{if}\,\,v\,\,\text{is a contact point,} \\ &\\   d_{\la(\mathfrak{e}_v)} \sin(\theta_v) & \text{if} \,\,v \,\,\text{is a crossing point.} \end{array}\right.\label{eq---IntertwinerContraction}
\end{equation}
\end{lem}

%\begin{lem}\label{lem---SignCoxeter} With the same assumption as in Lemma \ref{LEM---LemTraceVertex}, we can furthermore choose  the isomorphisms $H_{e,\partial_r\la(e)}\simeq\C$ for all $e\in E_+,$ such that \eqref{eq---IntertwinerContraction} holds true with $\<J^\la\#\tau_\mu\>_v^{\partial_r\la,l}\ge 0$ for any crossing point $v$.
% \end{lem}

 \begin{figure}
 
{\centering  \includegraphics[width=9cm,height=4.5cm]{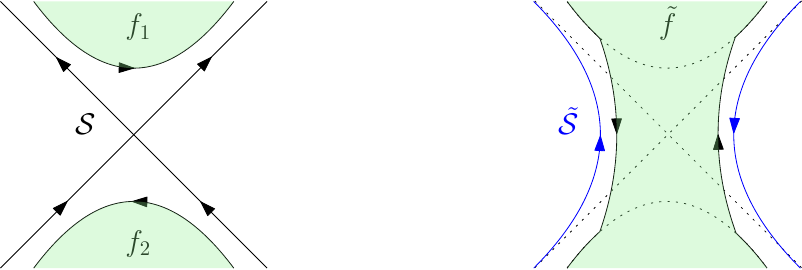}
}
 \caption{Merging operation at a crossing point of $\mathcal{S}$ used in Theorem \ref{THM---IRFOneSpinN}'s proof. \label{Fig--CrossingPMerge}  }
\end{figure}
\begin{proof}[Proof of Theorem \ref{THM---IRFOneSpinN}]  Let us consider first the second identity of \eqref{eq---DimEulerSwap}.  Assume $\la:F\to \hat{G}$ and  $f_1,f_2$ are two surfaces  with non-empty boundaries, given by faces of $\Gbb_\mathcal{S}$  meeting "north-south" at a contact point   $v\in V,$ that is, satisfying  $f_1=\mfn_v,f_2=\mfs_v.$

Desingularising  the intersection  at $v$ yields a  collection of loops $\tilde{\mathcal{S}}$ where $f_1,f_2$ are replaced by one surface $\tilde f$, obtained by gluing opposite sides of a rectangle along their boundaries as shown on {figure \ref{Fig--CrossingPMerge}}  so that  $\chi(\tilde f)=\chi(f_1\cup f_2)-1.$  Recalling \eqref{Rk---NSRep}, set  $\la(\tilde f)=\la(f_1)=\la(f_2).$ Either  $f_1=f_2$ and then   $m_{\tilde f}=m_{f_1}-2,$ or $f_1\not = f_2$ and  $m_{\tilde f}= m_{f_1}+m_{f_2}-2.$ In both 
cases
\[ \prod_{f\in F}  d_{\la(f)}^{\chi(f)-\frac{m_f}{2}}=d_{\la(\tilde f)}^{\chi(\tilde f)-\frac{m_{\tilde f}}{2} }    \prod_{f\in F\setminus\{ f_1,f_2\}} d_{\la(f)}^{\chi(f)-\frac{m_f}{2}}.  \]
Iterating this operation for all contact points of $V$ yields $\mathcal{S}_\la$ and thanks to remark \ref{Rk---CuttingLoopsNoAlternatingO},
\[\prod_{f\in F}  d_{\la(f)}^{\chi(f)-\frac{m_f}{2}}=\prod_{f\in F'}  d_{\la(f)}^{\chi(f)}=\mathcal{D}_{\mathcal{S},\la}.\]
Let us now prove the rest of the statement for  the case $G=U(N).$ The case $G=SU(N)$ follows from the exact same argument with $\Omega^*$ in place of $\Omega.$    

 Using  \eqref{eq---WilsonASSpinN} and combining Lemma \ref{eq---AdjCondSpinN} with  \ref{LEM---SpinNCharacFaces},   
\[\EE_{\mathcal{G}_\mathcal{S},U(N)}[W_{\mathcal{S}} W_\la]=\EE_{\mathcal{G}_\mathcal{S},U(N)}[S_{\mu,\tau}W_\la]= 1_{\la\in \Omega_N}  \frac{1}{\prod_{e\in E_+}d_{\partial_r\la(e)}}\prod_{v\in V}   \<J^\la\#I\>_v^{\partial_r\la,l}.  \]
Now 
\[\prod_{e\in E_+} d_{\partial_r\la(e)}^2=   \prod_{v\in V} \left(\prod_{e\in E_+ :\underline{e}=v} d_{\partial_r\la(e)} \prod_{e\in E_+ :\overline{e}=v} d_{\partial_r\la(e)} \right) =  \prod_{v\in V}  d_{\la(\mathfrak{n}_v)} d_{\la(\mathfrak{s}_v)} d^2_{\la(\mathfrak{e}_v)}. \]
We then find
\begin{equation}
 \frac{1}{\prod_{e\in E_+}d_{\partial_r\la(e)}} \prod_{v\in V}  d_{\la(\mathfrak{e}_v)}= \prod_{v\in V}\frac{1}{\sqrt{d_{\la(\mathfrak{s}_v)}d_{\la(\mathfrak{n}_v)}}} =\prod_{f\in F} \frac{1}{d_{\la(f)}^{m_f/2}}
\end{equation}
which together  with  Lemmas \ref{LEM---LemTraceVertex}  and \ref{lem---SignCoxeter} concludes this proof.
%and \ref{lem---SignCoxeter} concludes this proof.
\end{proof}

\begin{proof}[Proof of Lemma \ref{LEM---LemTraceVertex}] Let us fix for each $\a\in \hat{G}$ an invariant  inner product on $V_{\a}$ and endow $V_{\a}\ts V$  with the product inner products for each $\g\in \hat{G}.$  Denoting by $St$ the isomorphism class of the standard representation $V,$ for each choose $e\in E_+,$ 
choose $h_e\in  H_{\partial_l\la(e),St}^{\partial_r\la(e)}=\Hom_G(V_{\partial_r \la(e)},V_{\partial_l\la(e)}\ts V)$ with {$h_{e}^{*}h_{e}=\Id_{V_{\partial_r\la(e)}}.$ } Let us now fix $v\in V$ and set 
$\varepsilon=\la(\mfe_v),\eta=\la(\mfn_v),\omega=\la(\mfw_v),\sigma=\la(\mfs_v)$.  
Set 
%and $St$ for the isomorphism class of $V.$

 \begin{figure}
 
{\centering  \includegraphics[width=10cm,height=7cm]{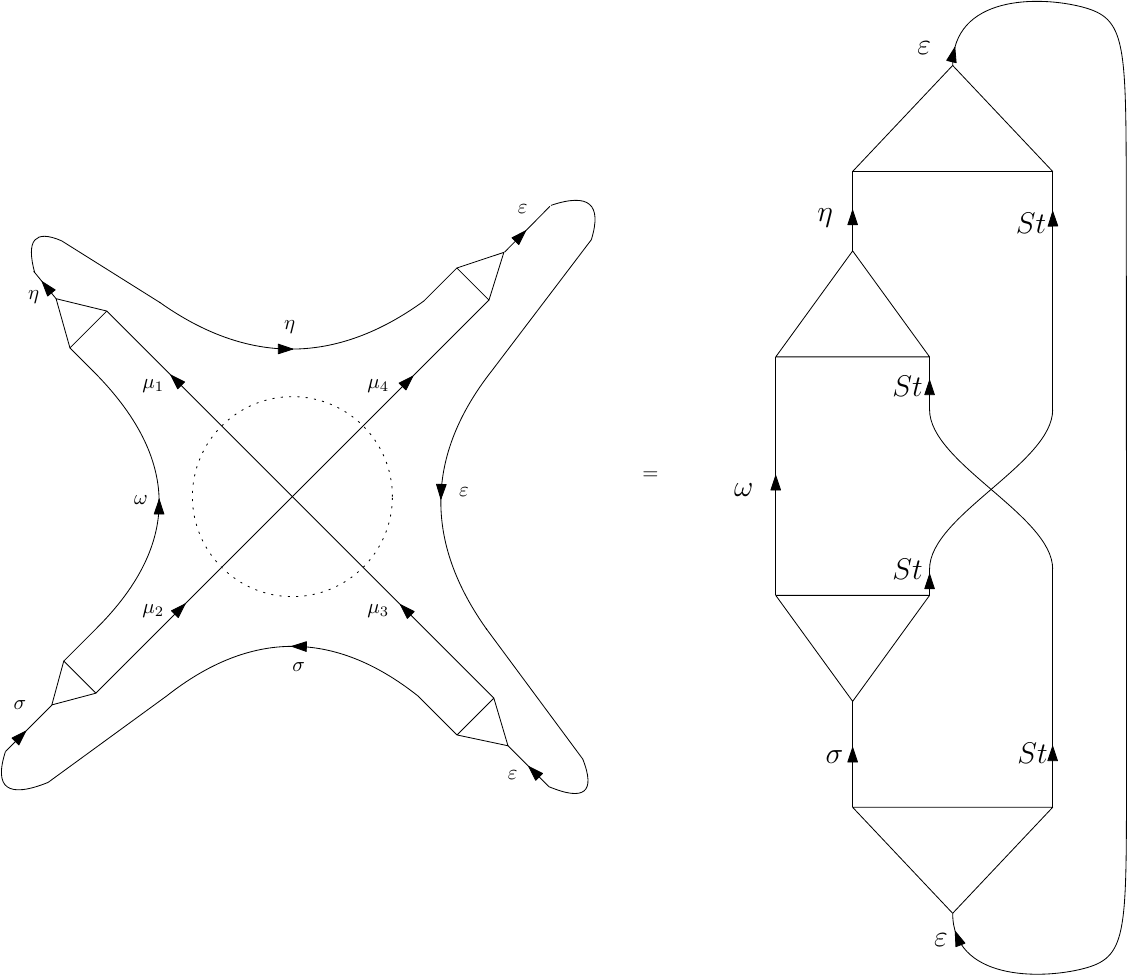}
}
 \caption{Following Figure \ref{Fig---PWAroundVertexMFree} with $I_v=\tau_v, \mu_1=\mu_2=\mu_3=\mu_4=St,$ graphical calculus of $\<J^\la\#\tau_\mu\>_v$ used in Lemma \ref{LEM---LemTraceVertex}'s proof.    \label{Fig--CoxeterDiag}  }
\end{figure}

Considering the graphical calculus of figure\footnote{and first expanding the left-hand side  choosing bases of $V_{\sigma},V_{\varepsilon},V_{\eta},V_{\omega}.$} \ref{Fig--CoxeterDiag},  using self-explanatory notations for the edges of $E_+$ meeting $v$,  
\begin{equation}
\< J^\la \#\tau_\mu \>_v^{\partial \la_r,l}=\Tr_{ V_{\varepsilon}}(h_{\mfn\mfe}^* \left(h_{\mfw\mfn}^*\ts \id_{V}\right)  \left(\id_{\om}\ts \tau\right)( h_{\mfw\mfs} \ts \id_{V}) h_{\mfs\mfe}).\label{eq---FlipCoeff}
\end{equation}
Using Proposition \ref{Prop---Reflection},   the last right-hand side is
\begin{equation}
\Tr_{ V_{\varepsilon}}( h_{\om\eta\varepsilon}^*  \left(\id_{\om}\ts \tau\right)h_{\om\sigma\varepsilon}) = \< h_{\omega\eta\varepsilon} ,s_{\om\varepsilon}(h_{\omega\sigma\varepsilon})\>_{HS}.\label{eq---FlipCoeffIP}
\end{equation}
Since   $s_{\om\varepsilon}$ is an isometry and $s_{\om\varepsilon}^2=\id_{\Hom_G(V_{\varepsilon},V_\om\ts V^{\ts 2})},$  it follows that $s=\pm\id$ when $m_{\om\varepsilon}=1,$ while  \eqref{eq---Reflection} holds for some $\th$ when $m_{\om\varepsilon}=2.$ The constraints on the first sign and on $\theta$ follows from Proposition \ref{Prop---Reflection} which together with \eqref{eq---NormalMultiplicityVectors} concludes this proof.  

%{\color{orange}\[\< J^\la \#\tau_\mu \>_v^{\partial \la_r,l}=\Tr_{ V_{\varepsilon}}(\left(P^{\eta\ts St}_{\varepsilon}\ts \id_{St}\right)\circ \left(P^{\om\ts St}_{\eta}\ts \id_{St}\right)\circ \left(\id_{\om}\ts \tau\right)\circ \left(P_{\omega\ts St}^\sigma \ts \id_{St}\right) \circ P_{\sigma\ts St}^{\varepsilon}). \]}
%Since 
%\[\Tr_{V_{\om}\ts St^{\ts 2}}()=.\]
%

\end{proof}
\begin{rmk} To prove Theorem \ref{thm---IRFformula} we only used that $s_{\om\varepsilon}$ is an isometry and not Proposition \ref{Prop---Reflection}.
\end{rmk}

\begin{lem} \label{Lem---BasisMultModule} Assume $G=U(N)$ or $SU(N).$      Let us write $\a\uparrow\b$ whenever $\a,\b\in \hat{G}$ with $\dim(\Hom_G(V_{\b}, V_{\a}\ts V))>0,$ and fix  then  $h_{\a\b}\in\Hom_G(V_{\b}, V_{\a}\ts V)\setminus\{0\}.$   

\begin{enumerate}
\item When $\a,\b\in \hat{G},$ 
\begin{equation}
m_{\a\b}=\dim(\Hom_G(V_{\b},V_{\a}\ts V^{\ts n}))=\#\{\a_1,\ldots,\a_{n-1}\in\hat{G}:\a\uparrow\a_1\ldots\a_{n-1}\uparrow \b \}  \label{eq---MultiTensn}
\end{equation}
when  $\sum_{i=1}^N(\b_i-\a_i)=n\ge 1$ and $0$ otherwise.  

\item For any $x,y\in \R^N,$ writing $x\prec y$ when $x_i\le y_i$ for all $i$, then for any $\a\not=\b\in \hat{G},$  $a\prec \b$ iff  $m_{\a\b}>0$.  When $\a,\b\in \hat{G}$ with $\sum_{i=1}^N(\b_i-\a_i)=2,$
 \begin{equation}
 m_{\om\varepsilon}  \in\{0,1,2\}.\label{eq---MultiTens2}
 \end{equation}
 \item Setting 
\begin{equation}
h_{\a_0\a_1\ldots\a_n}=   (h_{\a_{0}\a_{1}} \ts \id_{V^{\ts n-1}})\circ\ldots(h_{\a_{n-2}\a_{n-1}} \ts \id_{V})\circ h_{\a_{n-1}\a_n}\in\Hom_G(V_{\a_n},V_{\a_0}\ts V^{\ts n})\label{eq---LatticePath}
\end{equation}
for any $\a_0\uparrow\a_1\uparrow\ldots\uparrow \a_n$ defines a basis of  $\Hom_G(V_\varepsilon, V_{\om} \ts V^{\ts n}).$

\item  Assume  $V_{\a_0}$ endowed with an invariant inner product  for any $\a\in \hat{G}$ and define   inner products on  
$V_{\a_0}\ts V,\ldots,V_{\a_0}\ts V^{\ts n}$ taking products\footnote{setting e.g. $\<x \ts  v , y \ts w \>= \<x, y\> \< v,w\>$ for $x,y\in V_\a, v,w\in V$ where in the right-hand side $\<\cdot,\cdot\>$ denotes the inner products fixed on $V_\a$ and $V.$ }  for pure tensors.  Then $h_{\a\b}$ can be assumed to be an isometry for all $\a\uparrow \b$ with $\a\prec \a_0$ and   the  basis \eqref{eq---LatticePath} is orthogonal  for the Hilbert-Schmidt inner product with 
\begin{equation}
\|h_{\a_0 \a_1\ldots\a_n}\|_{HS}^2=d_{\a_n}\quad \forall \a_0\uparrow\a_1\uparrow\ldots\uparrow\a_n.\label{eq---NormalMultiplicityVectors}
\end{equation}
\end{enumerate}

\end{lem}
\begin{proof} The multiplicity \eqref{eq---MultiTens2} follows from applying $n$ times Pieri's rule. The rest of the proof is standard and  details are left to the reader.
\end{proof}

\begin{prop}  \label{Prop---Reflection}Assume $\a_*\in\hat{G}$ fixed and  $(h_{\a\b})_{\a\uparrow \b,\a_*\prec \a}$ as in Lemma \ref{Lem---BasisMultModule} and for any $\om,\varepsilon$ with $\sum_{i=1}^N(\varepsilon_i-\om_i)=2$ and $m_{\om\varepsilon}>0,$ define $s_{\om\varepsilon}\in\Hom(\Hom_G(V_{\varepsilon},V_\om\ts V^{\ts 2}))$ setting    
\begin{equation}
s_{\om\varepsilon}(h)=(\id_{\om}\ts \tau )\circ h\quad\forall h\in \Hom_G(V_{\varepsilon},V_\om\ts V^{\ts 2}).
\end{equation}
\begin{enumerate}[label=(\roman*)]
\item When $m_{\om\varepsilon}=1$ and  $\g\in\hat{G}$  with $\om\uparrow\g\uparrow\varepsilon,$  
\begin{equation}
 s_{\om\varepsilon}( h_{\om\g\varepsilon})=  {\epsilon} h_{\om\g\varepsilon}
\end{equation}
with $\epsilon\in\{-1,1\}$ given by
\begin{equation}
\epsilon= \frac{1}{(\varepsilon+\rho)_j-(\sigma+\rho)_i},
\end{equation}
where  $i,j\in \{1,\ldots,N\}$  are  such that 
$\eta=\om+\om_j,\sigma=\om+\om_i, \varepsilon=\om+\om_i+\om_j.$   
%%{\color{orange}with \begin{equation}
%%\epsilon=\left\{\begin{array}{ll} 1&\text{when}\quad\varepsilon=\om+2\om_i \quad\text{for some }i, \\&\\-1&\text{otherwise.} 
%%\end{array}\right.
%%\end{equation}} 
\item When $m_{\om\varepsilon}=2,$ $f=(\om,\sigma,\eta,\varepsilon)$ with  $\eta,\sigma\in\hat{G,}$  $\eta\not=\sigma,$ $\om\uparrow \eta \uparrow \varepsilon $ and $\om\uparrow \sigma \uparrow \varepsilon,$ 
\begin{equation}
s_{\om\varepsilon}( h_{\om\sigma\varepsilon})=  \cos(\th_f) h_{\om \sigma \varepsilon}+ \sin(\th_f) h_{\om\eta\varepsilon}\label{eq---Reflection}
\end{equation}
for some $\theta_f\in [0,2\pi)$ such that 
\begin{equation}
\cos(\theta_f)=\frac{1}{(\varepsilon+\rho)_j-(\sigma+\rho)_i}.\label{eq---ReflectionCos}
\end{equation}
and  $i,j\in \{1,\ldots,N\}$  are  such that 
$\eta=\om+\om_j,\sigma=\om+\om_i, \varepsilon=\om+\om_i+\om_j.$   
\end{enumerate}\end{prop}

\begin{proof}[Proof of Prop. \ref{Prop---Reflection}]   Whenever  $\a,\b\in \ZD/\Z$  with $\a\uparrow \b$ and $\tilde \a$ projects onto  $\a$,   there is a unique $\tilde \b\in\ZD$ with projection $\b$ and $\tilde \a\uparrow \tilde \b.$ Therefore, we can assume without loss of generality  that $G=U(N).$ Consider the  representations  $(\rho_i)_{1\le i\le 3}$ of $G$ with vector space $V_{\varepsilon}\ts V^{\ts 2}$ acting on one, two or three tensors, namely  with  
\[\rho_1(g)=\rho_\om(g)\ts  \id_V^{\ts 2},\rho_2(g)=\rho_\om(g)\ts g\ts \id_V\quad\text{and}\quad\rho_3(g)=\rho_\om(g)\ts g\ts g.\]
Consider the endomorphisms 
\[Y_1=\rho_1(\Delta_{\mfu(N)})-\rho_2(\Delta_{\mfu(N)}), Y_2=\rho_2(\Delta_{\mfu(N)})-\rho_3(\Delta_{\mfu(N)})\in\Hom(V_\om\ts V^{\ts 2}) \]
where $\Delta_{\mfu(N)}$ is the Laplacian on $\mfu(N)$ with inner product\footnote{here $\Tr_N$ is the standard non-normalised trace on $M_N(\C).$} 
\[\<X,Y\>_N=\Tr_N(X^*Y)\quad\forall X,Y\in \mfu(N).\]
Then choosing an orthonormal basis $(X_i)_{1\le i\le N^2}$ of $\mfu(N),$ 
\begin{equation}
\sum_{i=1}^{N^2}X_i^2=-N\Id_V\quad\text{and}\quad \sum_{i=1}^{N^2}X_i\ts X_i =- \tau
\end{equation}
while\footnote{In the expressions just above and below the tensor products are over $\C$.}
\begin{equation}
Y_1 (u \ts v\ts w)= -\sum_{i=1}^{N^2}    \left( u\ts X_i^2 v\ts w  +     2  (\rho_\om(X_i). u)\ts X_i v\ts w \right)
 \end{equation}
 and 
 \begin{equation}
Y_2 (u \ts v\ts w)= -\sum_{i=1}^{N^2}    \left( u\ts  v\ts  X_i^2 w  +     2  (\rho_\om(X_i). u)\ts v\ts  X_i  w  +2  u\ts X_i v\ts  X_i  w  \right).
 \end{equation}
Setting 
\begin{equation}
Z_i=Y_i -N\Id_{f} \quad\text{for}\,i\in\{1,2\}\quad\text{and}\quad\Theta=\id_\om\ts \tau
\end{equation}
the three last equations yield  $ Z_2=2 \Theta+  \Theta Z_1 \Theta  $ or equivalently
 \begin{equation}
Z_2\Theta-\Theta Z_1=2.  \label{eq---Hecke}
 \end{equation}
 Now since $\rho_\a(\Delta_{\mfu(N)})=(\|\a+\rho\|^2-\|\rho\|^2)\Id_{V_\a}$ for any $\a\in \ZD,  $
%  for  $\g$ with $\om\uparrow \g\uparrow \varepsilon,$
 \begin{equation}
 Z_1 h_{\om \sigma \varepsilon}= (\|\sigma+\rho  \|^2 -   \|\omega+\rho\|^2-N)h_{\om \sigma \varepsilon}\quad\text{and}\quad Z_2 h_{\om \sigma \varepsilon}= (\|\varepsilon+\rho  \|^2 -   \| \sigma+\rho\|^2-N)h_{\om \sigma \varepsilon}\label{eq---DiagLaplaPath}
 \end{equation}
 Since  $s_{\om\varepsilon}$ is an isometry with $s_{\om\varepsilon}^2=\id$, when  $m_{\om\varepsilon}=2, \sigma\not=\eta$ and  $\Theta h_{\om\sigma\varepsilon}=   s_{\om\varepsilon}(h_{\om\sigma\varepsilon})=\cos(\th)h_{\om\sigma\varepsilon}+\sin(\th)h_{\om\eta\varepsilon}$ for some $\th\in [0,2\pi).$  Using  \eqref{eq---Hecke} and \eqref{eq---DiagLaplaPath}, we conclude that 
\begin{equation}
\cos(\th) (\|\varepsilon+\rho  \|^2+\|\om+\rho  \|^2 -2   \| \sigma+\rho\|^2) = 2 \label{eq---CosEigLap}
\end{equation}
 and \eqref{eq---ReflectionCos} follows.
%\[\rho_1(g).  u\ts  v\ts w =  \rho_\om(g).u\ts v\ts w,  \rho_2(g).  u\ts  v\ts w =  \rho_\om(u)\ts g.v \ts w \quad\text{and}\quad  \rho_3(g).  u\ts  v\ts w= \rho_\om(g). u\ts g.v\ts g.w. \]
When  $m_{\om\varepsilon}=1,$ $\sigma=\eta,$ and using the same equations, $s_{\om\varepsilon}(h_{\omega\sigma\varepsilon})=\epsilon h_{\omega\sigma\varepsilon}$ with 
\[\epsilon (\|\varepsilon+\rho  \|^2+\|\om+\rho  \|^2 -2   \| \sigma+\rho\|^2) = 2.\]%{\color{orange}Case $m_{\om\varepsilon}=1$}
\end{proof}
 
For any $x\in \R^N$ and $k\ge 1,$ let us set $A_{x,k}=\{ y\in\R^N :x\prec y, \sum_{i}(y_i-x_i)\le k\}.$ 
 \begin{lem}\label{lem---SignCoxeter}   In Proposition \ref{Prop---Reflection}, for any $k\ge 1,$  it is furthermore possible to choose $(h_{\a\b})_{\a\uparrow \b,\a\in A_{\a_*,k}}$  such that  
 \[\theta_{f}\in(0,\pi) \quad\forall 
 f=(\om, \sigma,\eta,\varepsilon)\]
  with $ \om\uparrow \sigma\uparrow \varepsilon,\om\uparrow\eta\uparrow\varepsilon,\eta\not=\sigma$ and $\a_*\prec \om.$
 \end{lem}

 \begin{proof}  Assume without loss of generality that $G=U(N),$  $\a_*\in \hat{G},$ $ k\ge 1$ are  fixed and a choice of $(h_{\a\b})_{\a\uparrow \b,\a_*\prec \a}$ has been made as in Proposition \ref{Prop---Reflection}.  Let us reformulate the above question in terms of  cohomology.  Identify $\hat{G}$ with $\rho+\ZD=\Z_{sym}^N\cap \mathcal{W}$ where $\Z_{sym}=\Z+\lfloor \frac{N+1}{2}\rfloor$ and $\mathcal{W}=\{x\in \R^N:x_1<\ldots<x_N\}.$ Consider the 
 $N$-dimensional complex  CW-complex  $X_{\a_*,k}$ whose   cells are given by  the  cells of the lattice $\rho+\Z^N$ that are strictly included in the interior of the compact, convex set  $ \scrC_{\a_*,k}=\overline{W}\cap A_{\a_*,k}.$   
Identifying   $\rho\cap \ZD$ with $\hat G,$   $1$-cells of $X_{\a_*,k}$   are in bijection with pairs $(\a,\b)$ with $\a\in A_{\a_*,k}$ and  $\a\uparrow\b,$ while $2$-cells are in bijections with tuples $f=(\om,\sigma,\eta,\varepsilon)$  with $\om\in 
 A_{\a_*,k}, \om\uparrow \sigma\uparrow \varepsilon$, $\om\uparrow \eta\uparrow \varepsilon$ and $\eta\not=\sigma.$ Consider now the co-chain complex of $X_{\a_*,k}, $ $\partial: 0\to C^0(X_{\a_*,k},\Z/2\Z)\to C^1(X_{\a_*,k},\Z/2\Z)\to\ldots C^N(X_{\a_*,k},\Z/2\Z)\to0. $ Writing a $1$-cochain in  $C^1(X_{\a_*,k},\Z/2\Z)$  as $(\varphi(\a\b))_{\a\uparrow \b,\a\in A_{\a_*,k}}$ with $\varphi(\a\b)\in\{-1,1\},$ consider 
\[\tilde h_{\a \b}= \varphi(\a \b) h_{\a\b}.\]
Then for any  $2$-cell $f=(\om,\sigma,\eta,\varepsilon),$ $\partial \varphi(f)=\varphi(\om\eta)+\varphi(\eta\varepsilon)+\varphi(\sigma\varepsilon)+\varphi(\om\sigma)$ and
\begin{equation}
s_{\om\varepsilon}(\tilde h_{\om\sigma\varepsilon})=\cos(\th_f)\tilde h_{\om\sigma \varepsilon}+ \partial \varphi (f)\sin(\th_f) \tilde h_{\omega \eta \varepsilon}.\label{eq---Coboundary}
\end{equation}
When $f=(\om,\sigma,\eta,\varepsilon)$ is a $2$-cell of $X_{\a_*,k},$ since $m_{\om,\varepsilon}=2,$ by  \eqref{eq---ReflectionCos},   $\sin(\th_f)\not=0$.  Define a $2$-cochain setting 
\[ \psi(f)= \mathrm{sign}(\sin(\th_f)) \]
whenever $f=(\om,\sigma,\eta,\varepsilon)$ is a $2$-face.   By \eqref{eq---Coboundary}, it is enough to show that $\psi$ is co-exact.  As $\scrC_{\a_*,k}$ retracts to $X_{\a_*,k},$  $H^2(X_{\a_*,k},,\Z/2\Z)=H^2(\scrC_{\a_*,k},\Z/2\Z)=\{0\}$  and it is enough to show that $\psi$ is closed.

Consider a  $3$-cell $c$ of $X_{\a_*,k}.$ Label its vertices $\a, \a_{1},\a_2,\a_3,\a_{12},\a_{13},\a_{23},\b\in  \rho+\ZD$  such that $\sum_{l=1}^N (\b_l-\a_l) =3,$   $\a\uparrow \a_i$ for any $i$ and $ \a_i\uparrow \a_{\min(i,j),\max(i,j)} \uparrow  \b$ for  $i\not=j. $ 
%$\a\uparrow \a_i, \b_i\uparrow \b$ for  $1\le i,j\le 3, $ 
%
%
%
%while for $1\le i\le 2,$ $\a_i\uparrow  \b_i, \a_i\uparrow \b_3, \a_3\uparrow \b_i$ 
Let  $s_{12},s_{23}$ be the  flip maps acting on  $V_{\a}\ts V^{\ts 3},$ mapping $t\ts v_1\ts v_2\ts v_3$ to respectively  $t\ts v_2\ts v_1\ts v_3 $ and $t\ts v_1\ts v_3\ts v_2$ and denote abusively by the same symbol the elements of $\Hom(\Hom_G(V_{\b},V_{\a}\ts V^{\ts 3}))$ acting by post-composition.  Then for $i\not=j,$
\begin{equation}
\<s_{12}h_{\a\a_{i}\a_{ij}\b}, h_{\g_{12}} \>_{HS}=0\quad\text{and}\quad \<s_{23}h_{\a\a_{i}\a_{ij}\b}, h_{\g_{23}} \>_{HS}=0\label{eq---OrthoFlip}
\end{equation}
except when respectively  $\g_{12}\in\{\a\a_i\a_{i,j}\b,\a\a_{j}\a_{i,j}\b\}$ and $\g_{23}\in \{\a\a_i\a_{i,j}\b,\a\a_{i}\a_{i,l}\b\}$ with $l\not= i,j.$ Then \begin{equation}
\<s_{12}h_{\a\a_{i}\a_{ij}\b}, h_{\a\a_{j}\a_{i,j}\b} \>_{HS}= \<s_{\a \a_{i,j}} h_{\a\a_{i}\a_{ij}}, h_{\a\a_{j}\a_{ij}} \>_{HS} \label{eq---RestricIP2}
\end{equation}
and
\begin{equation}
\<s_{23}h_{\a\a_{i}\a_{ij}\b}, h_{\a\a_{i}\a_{i,l}\b} \>_{HS}=\<s_{\a_{i}\b}h_{\a_{i}\a_{ij}\b}, h_{\a_{i}\a_{i,l}\b} \>_{HS}.
\label{eq---RestricIP3}\end{equation}
 Since $s_{12}s_{23}s_{12}=s_{23}s_{12}s_{23},$ it follows that  

\[
 \<s_{\a \a_{13}}h_{\a \a_{1}\a_{13}}, h_{\a \a_{3}\a_{13}} \>_{HS}\<s_{\a_{3}\b}h_{\a_{3}\a_{13}\b}, h_{\a_{3}\a_{23}\b} \>_{HS} \<s_{\a \a_{23}}h_{\a \a_{1}\a_{23}}, h_{\a \a_{2}\a_{23}} \>_{HS}
 \]
and 
\[
\<s_{\a_{1}\b}h_{\a_{1}\a_{13}\b}, h_{\a_{1}\a_{12}\b} \>_{HS}\<s_{\a\a_{12}}h_{\a \a_1\a_{12}}, h_{\a \a_2\a_{12}} \>_{HS}\<s_{\a_{2}\b}h_{\a_{2}\a_{12}\b}, h_{\a_{2}\a_{2,3}\b} \>_{HS}
 \]
are both equal to 
\[\<s_{12}s_{23}s_{12}h_{\a\a_1\a_{13}\b},h_{\a\a_2\a_{23}\b}\>.\]
In particular, the product of their signs is $1$ and $\partial \psi(c)= 1.$
 \end{proof}
 The last part of the above argument is a variation of \cite[Thm 3.12]{Naz}.

 \begin{rmk}  As argued in \cite{IRFThierry}, the Makeenko-Migdal equations \cite{LevMF,DHKMM,DGHKMM} can be proved  from \ref{THM---IRFOneSpinN}. The curious reader can  indeed check it here using equation \eqref{eq---CosEigLap}.  
 \end{rmk}

  \section{Duality with the Brauer algebra}
  \label{Section---RepBrauer}
  The purpose of this section is to give an expression for the projection on traceless tensors in terms of the walled Brauer algebra that satisfies the bound \eqref{eq---BoundHarmonicProj}  on the coefficient of each 
  diagram. Besides we also give such an expression for  projections onto its isotopic $U(N)$-components.  We shall  give here expressions depending on the  Jucys-Murphy elements of $\mathcal{B}_{n,m}(N).$

Assume $n,m\ge 1,$   recall that 
\[X_i= (1\,i )+(2\, i)+\ldots + (i-1\,i)\quad\text{for}\quad i\le n,\] 
  and set  for  $ 1\le j\le m,$
\[Y_{j}=N- \<1 \, n+j\>-\<2 \, n+j\>-\ldots - \<n \, n+j\>+(n+1\,  n+j )+(n+2 \, n+j)+\ldots + (n+j-1\,n+j). \] 
Let us also define their "symmetric" versions 
\[\overline{X}_{j}= (n+1\,  n+j )+(n+2 \, n+j)+\ldots + (n+j-1\,n+j) \quad\text{for}\quad j\le m\]
and for  $ 1\le i\le n,$
\[\overline{Y}_{i}=N- \<i \, n+1\>-\<i\, n+2\>-\ldots - \<n \, n+m\>+(1\,  i)+(2 \, i)+\ldots + (i-1\,i),\]
with the convention $X_1=\overline{X}_1=0, Y_{1}=N- \<1 \, n+j\>-\<2 \, n+j\>-\ldots - \<n \, n+j\>$ and $\overline{Y}_{1}=N- \<1 \, n+1\>-\<1 \, n+2\>-\ldots - \<1 \, n+m\>.$ It is elementary to check that $X_1,X_2,\ldots,X_n,Y_1,\ldots,Y_{m}$   span\footnote{as well as $\overline{X}_1,\overline{X}_2,\ldots,\overline{X}_m,\overline{Y}_1,\ldots,\overline{Y}_{n}.$ } a commutative sub-algebra of $\mathcal{B}_{n,m}(N).$ It is moreover known 
\cite{VerNik,JungKim} how to co-diagonalise their action\footnote{Moreover, for %{\color{orange}$N\ge n+m$}
{$N\ge n+m$} it can also be shown \cite{JungKim} that  the center of $\mathcal{B}_{n,m}(N)$ is exactly
\[ \{p(X_1,\ldots,X_n,Y_1,\ldots, Y_{m}):p\in \La_{n,m}\}\]
where $\La_{n,m}$ is a the space of supersymmetric polynomials  $p\in\C[x_1,\ldots,x_n,y_1,\ldots,y_m]$  that is with 
\[p(x_{\a(1)},\ldots,x_{\a(n)},y_{\b(1)},\ldots,y_{\b(m)} )= p(x_1,\ldots, x_n,y_1,\ldots,y_m)\quad \forall \a\in S_n,\b\in S_m\]
and 
\[p(x_1,\ldots, x_{n-1},x_n,-x_n,y_2,\ldots,y_m)=p(x_1,\ldots, x_{n-1},0,0,y_2,\ldots,y_m).\]  } 
on $T_{n,m}.$ The following lemma then follows. 
 
 \begin{lem}\label{Lem---BrauerExpression} For any $N\ge n+m,$ $N+X_i+\overline{X}_j$ is invertible for all $i\le n,j\le m$ and 
\[ P_{N}^{n,m}=\rho_{N}(q_{n,m})\]
 where for $m\le n$
 \begin{equation}
 q_{n,m}= \prod_{1\le i\le n,1\le j\le m}(N+X_i+\overline{X}_j)^{-1}\prod_{1\le i\le n,1\le j\le m}(X_i+Y_j)\label{eq---ProjHarmoJM}
 \end{equation}
 and for $n<m,$
 \begin{equation}
 q_{n,m}= \prod_{1\le i\le n,1\le j\le m}(N+X_i+\overline{X}_j)^{-1}\prod_{1\le i\le n,1\le j\le m}(\overline{X}_j+\overline{Y}_i).\label{eq---ProjHarmoJM}
 \end{equation}
  \end{lem}
 From \eqref{eq---ProjHarmoJM}, it is elementary to deduce the bound \eqref{eq---BoundHarmonicProj} we are after by normalising by $N$ and expanding the products.  We prove a slightly sharper statement in Lemma \ref{lem---BoundTracelessProjCumulNormal} below based on the following observation. 
 
 \begin{lem} \label{Lem---TriangularBrauer} For any $\pi\in \mathcal{D}_{n,m},$ denote by $|\pi|$ the smallest integer $k$ with {$\pi= N^{-l}\tau_1\ldots \tau_k$} for some $l\ge 0,$ where $\tau_i$ are transpositions of $S_{n,m}$ or {Weyl-contractions} $\<a\, b\>$ with $a\le n<b.$ Then for any $\pi,\nu\in\mathcal{D}_{n,m},$ 
 \begin{equation}
 h(\pi)\le |\pi|
 \end{equation}
 and 
\begin{equation}
|c(\pi,\nu)|+{l(\pi,\nu)}\le |\pi|+|\nu|.\label{eq---TriangularIneq}
\end{equation}
 \end{lem}
 \begin{proof} Consider a sequence $\tau_1,\ldots,\tau_k$ of transpositions   or {Weyl-contractions} in $\mathcal{D}_{n,m}$ and the products  $\tau_1\tau_2\ldots \tau_{t}=N^{\ell_{t}}\pi_t\in\mathcal{B}_{n,m}(N)$  for 
 $1\le t\le k.$ Then $(\ell_t)_{1\le t\le k},(h(\pi_t))_{1\le t\le k}$ are non-increasing,  $h(\pi_t)-h(\pi_{t-1}), \ell_{t}-\ell_{t-1}\le 1$ for all $t>1.$ The equality  $\ell_{t}-\ell_{t-1}=1$  holds whenever $\tau_t$ is a Weyl-contraction and the bottom 
 horizontal string of $\tau_{t}$ matches two vertices matched by a top horizontal string of  $\pi_{t-1},$ while $h(\pi_t)-h(\pi_{t-1})=1$ whenever $\tau_t$ is a Weyl contraction whose bottom horizontal string does not meet 
any top horizontal string of $\pi_{t-1}$  in $\tau_t\pi_{t-1}.$  Therefore $h(\pi_t)\le t$ for all $t$ and when  $\ell_{t}-\ell_{t-1}=1,$   $\tau_{t}\pi_{t-1}=N \pi_t.$ Denoting by $t_1<t_2<\ldots <t_l$ the positive  increase 
 times of $\ell$, 
\begin{equation}
\tau_1\ldots\tau_k=N^l\pi\quad\text{and}\quad \tau_1 \ldots \hat\tau_{t_1}\ldots \hat\tau_{t_l}\ldots \tau_k=\pi\label{eq---ProductBrauer}
\end{equation}
 for some $\pi\in\mathcal{D}_{n,m},$ where $\hat{\tau_{i}}$ means the $i$th term is omitted in the product. We conclude that $h(\pi)\le |\pi|\le k-l.$  
 
 Consider now $\pi,\nu\in\mathcal{D}_{n,m}.$ From \eqref{eq---ProductBrauer} there are transpositions or Weyl-contractions  $\tau_{1},\ldots, \tau_{|\pi|+|\nu|},$  such that taking products in $\mathcal{B}_{n,m}(N),$  
 \[\tau_1\ldots\tau_{|\pi|}=\pi\quad\text{and}\quad \tau_{|\pi|+1}\ldots\tau_{|\pi|+|\nu|}=\nu.\]
 Since $\tau_1\ldots \tau_{|\pi|+|\nu|}=\pi\nu=N^{l(\pi,\nu)}c(\pi,\nu),$ $|c(\pi,\nu)|\le |\pi|+|\nu|-l(\pi,\nu).$
   \end{proof}
 
Thanks to lemma \ref{Lem---TriangularBrauer}, for any  real $\la\ge 1$ and integer $N\ge 1,$ setting for any $q=\sum_{\pi\in \mathcal{D}_{n,m}} \kappa_\pi (q) \pi\in \mathcal{B}_{n,m}(N),$
\begin{equation}
\|q\|_{1,\la}=  \sum_{\pi \in\mathcal{D}_{n,m}} \la^{|\pi|}|\kappa_\pi(q)|\label{eq----NormGroupAlgS}
\end{equation}
defines a sub-multiplicative norm  on $\mathcal{B}_{n,m}(N)$ for $\la \ge N$ and, by restriction, on $\C[S_{n,m}]$ for all $\la\ge 1.$    Since $h(\pi)\le |\pi|$ for any $\pi\in\mathcal{D}_{n,m}$, the following concludes 
the proof of \eqref{eq---BoundHarmonicProj}.

 \begin{lem} \label{lem---BoundTracelessProjCumulNormal}For any $n,m$ fixed,  consider $q_{n,m}$ as in Lemma \ref{Lem---BrauerExpression}. Then  \begin{equation}
 \sup_{N\ge n+m} \|q_{n,m}\|_{1,N}<\infty.
 \end{equation}
  \end{lem}

 \begin{proof} 
Let us assume w.l.o.g. $m\le n.$ For all  $1\le \la\le  N,1 \le i\le n,1\le j\le m, $ 
\begin{equation}
\left\|\frac{X_i}{N}\right\|_{1,\la^{-1} N}=\frac{(i-1)}\la ,\left\|\frac{\overline X_j}{N}\right\|_{1,\la^{-1} N}=\frac{j-1}\la \quad\text{and}\quad\left\|\frac{C_j}{N}\right\|_{1,N}= n, 
\end{equation}
 where  $C_j= \<1\, n+j\>+\ldots +\<n \,n+j\>.$
 Note that  for any $\sigma\in \mathcal{S}_{n,m},$ $|\sigma|\le n+m-1$ and  for $1\le \la\le N $ and $q\in \C[S_{n,m}],$
 \begin{equation}
 \|q\|_{1, N}\le \la^{n+m-1} \|q\|_{1,\la^{-1} N}. \label{eq---ScaleBrauerNorm}
 \end{equation}
 We conclude by sub-multiplicativity that  for all $i\le n,j\le m$ and $N\ge n+m,$
 \[\left\| \left(1+\frac{X_i+\overline{X}_j}{N}\right)^{-1}\right\|_{1,\frac{N}{n+m}}\le  \frac{n+m}{2} \]
 and  using first $\|\cdot\|_{1,N}$-sub-multiplicativity  followed by \eqref{eq---ScaleBrauerNorm},
 \begin{align*}
 \|q_{n,m}\|_{1,N}&\le (n+m)^{n+m} \|\prod_{i\le n,j\le m} (1+\frac{X_i+\overline{X}_j}{N})^{-1}\|_{1,\frac{N}{n+m}} \prod_{i\le n,j\le m}\|1+\frac{X_i+\overline{X}_j -C_j}{N}\|_{1,N}\\
 &\le (n+m)^{nm+n+m}(2n+m)^{nm}. 
 \end{align*}
 \end{proof}

  Let us now recall the notion of Gelfand-Zetlin vectors.  It shall provide us a basis of joint eigenvectors for  $X_1,\ldots, X_n,Y_1,\ldots,Y_m$ acting on $T_{n,m}.$ 
  
  Identifying $1\in \C$ with the unique diagram of $\mathcal{D}_{1,0}$ and, when $1\le k<n$ or $0\le l<m$,  any diagram of $\mathcal{D}_{k,0}$ or $\mathcal{D}_{n,l}$ with a diagram of $\mathcal{D}_{k+1,0}$ or $\mathcal{D}_{n,l+1}$  by  adding a vertical strand to its right,  there are natural injections of algebras
  \begin{equation}
  \C\subset \mathcal{B}_{1,0}(N)\subset \mathcal{B}_{2,0}(N)\subset \ldots \subset \mathcal{B}_{n,0}(N)\subset \mathcal{B}_{n,1}(N)\subset \ldots \subset \mathcal{B}_{n,m}(N).
  \end{equation}
  Denote for any  $0\le k\le n+m$ by $\mathcal{A}_k$ the $k$-th  smallest algebra in this tower.  We shall assume from here onwards that $N\ge n+m.$ Then the map $\rho_N:\mathcal{A}_k\to \End(T_{n,m})$ 
  is injective so that $\mathcal{A}_k$ is semi-simple for all $0\le k\le n+m.$ Moreover it can be shown   \cite{Nikitin} that for any $k>0$ the restriction of any irrep of $\mathcal{A}_k$ to $\mathcal{A}_{k-1}$ is multiplicity free. It   follows that  for any tuple $T=(\a_{n+m},\a_{n+m-1}\ldots,\a_2,\a_1)$ with  $\a_k$ irrep of $\mathcal{A}_k$ for all $k,$ denoting by $V^{\a_{k}}$ the vector space associated to $\a_{k}$ and by $(V^{\a_{k}})^{\a_l}$ the $\mathcal{A}_l$-isotopic component of type $\a_l$ when $l\le k$,
  \begin{equation}
  V_T=\bigcap_{k=1}^{n+m}(V^{\a_{n+m}})^{\a_k}
  \end{equation}
  is either null or one dimensional. The latter case occurs if and only if $\a_{k}$ occurs in the $\mathcal{A}_k$-decomposition of  $V^{\a_{k+1}}$ for all $k<n+m.$  A non-zero element $v_T\in V_T$ is called a Gelfand-Zetlin vector and $T$ a Brattelli's path.   Denoting by $\mathcal{GZ}(\a)$ the set of Gelfand-Zetlin paths starting from  a $\mathcal{A}_{n+m}$-irrep $\a,$
  \begin{equation}
  V^\a=\bigoplus_{T\in\mathcal{GZ}(\a)}V_T.\label{eq---GZDecomposition}
  \end{equation}

  \begin{lem} \cite[Prop. 4.2]{JungKim},\cite{Naz}  \label{Lem-WBJMSpectrum}Assume $N\ge n+m$ {and $n\ge m.$ }
  \begin{enumerate}
  \item Irreps of  $\mathcal{B}_{n,m}(N)$ are parametrised as $[\la,\mu]$ by pairs of integer partitions $\la,\mu$ with $|\la|\le n,|\mu|\le m, n-|\la|=m-|\mu|.$ 
  \item  Brattelli's paths are of the form  $T=([\la_{n+m},\mu_{n+m}],\ldots,[\la_{1},\mu_{1}])$  with 
  \begin{enumerate}
  \item $\mu_i=\emptyset$
  \item ${\la_{i}\subset \la_{i+1}}$ and $|\la_{i+1}|=|\la_{i}|+1$ for $1\le i< n,$
  \item for all $n\le i<n+m,$ 
  \[{\la_{i+1}\subset \la_{i}},\mu_{i}\subset \mu_{i+1}\quad\text{and}\quad|\la_{i}|+|\mu_{i+1}|=|\la_{i+1}|+|\mu_i|+1.\]
 
  \end{enumerate}
\item For any Bratelli's path $T=([\la_{n+m},\mu_{n+m}],\ldots,[\la_{1},\mu_{1}])$ and $v\in V_T,$  for all $i\le n,j\le m,$ $X_i.v=c_T(i) v $ and $Y_j.v=c_T(n+j)v$ where
\begin{equation}
c_T(i) = c(\la_{i}\setminus \la_{i-1}) \quad\text{for all}\quad 1\le  i\le n
\end{equation}
with $\la_0=\emptyset$ and for all $1\le j\le m,$
\begin{equation}
c_T(n+j)= \left\{\begin{array}{ll} c(\mu_{j+n}\setminus \mu_{j+n-1}) +N&\text{if}\quad |\mu_{j+n}|=|\mu_{j+n-1}|+1,\\-c(\la_{j+n-1}\setminus \la_{j+n}) &\text{if}\quad |\la_{j+n}|=|\la_{j+n-1}|+1.\end{array}\right.
\end{equation}

\item The $\mathcal{B}_{n,m}(N)\times U(N)$ decomposition of $T_{n,m}$ is  
\begin{equation}
T_{n,m}=\bigoplus_{0\le k\le \min(n,m), \la\vdash n-k,\mu\vdash m-k} V^{[\la,\mu]}\otimes V_{[\la,\mu]_N}
\end{equation}
where    $(V^{[\la,\mu]},\rho^{[\la,\mu]})$ denotes the $\mathcal{B}_{n,m}(N)$-irrep  parametrised by  $[\la,\mu]$ as in (1). 
\item For any $\la,\mu$ as in (1) and  $\pi\in\mathcal{D}_{n,m},$
\begin{equation}
\rho^{[\la,\mu]}(\pi)=0 \quad \text{if}\quad h(\pi)>n-|\la|.
\end{equation}
\item For any $\la\vdash n,\mu\vdash m,$  $V^{[\la,\mu]}$ is isomorphic to $V^{\la}\ts V^{\mu}$ as $S_{n,m}$-module and 
\begin{equation}
\overset{\circ}{T}_{n,m}=\bigoplus_{\la\vdash n,\mu\vdash m} V^{[\la,\mu]}\otimes V_{[\la,\mu]_N}.
\end{equation}
   \end{enumerate}
  \end{lem}
  
  The above lemma is a generalisation  of the Okounkov-Vershik approach to the representation of the symmetric group. We shall need the following application of the latter.
  
  \begin{lem}[\cite{JungKim}] For  any symmetric polynomial $p\in \C[x_1,\ldots, x_n]$ and any partition $\la\vdash n,$ 
  \[\rho_{V^\la}(p(X_1,\ldots,X_n))= p(c(\la_1),\ldots,c(\la_n\setminus \la_{n-1})) \id_{V^\la}\]
  where $(\la_n,\ldots,\la_1)$ is any {decreasing tuple} of partitions with $\la_n=\la$ and $|\la_{k-1}|=|\la_{k}|-1$ for all $k>1.$ 
  \end{lem}

%   \item for any $v_T\in V_T,$ setting $[\la_0,]$
%  \item
%  \begin{equation}
%  X_i.v_T= (\|[\la_i,\mu_i]_N+\rho\|^2-\|[\la_{i-1},\mu_{i-1}]_N+\rho\|^2) v_T \quad\text{and }\quad Y_j.v_T=  ([\la_{i},\mu]_N+...) v_T 
%  \end{equation}
  
\begin{proof}[Proof of Lemma \ref{Lem---BrauerExpression}] For any Brattelli's path $T=([\la_{n+m},\mu_{n+m}],\ldots,[\la_{1},\mu_{1}])$ with $k=n-|\la_{n+m}|=m-|\mu_{n+m}|\ge0,$ the sequence $(\la_1,\ldots,\la_{n+m})$ is strictly decreasing exactly $k$ times.  For such a time $t,$ $t\ge n$ with $|\la_{t+1}|=|\la_{t}|-1$, denote  by   $t'\le n$ the time at which %{\color{orange} the box }
{the box} $ \la_t\setminus \la_{t+1}$ is added that $t'$ such that  $\square=\la_{t'}\setminus \la_{t'-1}= \la_t\setminus \la_{t+1}.$ Then by (3) of Lemma \ref{Lem-WBJMSpectrum}, $X_{t'}$ and $Y_t$ act by multiplication by   $c(\square)$ and $-c(\square)$ on $V_T$.   Therefore for any Brattelli's path $T$ starting from $[\la,\mu]$ with $n-|\la|>0,$   $\prod_{t'\le n,t\le m}(X_{t'}+Y_t)v=0$ for any GZ-vector $v\in V_T.$  Since Gelfand-Zetlin vectors form a basis \eqref{eq---GZDecomposition} of $V^{[\la,\mu]}, $ $\rho^{[\la,\mu]}( \prod_{i\le n,j\le m}(X_{i}+Y_j))=0.$
 
 To conclude, when $N\ge n+m, $   for all $i\le n, j\le m,$  $\|X_i+\overline{X}_j\|_{1,1}=i+j-1<N$ and  $(N+X_i+\overline{X}_j)$  in $S_{n,m}.$ Moreover, by (5)  of Lemma \ref{Lem-WBJMSpectrum}, for $\la\vdash n,\mu\vdash m,$  Weyl contractions act trivially on $V^{[\la,\mu]}$ so that for all $i\le n,j\le m,$
 \begin{equation}
  \rho^{[\la,\mu]}(X_i+Y_j)=N+X_i+\overline{X}_j.
 \end{equation}
 We conclude that for  pair of partitions $\la,\mu$ with $|\la|\le n,|\mu|\le m, n-|\la|=m-|\mu|=k,$ 
 \[  \rho^{[\la,\mu]}(\prod_{1\le i\le n,1\le j\le m}(N+X_i+\overline{X}_j)^{-1}\prod_{1\le i\le n,1\le j\le m}(X_i+Y_j))=\left\{\begin{array}{ll} 1&\text{if}\quad k=0,\\
 0&\text{if}\quad k>0\end{array}\right.\]
 and by (4) and (6) of Lemma \ref{Lem-WBJMSpectrum},
 \begin{equation}
  \rho_N\left(\prod_{1\le i\le n,1\le j\le m}(N+X_i+\overline{X}_j)^{-1}\prod_{1\le i\le n,1\le j\le m}(X_i+Y_j)\right)=P_{N}^{n,m}.
 \end{equation}
\end{proof}
 
 \begin{rmk} In \cite{BGH}   the following alternative formula was proved
 \[q_{n,m}=\prod_{\la\in S\setminus \{0\} }(1-\la^{-1}C_{n,m}),\]
 where {$C_{n,m}=\sum_{i\le n,j\le m} \<i \, -j\>$ and  $S$ denotes the set of eigenvalues of $\rho_N(C_{n,m}).$ }
 \end{rmk}
 
\begin{rmk}   The above algebraic results also give a direct expression for the projectors on any Gelfand-Zetlin vectors \cite[Prop. 4.4]{JungKim}. However  it is not obvious how to use these expressions to find  bound such as \eqref{eq---BoundHarmonicProj}.
\end{rmk}

\section{Wick product of Newton polynomials and traceless-tensors: the Gross-Taylor formula} \label{sec---GrossTaylor}
For any permutation $\sigma\in S_n,$  let us define the Newton polynomial $p_\sigma$ as the function
\begin{equation}
p_\sigma(U)= \Tr_{V^{\ts n}}( \rho_N(\sigma) \rho_n(U))\quad \forall U\in U(N)
\end{equation}
on the unitary group. As it only depends  on the conjugacy class of $\sigma,$  we  also denote  it by $p_\la$  where  $\la\vdash n$ is the integer  partition of $n$ into cycle lengths of $\sigma,$ with the convention that $p_\emptyset =1.$ For any $n\ge 1$ and 
$U\in U(N),$ to  match   standard notations, we simply write 
\begin{equation}
p_n(U)=p_{(1\ldots n)}(U)=\Tr(U^n).
\end{equation}
It is a beautiful and standard fact discovered in\footnote{We recall an argument in next Lemma's proof.  Recall that the right-hand side of \eqref{eq---GaussianTraces} defines an inner product on  symmetric functions known as the Hall inner product .   }  \cite{DiacShah} that $(p_n(U))_{n\ge 1}$ are almost independent normal complex random variables in the following sense. For any partition $\la$ denote by $m_i(\la)$ is the number of parts of size $i$ in $\la.$  The  family  $(p_\la)_{\la\in\mathbb{Y},|\la|\le N}$  is orthogonal in  $\mathrm{L}^2(U(N))$ with
\begin{equation}
\int_{U(N)} p_\la(U)\overline{p_{\mu}(U)}dU=\EE[ \prod_i Z_i^{m_i(\la)} \overline{Z}_i^{m_i(\mu)}]= \delta_{\la,\mu} z_\la \label{eq---GaussianTraces}
\end{equation}
for any pair of partitions $\la,\mu$   with\footnote{Mind that it is true for all $\mu\in \Yb.$} $|\la|\le N.$ Here $(Z_i)_i$ are independent normal, centred, complex random variables of variance $i$ and so
\begin{equation}
z_\la=\prod_{i\ge 1}  i^{m_i(\la)}m_i(\la)!.
\end{equation}

Consider now the representation $\overset{\circ}{T}_{n,m}$ in place of $V^{\ts n}$ in the above definition.   For  $n,m\ge 1$ and $\a\in S_n,\b\in S_m$ define
\begin{equation}
p_{[\a,\b]}(U)=\Tr_{\overset{\circ}{T}_{n,m}}({\rho}_{N}(\a\times\b) \rho_{n,m}(U))\quad\forall U\in (N).
\end{equation}
It only depends on the conjugacy classes of $\a$ and $\b$ and when  $\la\vdash n,\mu\vdash m$ are the  partitions  into cycle lengths of $\a,\b,$ we shall also denote  it by $p_{[\la,\mu]}$ with the convention that 
\begin{equation}
p_{[\emptyset,\emptyset]} =1, p_{[\la,\emptyset]}=p_\la \quad\text{and}\quad p_{[\emptyset,\mu]}=\overline{p}_{\mu}.
\end{equation}
For $n\ge 1$ and $a,b\ge 0,$ we set 
\begin{equation}
p_{n}^{(a,b)}=p_{[(n^a),(n^b)]}.
\end{equation}
Recall \cite{AvramTaqqu} that when $X_1,\ldots, X_k$ are random variables  on a same probability space with $m_S=\EE[\prod_{i\in S} |X_i|] <\infty$ for any subset $S$ of $[k],$ then their Wick product  is defined recursively by
\[ : X_1\ldots X_k : =P(X_1,\ldots, X_k)\]
where $P \in \C[Y_1,\ldots, Y_k]$  is a polynomial of degree $k$ such that 
\[ \frac{\partial}{\partial Y_i} (P)(X_1,\ldots,X_k)= :X_1\ldots \hat{X_i}\ldots  X_k: \quad\forall 1\le i\le k\]
and 
\[ \EE[: X_1\ldots X_k :]=0 \]
where $\hat{X_i}$ means the $i$th symbol is omitted when $k\ge 2$ and for $k=1,$ $\frac{\partial}{\partial Y_1} P(Y_1)=1.$ For instance, when  $X,Y$ have second moments, 
\[:X:= X- \EE[X] \quad\text{and}\quad :XY:= XY- \EE[X]Y-\EE[Y]X+\EE[XY].\]
When $(X_1,\ldots,X_n)$ is a centered Gaussian vector,
\begin{equation}
:X_1\ldots X_n:= \sum_{\pi} \prod_{a\not=b:\{a,b\}\in \pi} (-\EE[X_aX_b]) \prod_{c: \{c\}\in \pi} X_c\label{eq---WickNormal}
\end{equation}
and 
{ \begin{equation}
X_1\ldots X_n= \sum_{\pi} \prod_{a\not=b:\{a,b\}\in \pi} \EE[X_aX_b] : \prod_{c: \{c\}\in \pi} X_c:\label{eq---ProducttoWickNormal}
\end{equation}}
where the summation is over partitions of $[n]$ with blocs of size $1$ or  $2.$ Then for any subset $A,B$ of $[n],$ 
\begin{equation}
\EE[:\prod_{i\in A} X_i: :\prod_{i\in B}  X_i:]= \sum_{\pi} \prod_{\{a,b\}\in\pi,a\in A, b\in B} \EE[X_aX_b]\label{eq---ProductWick}
\end{equation}
where the sum runs over partition of $A\sqcup B$ into blocs of size $2.$

\begin{prop} \label{Prop---WickP} For any $\la,\mu\in \Yb$ with $|\la|+|\mu|\le N,$ 
\begin{equation}
p_{[\la,\mu]}(U) = :p_{\la}(U) p_{\mu}(U^{-1}): .
\end{equation} \end{prop}

This observation follows from the almost Gaussianity \eqref{eq---GaussianTraces} and the fact that  $(p_{[\la,\mu})_{\la,\mu}$ are an orthogonalisation of $(p_\la \overline{p}_\mu)_{\la,\mu}$. For all $n\ge 0,$ let  $\mathcal{F}_n$ be the linear span of $\{p_\la\overline{p}_\mu : |\la|+|\mu|\le n \}$ and $\mathcal{F}_{-n}=\{0\}.$ 

\begin{lem}  \label{lem---OrthoNewton} The family $(p_{[\la,\mu]})_{\la,\mu\in \Yb, |\la|+|\mu|\le N} $ is the unique  orthogonal family of $L^2(U(N))$ with  
\begin{equation}
\<p_{[\la,\mu]},p_{[\tilde\la,\tilde\mu]}\>_{\mathrm{L}^2(U(N))}=  z_{\la}z_{\mu}\delta_{\la,\tilde\la}\delta_{\mu,\tilde\mu}\label{eq---OrthoHarmoNewton}
\end{equation}
and
\begin{equation}
p_{[\la,\mu]}-p_{\la}\overline{p}_\mu\in \mathcal{F}_{|\la|+|\mu|-2} \label{eq---HierarchyNewton}
\end{equation}
for all $ \la,\tilde\la,\mu,\tilde\mu\in\Yb$ with $ |\la|+|\mu|,  |\tilde\la|+|\tilde\mu|  \le N.$ 
 \end{lem}
 
Let us give an argument that builds on the  one of \cite{DiacShah} that we recall now for completeness. 
 \begin{proof}[Proof of \eqref{eq---GaussianTraces}] For any $\sigma\in S_n,$
\[p_\sigma=\sum_{\la\vdash n,\ell(\la)\le N}\chi^\la(\sigma)\chi_\la\]
and for $\a\in S_n,\b\in S_m$, by orthogonality of $U(N)$-characters 
\[\<p_{\a},p_\b\>_{L^2(U(N))}=\delta_{n,m}\sum_{\la\vdash n,\ell(\la)\le N} \chi^\la(\a)\chi^\la(\b^{-1}).\]
In particular \eqref{eq---GaussianTraces} holds when $|\la|\not=|\mu|.$  Also for $n\le N,$ denoting by $\chi(\mu)$ the value of an $S_n$-character $\chi$ on a conjugacy class with cycle length type $\mu,$ by orthogonality of characters $(z_\mu^{-\frac{1}{2}}\chi^\la(\mu))_{\la,\mu\vdash n}$ is an orthogonal matrix. Orthogonality of its columns together with the last display yield  \eqref{eq---OrthoHarmoNewton}  when $|\la|\le N.$

 \end{proof}
 \begin{proof}[Proof of Lemma \ref{lem---OrthoNewton}]

Assume  $n,m\ge 1$.  For any $\a\times\b\in S_n\times S_m$  by Koike-Schur-Weyl duality
\begin{equation}
p_{[\a,\b]}=\sum_{\la\vdash n,\mu\vdash m,\ell(\la)+\ell(\mu)\le N}\chi^{[\la,\mu]}(\a\times \b)\chi_{[\la,\mu]}\label{eq---HarmoNewtonTOCharacters}
\end{equation}
and by  orthogonality of $U(N)$-characters, for $\pi\times\sigma\in S_{n',m'},$
\begin{equation}
\<p_{[\a, \b]},p_{[\pi,\sigma]}\>_{L^2(U(N))}=\delta_{n,n'}\delta_{m,m'}\sum_{\la\vdash n,\mu\vdash m,\ell(\la)+\ell(\mu)\le N} \chi^\la(\a) \chi^\mu(\b) \chi^\la(\pi^{-1}) \chi^\mu(\sigma^{-1}). \label{eq---orthoNewton}
\end{equation}
When $N\ge n+m,$ since $(z_\mu^{-\frac{1}{2}}\chi^\la(\mu))_{\la,\mu\vdash n}$ is orthogonal, the last display equals
\[(\sum_{\la\vdash n} \chi^\la(\a) \chi^\la(\pi^{-1}))(\sum_{\mu\vdash m}\chi^\mu(\b) \chi^\mu(\sigma^{-1}))=\delta_{[\a],[\pi]}\delta_{[\b],[\sigma]}z_{[\a]}z_{[\b]}.\]
This concludes \eqref{eq---OrthoHarmoNewton}.  Now writing $p_{[\a,\b]}(U)=\Tr_{T_{n,m}}(\rho_{N} ( \a\times \b)\rho_N ( q_{n,m}) \rho_{n,m}(U)),$ the identity   \eqref{eq---HierarchyNewton} follows from Lemma 
\ref{Lem---HarmoProjectorIdCompo} and is valid for all $\la,\mu\in\Yb$ with $|\la|+|\mu|\le N.$   By induction on  $n$ with $n\le N,$  it implies that $\{p_{[\la,\mu] :|\la|+|\mu|\le n}\}$ is a basis of $\mathcal{F}_n$ and $p_\la\overline p_\mu-p_{[\la,\mu]} $ is the orthogonal projection of $p_{\la}\overline{p}_\mu$ on $\mathcal{F}_{n-1}$ for any $\la,\mu\in \Yb$ with $|\la|+|\mu|\le N.$ 
 
\end{proof}

 \begin{rmk} It follows also from \eqref{eq---orthoNewton} that  \eqref{eq---OrthoHarmoNewton} holds when $|\la|+|\mu|\le N< |\tilde\la|+|\tilde\mu|. $
 \end{rmk}
 
 \begin{proof}[Proof of Prop. \ref{Prop---WickP}]  For any $\la,\mu\in \Yb,$ denote by $Q_{\la,\mu}$ the polynomial with 
 \[: \prod_i Z_{i}^{m_i(\la)}  \overline{Z}_i^{m_i(\mu)}:= Q_{\la,\mu}(Z_i,\overline{Z}_i,i\ge 1).\] 
 When  $|\la|+|\mu|\le N,$ since  $:p_{\la}(U) p_\mu(U^{-1}):$ only depends on moments $\EE[p_{\nu}\overline{p}_{\pi}]$ with $|\nu|+|\pi|\le N,$ by \eqref{eq---GaussianTraces}, 
 \[:p_{\la}(U) p_\mu(U^{-1}):=Q_{\la,\mu}(p_i,\overline{p}_i,i\ge 1).\]
In particular thanks to \eqref{eq---WickNormal}, $:p_{\la}(U) p_\mu(U^{-1}):-p_\la\overline{p}_\mu\in\mathcal{F}_{|\la|+|\mu|-2}.$ Similarly for $|\la|+|\mu|, |\tilde\la|+|\tilde\mu|\le  N ,$ since $|\la|+|\tilde\mu|\le N$ or  $|\tilde \la|+|\mu|\le N$, using 
first \eqref{eq---GaussianTraces} and then \eqref{eq---ProductWick},
\begin{align*}
\EE[:p_{\la} \overline{p}_\mu:  \overline{: p_{\tilde\la} \overline p_{\tilde\mu}:}]  &=\EE[:p_{\la} \overline{p}_\mu:  :\overline p_{\tilde \la} p_{\tilde\mu}: ]\\
&= \EE[ : \prod_i Z_{i}^{m_i(\la)}  \overline{Z}_i^{m_i(\mu)}):  : \prod_i \overline{Z}_{i}^{m_i(\tilde\la)}  Z_i^{m_i(\tilde\mu)}):  ]\\
&=\delta_{\la,\tilde \la}\delta_{\mu,\tilde\mu}\prod_{i\ge 1}(i^{m_i(\la)+m_i(\mu)} m_i(\la)!m_i(\mu)!)=  \delta_{\la,\tilde \la}\delta_{\mu,\tilde\mu} z_{\la}z_\mu.
\end{align*}
 The identity follows then from Lemma \ref{lem---OrthoNewton}.

  \end{proof}
  
  When $|\la|+|\mu|\le N,$ the functions $p_{[\la,\mu]}$ can therefore be written explicitly in terms of standard Newton functions.  Define a partial order on $\mathbb{Y}$ setting 
 \begin{equation}
 \la\le \mu \quad \text{if and only if} \quad  m_i(\la)\le m_i(\mu)\quad \text{for all} \, i\ge 1.
 \end{equation}
  For instance, for $\la,\mu\in\mathbb{Y},$ $\la\wedge \mu$ is the partition with  $m_i(\la\wedge \mu)=m_i(\la) \wedge m_i(\mu)$  for all $i.$ When $\la\le \mu,$ let $\mu\setminus\la\in\mathbb{Y}$ be the partition with 
  \begin{equation}
  m_i(\mu\setminus\la)=m_i(\mu)-m_i(\la)\quad\forall i\ge 1.
  \end{equation}
For any $\la,\mu\in\mathbb{Y}$ with $\la \le \mu,$  
\begin{equation}
c_{\la, \mu }=\prod_{i\ge 1} { m_i(\mu)\choose m_i(\la)}\label{eq---binomial}
\end{equation}
is the number of ways to yield $\mu\setminus\la$ by deleting some parts of $\mu.$ For $\nu,\la,\mu\in \Yb$ with $\nu\le \la\wedge \mu,$ set 
\begin{equation}
C_\nu^{[\la,\mu]}= z_\nu  c_{\nu,\la}  c_{\nu,\mu}=\prod_i \frac{m_i(\la)! m_i(\mu)! i^{m_i(\nu)}}{ m_i(\nu)! m_i(\la \setminus \nu)! m_i(\mu\setminus\nu)!}. 
\end{equation}

\begin{lem}[Gross-Taylor formulas]  For all $\la,\mu\in\Yb$ with $|\la|+|\mu|\le N,$
\begin{equation}
p_{[\la,\mu]}=\prod_i p_{i}^{(m_i(\la),m_i(\mu))}. \label{eq---factoHarmoNewton}
\end{equation}
Moreover 
\begin{equation}
p_{\la,\mu}= \sum_{\nu\in\mathbb{Y}:  \,\nu \le \la\wedge \mu}  C_{\nu}^{[\la,\mu]} p_{[\la\setminus \nu,\mu\setminus\nu]} 
\end{equation}
and
\begin{equation}
p_{[\la,\mu]}=\sum_{\nu\in\mathbb{Y}:  \,\nu \le \la\wedge \mu} (-1)^{\ell(\nu)}  C_{\nu}^{[\la,\mu]}  p_{\la\setminus \nu}\overline p_{\mu\setminus\nu}. \label{eq---GTHarmoNewton} \end{equation}
 \end{lem}
 \begin{proof} The first formula follows from the third. Thanks to Proposition \ref{Prop---WickP}, the second and third  equations follow from  \eqref{eq---WickNormal} and \eqref{eq---ProducttoWickNormal}.
 \end{proof}

 \begin{rmk} The second formula implies that for any $a,b\ge0,i\ge 1$ with $a+b\le \frac N i,$
 \begin{equation}
  p_i^{(a,b)}=\sum_{k=0}^{a\wedge b} (-1)^k k! i^k {a \choose k} {b \choose k}  p^{a-k}_i \overline{p}^{b-k}_i.
 \end{equation}
 \end{rmk}

The above three formulas are due to Gross and Taylor \cite[Appendix A]{GTII}, who also gave the outline of an argument showing $(p_{[\la,\mu]})_{\la,\mu}$ is the orthogonalisation of $(p_{\la}\overline{p}_\mu)_{\la,\mu}$. \footnote{ In contrast to the current proof, Gross and Taylor do not  mention nor use the relation to Wick products. Instead they make an Ansatz for the form \eqref{eq---factoHarmoNewton} of $p_{[\la,\mu]}$ and provide instead a direct computation  to show the  
orthogonality \eqref{eq---OrthoHarmoNewton}.}  

In turn the last proposition  yields as well the expression \eqref{eq---GenFrobenius} for irreducible $U(N)$-characters generalising Fröbenius formula. For any $n\ge 0,$ set 
\begin{equation}
\mathcal{H}_n= \mathrm{Span}\{ p_{[\la,\mu]} : \la,\mu \in \Yb \,\text{with}\, : |\la|+|\mu|=n \}
\end{equation}
and denote by $P_n$ the orthogonal projection of $\mathrm{L}^2(U(N))$ onto $\mathcal{H}_n$.

\begin{coro} \label{coro---FrobeniusGT} For any $\la,\mu\in \Yb $ with $|\la|+|\mu|\le N,$ 
\begin{align}
\chi_{[\la,\mu]_N}&= P_{|\la|+|\mu|}(\chi_\la \overline{\chi}_\mu)\label{eq---ProjectionSchur}\\
&=\frac{1}{n!m!}\sum_{\a\times \b\in S_{n,m}} \chi^\la(\a)\chi^\mu(\b) p_{[\a,\b]}=\frac{1}{n!m!}\sum_{\a\times \b\in S_{n,m}} \chi^\la(\a)\chi^\mu(\b) :p_\a\overline{p}_\b :.\label{eq---GenFrobenius} 
\end{align}
Moreover 
\[d_{\la,\mu}= \frac{1}{n! m!} \chi^{[\la,\mu]}(\Omega_{n,m})\]
with 
\begin{align}
 \Omega_{n,m}&= N^{n+m}\sum_{\a\times \b\in S_{n,m}} N^{-|\a|-|\b|}  \EE\left[\prod_{k\ge 1} (1+\frac{i}{N} Z_k)^{m_k(\a)} (1+\frac {i} N \overline{Z_k})^{m_k(\b)} \right]\a\times\b \label{eq---OmegaFunctionGT}\\
 &=\sum_{\a\times\b\in S_{n,m}} \EE\left[\prod_{k\ge 1} (N+i Z_k)^{m_k(\a)} (N+i \overline{Z_k})^{m_k(\b)}\right]\a\times\b,
\end{align}
where $(Z_k)_{k\ge 1}$ are independent,  with $Z_k\sim\mathcal{N}_\mathbb{C}(0,k)$ and  $m_k(\sigma)$ stands for the number of cycles of length $k$ of  a permutation $\sigma$ for all $k\ge 1.$
\end{coro}
  
\begin{proof}  The first identity of \eqref{eq---GenFrobenius} follows the Koike-Schur-Weyl identity \eqref{eq---HarmoNewtonTOCharacters} and the orthogonality of characters.  Together with the orthogonality relations \eqref{eq---OrthoHarmoNewton}, this identity yields  \eqref{eq---ProjectionSchur}.  Evaluating \eqref{eq---GenFrobenius} at $\mathrm{Id}_N,$ the last relation follows from \eqref{eq---GTHarmoNewton} and that for any $\nu,\la,\mu\in\Yb$ with $\nu\le \la\wedge\mu,$
\[\sum_{\nu\in\mathbb{Y}:  \,\nu \le \la\wedge \mu} (-1)^{\ell(\nu)}  C_{\nu}^{[\la,\mu]}  p_{\la\setminus \nu}(\Id_N)\overline p_{\mu\setminus\nu}(\Id_N)= \EE[\prod_{k} (N+i  Z_k)^{m_i(\la)}(N+i\overline{Z}_k)^{m_k(\mu)}] . \]
\end{proof}

}  

\section*{Acknowledgement}

Many thanks are due to G.Cébron, B. Hall, T. Lemoine,  T. L\'evy  and M. Magee for  discussions related to this work. We  are in particular grateful to Thierry L\'evy for sharing in 2022 the formula \eqref{eq---WilsonLoopFaceCharac}  and  ideas of its proof,  which has been the impetus for starting this project.

\bibliographystyle{alpha}

\newcommand{\etalchar}[1]{$^{#1}$}

\end{document}